\documentclass[11pt]{article}

\usepackage{mystyle}

\usepackage{graphicx,subfigure}

\usepackage[utf8]{inputenc} 
\usepackage[T1]{fontenc}    
\usepackage{lmodern}
\usepackage{hyperref}       
\usepackage{url}            
\hypersetup{colorlinks,linkcolor={red},citecolor={blue},urlcolor={red}}  
\usepackage{booktabs}       
\usepackage{amsfonts}       
\usepackage{nicefrac}       
\usepackage{microtype}      %
\usepackage{color}
\def\l{\left}
\def\r{\right}

\usepackage{graphicx}				 
\usepackage{amssymb}
\usepackage{amsmath}
\usepackage{amsthm}

\usepackage{enumitem}
\usepackage{algorithm}
\def\l{\left}
\def\r{\right}

\newcommand{\mc}[1]{\mathcal{#1}}

\title{Robust One-Bit Recovery via ReLU Generative Networks: Near-Optimal Statistical Rate and Global Landscape Analysis}

\begin{document}
\author{Shuang Qiu\footnote{Equal Contribution}  \thanks{University of Michigan.
Email: \texttt{qiush@umich.edu}.} 
       \qquad
	Xiaohan Wei${}^*$\thanks{Facebook, Inc.
Email: \texttt{ubimeteor@fb.com}.}  
	\qquad
       Zhuoran Yang\thanks{Princeton University.
    Email: \texttt{zy6@princeton.edu}.}
       }


\maketitle

\begin{abstract}
We study the robust one-bit compressed sensing problem whose goal is to design an algorithm that faithfully recovers any sparse target vector $\theta_0\in\mathbb{R}^d$   \textit{uniformly} via $m$ quantized noisy measurements. Specifically, we consider a new framework for this problem where the sparsity is implicitly enforced via mapping a low dimensional representation $x_0 \in \mathbb{R}^k$  through a known $n$-layer ReLU generative network $G:\mathbb{R}^k\rightarrow\mathbb{R}^d$ such that $\theta_0 = G(x_0)$.  Such a framework poses low-dimensional priors on $\theta_0$ without a known sparsity basis. We propose to recover the target $G(x_0)$ solving an unconstrained empirical risk minimization (ERM). Under a weak \textit{sub-exponential measurement assumption}, we establish a joint statistical and computational analysis. In particular, we prove that the ERM estimator in this new framework achieves a statistical rate of $m=\widetilde{\mathcal{O}}(kn \log d /\varepsilon^2)$ recovering any $G(x_0)$ uniformly up to an error $\varepsilon$. When the network is shallow (i.e., $n$ is small), we show this rate matches the information-theoretic lower bound up to logarithm factors of $\varepsilon^{-1}$. From the lens of computation, we prove that under proper conditions on the network weights, our proposed empirical risk, despite non-convexity, has no stationary point outside of small neighborhoods around the true representation $x_0$ and its negative multiple;  furthermore, we show that the global minimizer of the empirical risk stays within the neighborhood around $x_0$ rather than its negative multiple under further assumptions on the network weights.
\end{abstract}

\section{Introduction}
Quantized compressed sensing investigates how to design the sensing procedure, quantizer, and reconstruction algorithm so as to recover a high dimensional vector from a limited number of quantized measurements. The problem of one-bit compressed sensing, which aims at recovering a target vector $\theta_0\in\mathbb R^d$ from single-bit observations $y_i = \sign(\dotp{a_i}{\theta_0}),~i\in\{1,2,\cdots,m\},~m\ll d$ and random sensing vectors $a_i\in\mathbb{R}^d$, is particularly challenging. 
Previous theoretical successes on this problem (e.g. \citet{jacques2013robust,plan2013robust,zhu2015towards}) mainly rely on two key assumptions: (1) The Gaussianity of the sensing vector $a_i$. (2) The sparsity of the vector $\theta_0$ on a given basis. However, the practical significance of these assumptions is rather limited in the sense that it is difficult to generate Gaussian vectors and high dimensional targets in practice are often distributed near a low-dimensional manifold rather than sparse on some given basis. The goal of this work is to make steps towards addressing these limitations. 

\subsection{Sub-Gaussian One-Bit Compressed Sensing}

As investigated in \citet{ai2014one}, sub-Gaussian one-bit compressed sensing can easily fail regardless of the recovery algorithms. More specifically, consider two sparse vectors: $\theta_1 = [1,~0,~0,~\cdots,~0]$, $\theta_2 = [1,~-1/2,~0,~\cdots,~0]$, and i.i.d. Bernoulli sensing vectors $a_i$, where each entry takes $+1$ and $-1$ with equal probabilities. Such sensing vectors are known to perform optimally in the ordinary linear compressed sensing scenario, but cannot distinguish between $\theta_1$ and $\theta_2$ in the current one-bit scenario regardless of algorithms.  Moreover,  \citet{ai2014one,10.1093/imaiai/iay006} further  propose \textit{non-consistent} estimators whose discrepancies are measured in terms of certain distances between the Gaussian distribution and the distribution of the sensing vectors.

A major step towards consistent non-Gaussian one-bit compressed sensing is called \textit{dithering}, which has been considered in several recent works \citep{xu2018quantized, dirksen2018non}. The key idea is that instead of $y_i = \sign(\dotp{a_i}{\theta_0})$, one considers a new procedure by adding artificial random noise $\tau_i$ before quantization: 
$y_i = \sign(\dotp{a_i}{\theta_0} + \tau_i),~i\in\{1,2,\cdots,m\}$.  In addition, \citet{dirksen2018non} proposes a new computationally-efficient convex recovery algorithm and shows that under the new quantization procedure and the sub-Gaussian assumption on $a_i$, one can achieve the best known statistical rate\footnote{In this paper, we use $\tilde{\mathcal{O}}(\cdot)$ to hide the logarithm factors.} 
$m=\tilde{\mathcal{O}}(k\log d/\varepsilon^4)$ estimating \textit{any $k$ sparse $\theta_0\in\mathbb{R}^d$ within radius $R$ uniformly up to error $\varepsilon$ with high probability.}
 \citet{dirksen2018robust-2} further shows that the same algorithm can achieve the rate 
$m=\tilde{\mathcal{O}}(k\log d/\varepsilon^2)$ for vectors $a_i$ sampled from a specific circulant matrix. Without computation tractability, \citet{jacques2013robust,plan2013robust,dirksen2018non} also show that one can achieve the near-optimal rate solving a non-convex constrained program with Gaussian and sub-Gaussian sensing vectors, respectively. 
It is not known though if the optimal rate is achievable via \textit{computationally tractable algorithms}, not to mention more general measurements than Gaussian/sub-Gaussian vectors. 

It is also worth emphasizing that the aforementioned works  \citep{plan2013robust,xu2018quantized,dirksen2018non,dirksen2018robust-2} obtain uniform recovery results which hold with high probability for all $k$ sparse $\theta_0\in\mathbb{R}^d$ within radius $R$. 
The ability of performing uniform recovery potentially allows $\theta_0$ to be adversarially chosen with the knowledge of the algorithm. 
It is a characterization of ``robustness'' not inherited in the non-uniform recovery results  \citep{plan2013robust, zhang2014efficient, goldstein2018structured, thrampoulidis2018generalized}, which provide guarantees \textit{recovering an arbitrary but fixed sparse vector $\theta_0$.} However, with the better result comes the greater technical difficulty unique to one-bit compressed sensing known as the \textit{random hyperplane tessellation problem}. Simply put, uniform recoverability is, in some sense, equivalent to the possibility of constructing a binary embedding of a sparse set into the Euclidean space via random hyperplanes. See \citet{plan2014dimension,dirksen2018non} for details
.

\subsection{Generative Models and Compressed Sensing}

Deep generative models have been applied to a variety of modern machine learning areas. In this work, we focus on using deep generative models to solve inverse problems, which has find extensive empirical successes in image reconstructions such as 
super-resolution \citep{sonderby2016amortised, ledig2017photo}, image impainting \citep{yeh2017semantic} and medical imaging \citep{hammernik2018learning, yang2018dagan}. In particular, these generative model based methods have been shown to produce comparable results to the classical sparsity based methods with much fewer (sometimes 5-10x fewer) measurements, which will greatly benefit application areas such as magnetic resonance imaging (MRI) and computed tomography (CT), where the measurements are usually quite expensive to obtain. In contrast to widely recognized empirical results, theoretical understanding of generative models remains limited.

In a recent work, \citet{bora2017compressed} considers a linear model 
$
\mf y = \mathbf{A}G( x_0)+\eta,
$
where $\mathbf{A}$ is a Gaussian measurement matrix, $\eta$ is a bounded noise term and $G(\cdot)$ is an $L$-Lipschitz generative model. By showing that the Gaussian measurement matrix satisfies a restricted eigenvalue condition (REC) over the range of $G(\cdot)$, the authors prove the $L_2$ empirical risk minimizer
\begin{equation}\label{eq:l2-risk}
\hat{x}\in\arg\min_{x\in\mathbb{R}^k} 
\|\mathbf{A}G( x)-\mathbf y\|_2^2
\end{equation}
satisfies an estimation error bound $\|\eta\|_2+\varepsilon$ when the number of samples is of order $\mathcal{O}(k\log (L/\varepsilon)/\varepsilon^2)$. They further show that the $\log(1/\varepsilon)$ term in the error bound can be removed when $G(\cdot)$ is a multilayer ReLU network. In addition,  \citet{hand2018global,huang2018provably} consider the same linear model with the aforementioned $L_2$ empirical risk minimizer and an $n$-layer ReLU network $G(\cdot)$. They show when the noise in the linear model is small enough, the measurement matrix satisfies range restricted concentration, which is stronger than REC, 
$m\geq\mathcal{O}(kn\log d ~\text{poly}(\varepsilon^{-1}))$\footnote{Here $\text{poly}(\varepsilon^{-1})$ stands for polynomial dependency on $\varepsilon^{-1}$.}, and suitable conditions on the weights of the ReLU function hold, the $L_2$ empirical risk enjoys a favorable landscape. Specifically, there is no spurious local stationary point outside of small neighborhoods of radius 
$\mathcal{O}(\varepsilon^{1/4})$ around the true representation $x_0$ and its negative multiple, and with further assumptions, the point $\hat{x}$ is guaranteed to be located around $x_0$ instead of its negative multiple. Moreover, 
\citet{liu2019information,kamath2019lower} study sample complexity lower bounds for the generative compressed sensing model as \eqref{eq:l2-risk}.

More recently, generative models have been applied to scenarios beyond linear models with theoretical guarantees. \citet{wei2019statistical} considers a non-linear recovery using a generative model, where the link function is assumed to be differentiable and the recovery guarantee is non-uniform. \citet{hand2019global} studies the landscape of $L_2$ empirical risk for blind demodulation problem with an $n$-layer ReLU generative prior. Using the same prior,
\citet{hand2018phase} analyzes the landscape of the amplitude flow risk objective for phase retrieval. Furthermore, \citet{aubin2019spiked} investigates the spiked matrix model using generative priors with linear activations. Besides these studies, there is another line of work investigating the problem of compressed sensing via generative models by the approximate message passing framework, e.g. \citet{manoel2017multi,pandit2020inference}.

\subsection{Summary of the Main Results}
We introduce a new framework for robust \emph{dithered} one-bit compressed sensing where the structure of target vector $\theta_0$ is represented via an $n$-layer ReLU network $G:\mathbb{R}^k\rightarrow\mathbb{R}^d$, i.e., $\theta_0 = G(x_0)$ for some 
$x_0\in\mathbb{R}^k$ and $k\ll d$. Building upon this framework, we propose a new recovery model which is related to solving an unconstrained ERM, with $\widehat x_m$ being the solution to the proposed ERM. We show that this model enjoys the following favorable properties: 
\begin{itemize}[leftmargin=*]

\item Statistically, when taking measurements $a_i$ to be \textit{sub-exponential} random vectors, with high probability and uniformly for any $G(x_0)\in G(\mathbb{R}^k)\cap \mathbb{B}_2^d(R)$, where $\mathbb{B}_2^d(R)$ is the ball of radius $R>0$ centered at the origin, $G(\widehat x_m)$ recovers the true vector $G(x_0)$ up to error $\varepsilon$ when the number of samples 
$m = \widetilde{\mathcal{O}}(kn\log d/\varepsilon^2)$. In particular, our result does not require REC-type assumptions adopted in previous analysis of generative signal recovery works and at the same time weakens the known sub-Gaussian assumption adopted in previous sparse one-bit compressed sensing works. 
Moreover, we further establish an information-theoretic lower bound for the sample complexity. 
When the number of layers $n$ is small, we show that the proved statistical rate matches the information-theoretic lower bound up to logarithm factors of $\varepsilon^{-1}$.

\item Computationally, building upon the previous methods guaranteeing uniform recovery, we show that solving the ERM and approximate the true representation $x_0\in\mathbb R^k$ can be tractable under further assumptions on ReLU networks. 
More specifically, we prove with high probability, there always exists a descent direction outside of two small neighborhoods around $x_0$ and $-\rho_n x_0$ with radius $\mathcal{O}(\varepsilon_\WDC^{1/4})$ respectively, where $0<\rho_n \leq 1$ is a factor depending on $n$. This holds uniformly for any $x_0\in \mathbb{B}_2^k(R')$ with $R'= (0.5+\varepsilon_\WDC)^{-n/2}R$, when the ReLU network satisfies a Weight Distribution Condition with a parameter $\varepsilon_\WDC>0$ and $m = \widetilde{\mathcal{O}}(kn \log d /\varepsilon_\WDC^2)$. Furthermore, when $\varepsilon_\WDC$ is sufficiently small, one guarantees that the solution $\widehat{x}_m$ stays within the neighborhood around $x_0$ rather than $-\rho_n x_0$. Our result is achieved under quantization errors and without assuming the REC-type conditions, thereby improving upon previously known computational guarantees for ReLU generative signal recovery in linear models with small noise.
 \end{itemize}
From a technical perspective, our proof makes use of the special piecewise linearity property of ReLU network. The merits of such a property in the current scenario are two-fold: (1) It allows us to replace the generic chaining type bounds commonly adopted in previous works, e.g. \citet{dirksen2018non}, by novel arguments that are ``sub-Gaussian free''. 
(2) From a hyperplane tessellation point of view, we show that for a given accuracy level, a binary embedding of $G(\mathbb{R}^k)\cap \mathbb{B}_2^d(R)$ into Euclidean space is ``easier'' in that it requires less random hyperplanes than that of a bounded $k$ sparse set, e.g. \citet{plan2014dimension,dirksen2018non}.

\vspace{3.5pt}
\noindent\textbf{Notation.} Throughout this paper, let $\mathcal S^{d-1}$ and $\mathbb B_2^d(r)$ be the unit Euclidean sphere and the Euclidean ball of radius $r$ centered at the origin in $\mathbb{R}^d$, respectively. We also use $\mathcal B(x,r)$ to denote the Euclidean ball of radius $r$ centered at $x \in \mathbb{R}^k$.
For a random variable $X\in\mathbb{R}$, the $L_p$-norm ($p\geq1$) is denoted as $\|X\|_{L_p}=\expect{|X|^p}^{1/p}$. The Olicz 
$\psi_1$-norm is denoted
$\|X\|_{\psi_1}:= \sup_{p\geq1}p^{-1}\|X\|_{L_p}$. We say a random variable is sub-exponential if its $\psi_1$-norm is bounded. A random vector $ x\in\mathbb{R}^d$ is sub-exponential if there exists a a constant $C>0$ such that
$\sup_{t\in\mathcal S^{d-1}}\|\dotp{ x}{t}\|_{\psi_1}\leq C$. We use $\| x\|_{\psi_1}$ to denote the minimal $C$ such that this bound holds. Furthermore, $C,C',C_0, C_1,\ldots$ and $c,c',c_0,c_1,\ldots$ denote absolute constants, and their actual values can be different per appearance. We let $[n]$ denote the set $\{1,2,\ldots,n\}$. We denote $\mf I_p$ and $\mf{0}_{p\times d}$ as a $p \times p$ identity matrix and a $p\times d$ all-zero matrix respectively.

\section{Model}
In this paper, we focus on one-bit recovery model in which one observes quantized measurements of the following form
\begin{equation}\label{eq:def-y}
y = \sign(\dotp{a}{G(x_0)} + \xi + \tau),
\end{equation}
where $a\in\mathbb{R}^d$ is a random measurement vector, $\xi\in\mathbb{R}$ is a random pre-quantization noise with an unknown distribution, $\tau$ is a random quantization threshold (i.e., \emph{dithering noise}) which one can choose, and $x_0\in\mathbb{R}^k$ is the unknown representation to be recovered.
We are interested the high-dimensional scenario where the dimension of the representation space $k$ is potentially much less than the ambient dimension $d$. The function $G:\mathbb{R}^k\rightarrow\mathbb{R}^d$ is a fixed ReLU neural network of the form:
\begin{equation}\label{eq:relu}
G(x) = \sigma\circ(W_n \sigma\circ(W_{n-1}\cdots\sigma\circ(W_1x))),
\end{equation}
where $\sigma\circ(\cdot)$ denotes the entry-wise application of the ReLU activation function $\sigma(\cdot) = \max\{\cdot, 0\}$ on a vector. We consider a scenario where the number of layers $n$ is smaller than $d$ and the weight matrix of the $i$-th layer is $W_{i}\in\mathbb{R}^{d_i\times d_{i-1}}$ with $d_n=d$ and $d_i\leq d,~\forall i\in[n]$.
 Throughout the paper, we assume that $G(x_0)$ is bounded, i.e. there exists an $R\geq1$ such that $\|G(x_0)\|_2\leq R$, and
we take $\tau\sim\text{Unif}[-\lambda,+\lambda]$, i.e. a uniform distribution bounded by a chosen parameter $\lambda>0$. 
Let 
$\{(a_i,y_i)\}_{i=1}^m$ be i.i.d. copies of $(a,y)$. Our goal is to compute an estimator $G(\hat x_m)$ of $G(x_0)$ such that 
$\|G(\hat x_m) - G(x_0)\|_2$ is small.

We propose to solve the following ERM for estimator $\hat x_m$:
\begin{align}
\begin{aligned}\label{eq:erm}
\hat x_m:=\arg\min_{x\in\mathbb{R}^k} \Big\{ L(x)
:= \|G(x)\|_2^2 - \frac{2\lambda}{m}\sum_{i=1}^my_i\dotp{a_i}{G(x)} \Big\},
\end{aligned}
\end{align}
where $y_i = \sign (\dotp{a_i}{G(x_0)} + \xi_i+\tau_i)$. It is worth mentioning that, in general, there is no guarantee that the minimizer of $L(x)$ is unique. Nevertheless, in Sections \S \ref{sec:statistical} and \S \ref{sec:computational}, we will show that any solution $\hat x_m$ to this problem must satisfy the desired statistical guarantee, and stay inside small neighborhoods around the true signal $x_0$ and its negative multiple with high probability. 

\section{Main Results}
In this section, we establish our main theorems regarding statistical recovery guarantee of $G(x_0)$ and the associated information-theoretic lower bound in Sections \S \ref{sec:statistical} and
\S \ref{sec:lowerbound}. The global landscape analysis of the empirical risk $L(x)$ is presented in Section \S \ref{sec:computational}.  

\subsection{Statistical Guarantee}\label{sec:statistical}
We start by presenting the statistical guarantee of using ReLU network for one-bit compressed sensing. 
Our statistical guarantee relies on the following assumption on the measurement vector and noise.
\begin{assumption}\label{as:moments}
The measurement vector $a\in\mathbb R^d$ is mean 0, isotropic and sub-exponential. The noise $\xi$ is also a sub-exponential random variable.
\end{assumption}
Under this assumption, we have the following main statistical performance theorem.
\begin{theorem}\label{thm:main-1}
Suppose Assumption \ref{as:moments} holds and consider any $\varepsilon\in(0,1)$. Set $C_{a,\xi,R}=\max\{c_1(R\|a\|_{\psi_1} \allowbreak +\|\xi\|_{\psi_1}),1\}$, $\lambda\geq4C_{a,\xi,R}\cdot\log(64C_{a,\xi,R}\cdot\varepsilon^{-1})$, and 
\begin{align}
\begin{aligned}\label{eq:sample-complexity}
m &\geq c_2\lambda^2\log^2(\lambda m)\big[kn\log(ed)   + k\log(2R) + k\log m+ u \big]/\varepsilon^2.
\end{aligned}
\end{align}
Then, with probability at least $1-c_3\exp(-u),~\forall u\geq0$, any solution $\hat x_m$ to \eqref{eq:erm} satisfies 
$$\|G(\hat x_m)-G(x_0)\|_2\leq \varepsilon$$ 
for all $x_0$ such that $\|G(x_0)\|_2\leq R$, where $c_1,~c_2, ~c_3\geq1$ are absolute constants. 
\end{theorem}
\begin{remark}[Sample Complexity]\label{remark:sample_comp}
One can verify that the sample complexity enforced by \eqref{eq:sample-complexity} holds when $
m\geq  C \log^4(C_{a,\xi,R}\cdot\varepsilon^{-1})(kn\log(ed) + k\log(2R) + k\log (\varepsilon^{-1})+ u)/\varepsilon^2
$, where $C$ is a large enough absolute constant.
This gives the $\mathcal{O}(kn\log^4(\varepsilon^{-1})(\log d + \log(\varepsilon^{-1}))/\varepsilon^2)$, or equivalently $\widetilde{\mathcal{O}}(kn\log d /\varepsilon^2)$, sample complexity. In particular, when the number of layers $n$ is small, our result meets the optimal rate of sparse recovery (up to a logarithm factor) and demonstrate the effectiveness of recovery via generative models theoretically. 
The dependence on the number of layers $n$ results from the fact that our bound counts the number of linear pieces split by the ReLU generative network (see Lemma \ref{lem:ball-supremum} for details). Measuring certain complexities of a fixed neural network via counting linear pieces arises in several recent works (e.g. \citet{lei2018geometric}), and the question whether or not the linear dependence on $n$ is tight warrants further studies. 
\end{remark}

Note that our result is a \textit{uniform recovery result} in the sense that the bound $\|G(\hat x_m)-G(x_0)\|_2\leq \varepsilon$ holds with high probability uniformly for any target $x_0\in\mathbb{R}^k$ such that $\|G(x_0)\|_2\leq R$. This should be distinguished from known bounds \citep{plan2013robust, zhang2014efficient, goldstein2018structured, thrampoulidis2018generalized} on sparse one-bit sensing which hold only for a \emph{fixed} sparse vector. The boundedness of $G(x_0)$ is only assumed for theoretical purpose, which could be removed for practice.

Moreover, apart from the ReLU network, the proof of this theorem can be extended to other networks possessing the piecewise linearity property. Whether the analysis can be applied to networks with  a wider class of nonlinear activation functions remains to be further studied. The proof sketch is presented in Section \S \ref{sec:proof_lower_bound} with more proof details in Supplement \S \ref{sec:detailed_proof_lower_bound}.

\subsection{Information-Theoretic Lower Bound} \label{sec:lowerbound}

In this section, we show that when the network is shallow, i.e., $n$ is small, for any $k$ and $d$, there exists a ReLU network of the form \eqref{eq:relu} such that the above rate in Theorem \ref{thm:main-1} is optimal up to some logarithm factors. More specifically, we have the following theorem.

\begin{theorem}\label{thm:lower-bound}
For any positive $k$ and $d$ large enough such that $k\ll d $ with $k\leq d/4$, there exists a generative network $G$ of the form \eqref{eq:relu} with a $k+1$ dimensional input,  depth $n=3$ such that for 
the linear model before quantization: $\check{y} = \dotp{a}{\theta_0} + \xi$, where $\theta_0\in G(\mathbb R^{k+1})\cap \mathbb B_2^d(1)$,  $\xi\sim\mathcal N(0,1)$, $a\sim\mathcal N(0, \mathbf I_d)$ and $m\geq c_1k\log(d/k)$, we have
\begin{align}\label{eq:lower_bound}
\inf_{\widehat{\theta}} \sup_{\theta_0 \in G(\mathbb R^{k+1})\cap \mathbb B_2^d(1)} \EE \|\widehat \theta - \theta_0\|_2 \geq c_2\sqrt{\frac{k\log (d/k)}{m}}, 
\end{align} 
where $c_1,c_2>0$ are absolute constants and the infimum is taken over all estimators $\widehat\theta$ generated by all possible algorithms depending only on $m$ i.i.d. copies 
of $(a,\check{y})$.
\end{theorem}

This theorem gives a lower bound of sample complexity over the set of all algorithms $\tilde{\mathcal{A}}$ recovering $\theta_0$ from the noisy linear model $\check{y} = \dotp{a}{\theta_0} + \xi$ by observing  ${(a_i,\check{y}_i)}_{i=1}^m$.
It gets connected to the one-bit dithered observations as follows: We consider a subset $\mathcal{A}\subseteq\tilde{\mathcal{A}}$ of algorithms, which adds dithering noise $\tau_i$ and then uses quantized observations $(a_i,y_i)_{i=1}^m$ to recover $\theta_0$, where $y_i = \sign(\check{y}_i+\tau_i)$. The corresponding estimators generated by any algorithm in $\mathcal{A}$ will also satisfy \eqref{eq:lower_bound}. Thefore, we have the following corollary of Theorem \ref{thm:lower-bound}, which gives the lower bound of sample complexity for one-bit recovery via a ReLU network.

\begin{corollary}
For any positive $k$ and $d$ large enough such that $k\ll d $ with $k\leq d/4$, there exists a generative network $G$ of the form \eqref{eq:relu} with a $k+1$ dimensional input,  depth $n=3$ such that for 
the quantizd linear model: $y = \sign(\dotp{a}{\theta_0} + \xi + \tau)$, where $\theta_0\in G(\mathbb R^{k+1})\cap \mathbb B_2^d(1)$,  $\xi\sim\mathcal N(0,1)$, $a\sim\mathcal N(0, \mathbf I_{d})$ and $m\geq c_1k\log(d/k)$, we have
\[
\inf_{\widehat{\theta}} \sup_{\theta_0\in G(\mathbb R^{k+1})\cap \mathbb B_2^d(1)} \EE \|\widehat \theta - \theta_0\|_2 \geq c_2\sqrt{\frac{k\log (d/k)}{m}},
\] 
where $c_1,c_2>0$ are absolute constants and the infimum is taken over all estimators $\widehat\theta$ generated by all possible algorithms depending on $m$ i.i.d. copies $(a_i,y_i)_{i=1}^m$ of $(a, y)$.
\end{corollary}
\begin{remark} [Lower Bound]
This corollary indicates that the sample complexity recovering $\theta_0$ within error $\varepsilon$ is at least $\Omega \big( k \log(d/k)/\varepsilon^2 \big)$. Thus, when the ReLU network is shallow (the depth $n$ is small) and $k \ll d$, the sample complexity we have obtained in Theorem \ref{thm:main-1} and Remark \ref{remark:sample_comp} is near-optimal up to logarithm factors of $ \varepsilon^{-1}$ and $ k$.
\end{remark}

The proof is inspired by an observation in \citet{liu2019information} that for a specifically chosen ReLU network (with offsets), the linear recovery problem considered here is equivalent to a group sparse recovery problem. The main differences here, though, are two-fold: first, we need to tackle the scenario where the range of the generative network is restricted to a unit ball; second, our ReLU network \eqref{eq:relu} has \emph{no offset}. The proof is postponed in Section \S \ref{sec:proof_lower_bound} with more proof details in Supplement \S \ref{sec:detailed_proof_lower_bound}.

\subsection{Global Landscape Analysis}\label{sec:computational}

In this section, we present the theoretical properties of the global landscape of the proposed empirical risk $L(x)$ in \eqref{eq:erm}. We start by introducing some notations used in the rest of this paper. For any fixed $x$, we define $W_{+, x}:=\mathrm{diag}(Wx>0) W$, where $\mathrm{diag}(Wx>0)$ is a diagonal matrix whose $i$-th diagonal entry is $1$ if the product of the $i$-th row of $W$ and $x$ is positive, and $0$ otherwise. Thus, $W_{+, x}$ retains the rows of $W$ which has a positive product with $x$, and sets other rows to be all zeros. We further define $W_{i,+,x}:= \mathrm{diag}(W_iW_{i-1,+,x} \cdots W_{1,+,x}x>0)W_i$ recursively, where only active rows of $W_i$ are kept, such that the ReLU network $G(x)$ defined in \eqref{eq:relu} can be equivalently rewritten as $G(x) = (\Pi_{i=1}^n W_{i,+,x}) x := W_{n,+,x} W_{n-1,+,x}\cdots W_{1,+,x} x $. Next, we introduce the Weight Distribution Condition, which is widely used in recent works to analyze the landscape of different empirical risks \citep{hand2018global,hand2018phase,huang2018provably}.

\begin{definition}[Weight Distribution Condition (WDC)]  \label{def:wdc}
A matrix $W$ satisfies the Weight Distribution Condition with $\varepsilon_\WDC > 0$ if for any nonzero vectors $x,z\in \R^{p}$, 
\begin{align*}
\l\| W_{+,x}^\top W_{+,z} - Q_{x,z} \r\|_2 \leq \varepsilon_\WDC,  
\end{align*}
where $Q_{x,z} := \frac{\pi-\angle (x,z)}{2 \pi} \mf I_{p} + \frac{\sin \angle (x,z) }{2\pi } M_{\hat{x} \leftrightarrow \hat{z}}$ with $M_{\hat{x} \leftrightarrow \hat{z}}$ being the matrix transforming $\hat{x}$ to $\hat{z}$, $\hat{z}$ to $\hat{x}$, and $\vartheta$ to $0$ for any $\vartheta \in \mathrm{span} (\{ x,z  \})^{\bot}$. We denote $\hat{x} = \frac{x}{\|x\|_2}$ and $\hat{z} = \frac{z}{\|z\|_2}$ as normalized $x$ and $z$.

\end{definition}

Particularly, the matrix $M_{\hat{x} \leftrightarrow \hat{z}}$ in the definition of WDC is defined as 
\begin{align*}
M_{\hat{x} \leftrightarrow \hat{z}}:=U^\top \begin{bmatrix}
\cos \angle (x,z) & \sin \angle (x,z) & 0\\ 
\sin \angle (x,z) & -\cos \angle (x,z) & 0\\ 
0 & 0 & \mf{0}_{(p-2)\times (p-2)}
\end{bmatrix} U,
\end{align*}
where the matrix $U$ denotes a rotation matrix such that $U \hat{x} = e_1$ and $U \hat{z} = \cos\angle (x,z) \cdot e_1 + \sin \angle (x,z) \cdot  e_2$ with $e_1 = [1,0,\cdots, 0]^\top$ and $e_2 = [0,1,0,\cdots, 0]^\top$. Moreover, if $\angle (x,z) = 0$ or $\angle (x,z) = \pi$,  then we have $M_{\hat{x} \leftrightarrow \hat{z}} = \hat{x} \hat{x}^\top$ or $M_{\hat{x} \leftrightarrow \hat{z}} = -\hat{x} \hat{x}^\top$ respectively.

Intuitively, the WDC characterizes the invertibility of the ReLU network in the sense that the output of each layer of the ReLU network nearly preserves the angle of any two input vectors. 
As is shown in \citet{hand2018global}, for any arbitrarily small $\varepsilon_\WDC>0$, if the network is sufficiently expansive at each layer, namely $d_i \geq c d_{i-1} \log d_{i-1}$ for all $i\in [n]$ with $d_i$ being polynomial on $\varepsilon_\WDC^{-1}$, and entries of $W_{i}$ are i.i.d. $\mathcal{N}(0, 1/ d_i)$, then $W_i \in \R^{d_i \times d_{i-1}}$ for all $i \in [n]$ satisfies WDC with constant $\varepsilon_\WDC$ with high probability. In particular, it does not require $W_i$ and $W_j$ to be independent  for $i \neq j$. The question whether WDC is necessary for analyzing the
computational aspect of the generative network remains open and warrants further studies.

Next, we present the Theorems \ref{thm:determine_converge} and \ref{thm:optimum_compare}. We denote the directional derivative along the direction of the non-zero vector $z$ as $D_{z} L(x) = \lim_{t\rightarrow 0^+}\frac{L(x+t\hat{z}) - L(x)}{t}$ with $\hat{z} = \frac{z}{\|z\|_2}$. Specifically, $D_{z} L(x)$ equals $\langle  \nabla L(x), \hat{z} ~\rangle$ if $L(x)$ is differentiable at $x$ and otherwise equals $\lim_{N\rightarrow +\infty} \langle \nabla L(x_N), \hat{z}~ \rangle$. Here $\{x_N\}_{N \geq 0}$ is a sequence such that $x_N \rightarrow x$ and $L(x)$ is differentiable at any $x_N$. The existence of such a sequence is guaranteed by the piecewise linearity of $G(x)$. Particularly, for any $x$ such that $L(x)$ is differentiable, the gradient of $L(x)$ is computed as $\nabla L(x) = 2 (\Pi_{j=1}^n W_{j, +, x})^\top (\Pi_{j=1}^n W_{j, +, x} ) x - \frac{2\lambda}{m} \sum_{i=1}^m y_i (\Pi_{j=1}^n W_{j, +, x} )^\top a_i$. 

\begin{theorem} \label{thm:determine_converge} 
Suppose that $G(\cdot)$ is a ReLU network with weights $W_i$  satisfying WDC with $\varepsilon_\WDC$ for all $i\in [n]$ where $n > 1$. Let $v_x = \lim_{x_N \rightarrow x} \nabla L(x_N)$ where $\{x_N\}$ is the sequence such that $\nabla L(x_N)$ exists for all $x_N$ (and $v_x = \nabla L(x)$ if $L(x)$ is differentiable at $x$). If $\varepsilon_\WDC$ sastisfies $c_1 n^8 \varepsilon_\WDC^{1/4} \leq 1$, by setting $\lambda\geq4C_{a,\xi,R}\cdot\log(64C_{a,\xi,R}\cdot\varepsilon_\WDC^{-1})$ and $m \geq c_2 \lambda^2   \log^2(\lambda m) (kn\log(ed) + k\log(2R) + k\log m+ u)  /\varepsilon_\WDC^2,$ then with probability $1-c_3\exp(-u)$, for any nonzero $x_0$ satisfying $\|x_0\|_2 \leq R (1/2 + \varepsilon_\WDC)^{-n/2}$, the directional derivatives satisfy
\begin{itemize} [noitemsep,leftmargin=*]
\item[\textbf{1.}] If $\|x_0\|_2 >  \check{\delta}$, then
\begin{align*}
&
D_{-v_x} L(x) < 0, \ \  \forall x \notin \mathcal{B} (x_0, \delta_1 ) \cup \mathcal{B} (-\rho_n x_0, \delta_2) \cup \{0\}, \\
&
D_{w} L(0) < 0, \quad\ \  \forall w \neq 0.
\end{align*}
\item[\textbf{2.}] If $\|x_0\|_2 \leq \check{\delta}$, then 
\begin{align*}
D_{-v_x} L(x) < 0, \ 
\forall x \notin \mathcal{B} (x_0, \delta_1) \cup \mathcal{B} (-\rho_n x_0, \delta_2) \cup \mathcal{B} (0, \check{\delta}),
\end{align*}
\end{itemize}
where we have $\check{\delta} = 2^{n/2} \varepsilon^{1/2}_\WDC$, $\delta_1 = c_4 n^3 \varepsilon^{1/4}_\WDC\|x_0\|_2$, $\delta_2=c_5 n^{14} \varepsilon^{1/4}_\WDC \|x_0\|_2$, and $0 < \rho_n \leq 1$ with $\rho_n \rightarrow 1$ as $n \rightarrow \infty$.
\end{theorem}
\begin{remark}[Interpretation of Theorem \ref{thm:determine_converge}]
Note that in the above theorem, Case 1 indicates that the when the magnitude of the true representation $\|x_0\|^2_2$ is larger than 
$\check{\delta}^2 = \mc O(\varepsilon_\WDC)$ (signal $x_0$ is strong), the global minimum lies in small neighborhoods around $x_0$ and its scalar multiple $-\rho_n x_0$, while for any point outside the neighborhoods of $x_0$ and $-\rho_n x_0$, one can always find a direction with a negative directional derivative. Note that $x = 0$ is a local maximum due to $D_{w} L(0) < 0$ along any non-zero directions $w$. One the other hand, Case 2 implies that when $\|x_0\|^2_2$ is smaller than $\check{\delta}^2$, the global minimum lies in the neighborhood around $0$ (and thus around $x_0$). We will see in Theorem \ref{thm:optimum_compare} that one can further pin down the global minimum around the true $x_0$ for Case 1.
\end{remark}

The next theorem shows that in Case 1 of Theorem \ref{thm:determine_converge}, under certain conditions, the true global minimum lies around the true representation $x_0$ instead of its negative multiple. 

\begin{theorem}\label{thm:optimum_compare}
Suppose that $G(\cdot)$ is a ReLU network with wights $W_i$ satisfying WDC with error $\varepsilon_\WDC$ for all $i \in [n]$ where $n > 1$. Assume that $c_1 n^3 \varepsilon^{1/4}_\WDC \leq 1$ , and $x_0$ is any nonzero vector satisfying $\|x_0\|_2 \leq R (1/2 + \varepsilon_\WDC)^{-n/2}$. Then, setting $\lambda\geq4C_{a,\xi,R}\cdot\log(64C_{a,\xi,R}\cdot\varepsilon_\WDC^{-1})$ and 
$m \geq c_2  \lambda^2   \log^2(\lambda m) (kn\log(ed) + k\log(2R) + k\log m+ u)  /\varepsilon_\WDC^2$, with probability $1-2c_3 \exp(-u)$, for any $x_0$ such that $\|x_0\|_2\geq \check{\delta} = 2^{n/2} \varepsilon^{1/2}_\WDC$, 
the risk $L(\cdot)$ satisfies
\begin{align*}
L(x) < L(z), \ \  \forall x \in \mathcal{B}(\varphi x_0, \delta_3 )~ \text{ and }~ \forall z \in \mathcal{B}(-\zeta x_0,  \delta_3 ),
\end{align*}
where $\varphi, \zeta$ are any scalars in $[\rho_n, 1]$ and $\delta_3 = c_4 n^{-5} \|x_0\|_2$. Particularly, we have that the radius $\delta_3 < \rho_n \|x_0\|_2, \forall n > 1$, such that $0 \notin \mathcal{B}(\varphi x_0, \delta_3 )$ and $0 \notin \mathcal{B}(-\zeta x_0,  \delta_3 )$.
\end{theorem}
\begin{remark}[Interpretation of Theorem \ref{thm:optimum_compare}]
The significance of Theorem \ref{thm:optimum_compare} is two-fold: first, it shows that the value of the empirical risk $L(x)$ is always smaller around $x_0$ compared to its negative multiple $-\rho_n x_0$; second, when the network is sufficiently expansive such that $\varepsilon_\WDC$ is small, i.e. $c n^{19} \varepsilon^{1/4}_\WDC \leq 1$, along with Case 1 in Theorem \ref{thm:determine_converge}, we have $\mathcal{B} (x_0, \delta_1 ) \subseteq \mathcal{B}(\varphi x_0, \delta_3 )$ and $ \mathcal{B} (-\rho_n x_0, \delta_2) \subseteq \mathcal{B}(-\zeta x_0,  \delta_3 )$ for some $\varphi$ and $\zeta$, so that one can guarantee that the global minimum of $L(x)$ stays around $x_0$. Since we do not focus on optimizing the order of $n$ in our results, further improvement of such a dependency will be one of our future works. 
\end{remark}

\begin{figure}[!t]
     \centering
     \begin{subfigure}
         \centering
         \includegraphics[width=0.48\textwidth]{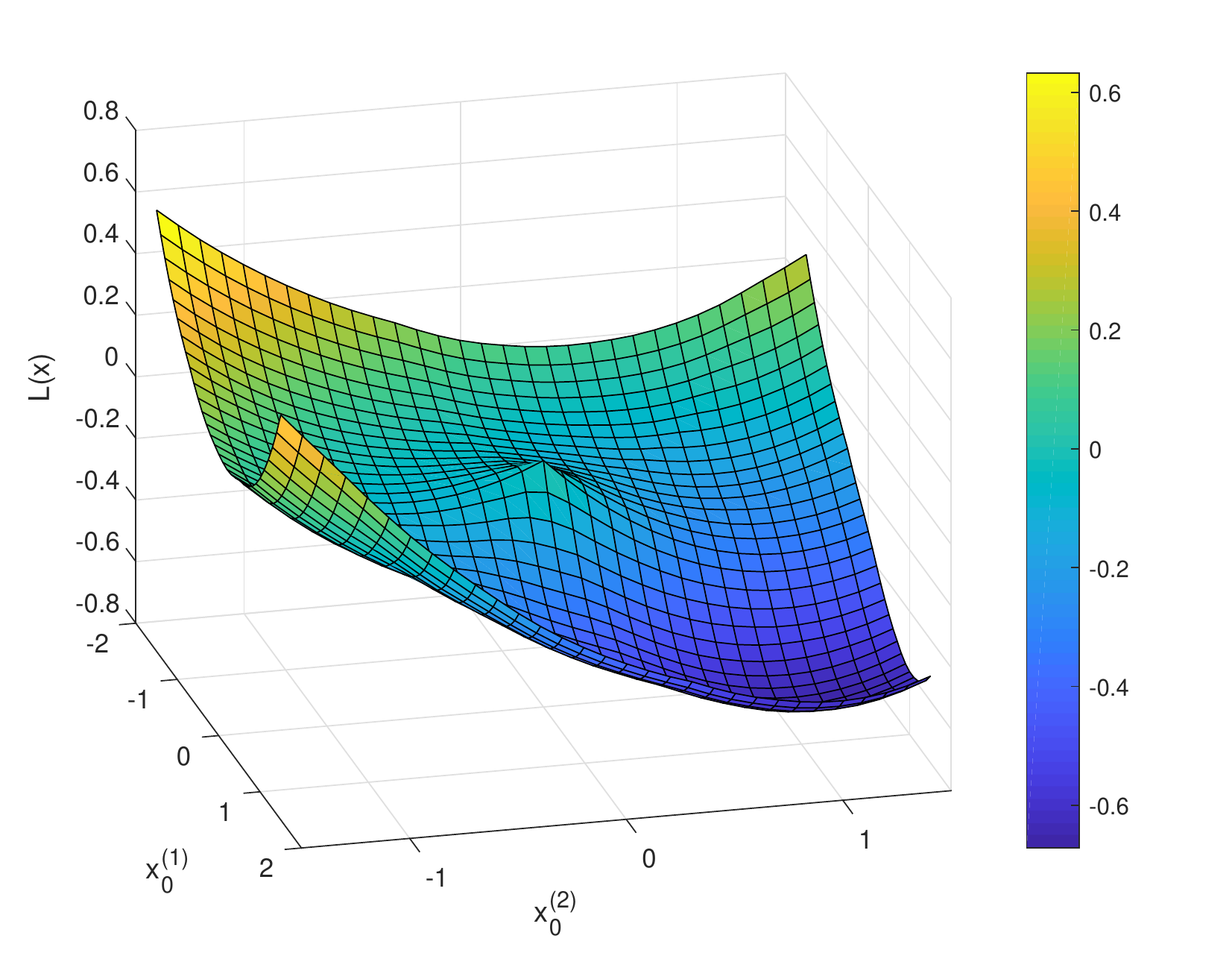}
     \end{subfigure}
     \begin{subfigure}
         \centering
         \includegraphics[width=0.46\textwidth]{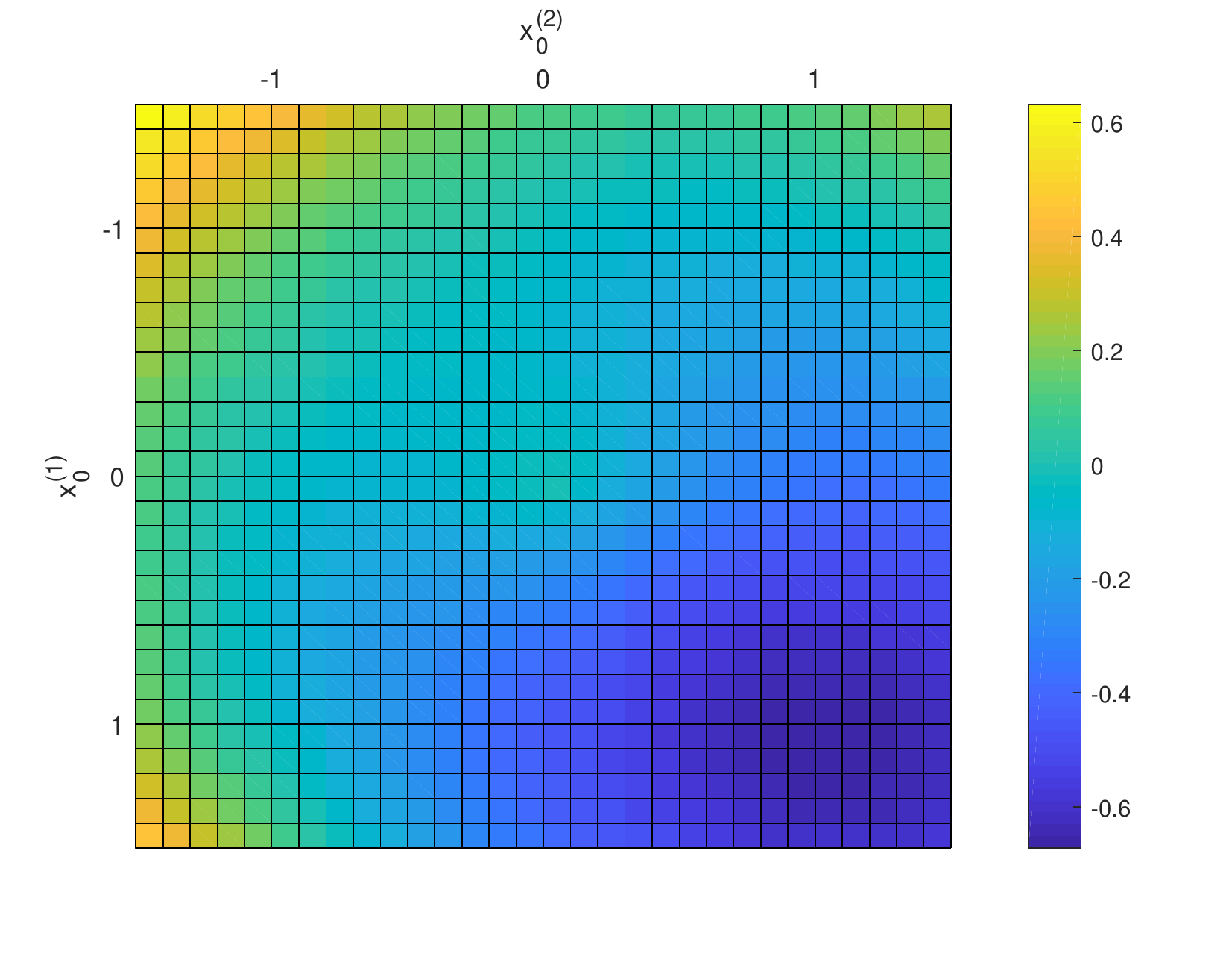}
     \end{subfigure}
     \setlength{\belowcaptionskip}{-0.2cm}
	 \caption{Illustration of landscape for $L(x)$. We build a two-layer ReLU network $G(\cdot)$ with input $x_0$ where $x_0 = [1,1]^\top$, Gaussian weights $W_1 \in \RR^{64 \times 2}$ and $W_2 \in \RR^{1024 \times 64}$ such that $k = 2$ and $d = 1024$. The samples $\{(a_i, y_i)\}_{i=1}^m$ are generated via standard Guassian vector $a_i$ and $y_i = \sign( \langle a_i, G(x_0) \rangle + \xi_i + \tau_i)$ with noise $\xi_i \sim \mc N(0,0.01)$, dithering $\tau_i \sim \text{Unif}(-10,10)$, and a large sample number $m \rightarrow +\infty$.}
	 \label{fig:landscape}
\end{figure}

For better understanding of the landscape analyzed in Theorem \ref{thm:determine_converge} and Theorem \ref{thm:optimum_compare}, we illustrate the landscape of $L(x)$ via simulation in Figure \ref{fig:landscape}. The simulation is based on a large sample number $m \rightarrow + \infty$, which intends to show the landscape of the expectation of the risk $L(x)$. We are more interested in Case 1 of Theorem \ref{thm:determine_converge}, where $x_0$ can be potentially recovered. By letting $x_0 = [1,1]^\top$ which is sufficiently far away from the origin, Figure \ref{fig:landscape} shows that there are no stationary points outside the neighbors of $x_0$ and its negative multiple and the directional derivatives along any directions at the origin are negative, which matches the Case 1 of Theorem \ref{thm:determine_converge}. In addition, the function values at the neighbor of $x_0$ is lower than that of its negative multiple, which therefore verifies the result in Theorem \ref{thm:optimum_compare}. The landscape will further inspire us to design efficient algorithms to solve the ERM in \eqref{eq:erm}.

\subsection{Connections with Invertibility of Neural Network}

As a straightforward corollary to Theorems \ref{thm:determine_converge} and \ref{thm:optimum_compare}, we obtain the approximate invertibility of ReLU network under noisy quantized measurements. Specifically, previous results \citep{hand2018global, gilbert2017towards, arora2015deep} show that under proper assumptions, one can invert the neural network (NN) and approximate $x_0$ by observing the outcome $G(x_0)$ and solving $\argmin_{x\in\mathbb{R}^d}\|G(x) - G(x_0)\|_2$. Here, we consider a generalized version of the previous setting in the sense that instead of observing the full $G(x_0)$, we only observe the randomly probed and quantized information $ \frac{\lambda}{m}\sum_{i=1}^m\sign(\dotp{a_i}{G(x_0)} + \tau_i)a_i$. Theorems \ref{thm:determine_converge} and \ref{thm:optimum_compare} essentially show that by solving following minimization problem:
$
\argmin_{x\in\mathbb{R}^k}\big\|G(x) - \frac{\lambda}{m}\sum_{i=1}^m\sign(\dotp{a_i}{G(x_0)} + \tau)a_i\big\|_2,
$
one can still invert the NN and approximate the true representation $x_0$.

On the other hand, without this random sensing vector $a_i$,  
it is not always possible to approximate $x_0$ via directly quantized measurements $\sign([G(x_0)]_i +  \tau_i), \forall i \in [d]$.
A simple example would be a $G(x_0)$ which is exactly sparse (e.g. $G(x_0) = \sigma\circ([\mf I_{k\times k}~ \mf{0}_{k\times(d-k)}]^\top x_0)$)
and $x_0$ is entrywise positive. Then, $G(x_0)$ corresponds to a vector with first $k$ entries being $x_0$ and other entries 0. In this case, the observations $\sign([G(x_0)]_i +  \tau_i), \forall i\in [d]$, are just 
$\sign(x_{0,i}+\tau_i), \forall i \in [k]$, and 0 otherwise. It is then obvious to see that any estimation procedure would incur a constant error estimating $x_0$ regardless of the choices $\tau_i$.

\section{Proofs of Main Results}

\subsection{Proof of Theorem \ref{thm:main-1}} \label{sec:proof_main}

Consider the excessive risk $L(x) - L(x_0)$ for any $x\in\mathbb{R}^k$. Our goal is to show that under the conditions that $m$ is sufficiently large and $\lambda$ is set properly, with high probability, for any $x\in\mathbb{R}^k$ and any $x_0\in\mathbb{R}^k$ satisfying $\|G(x_0)\|_2\leq R$, if $\|G(x) - G(x_0)\|_2>\varepsilon$, then $L(x) - L(x_0)>0$ holds. 
By proving this claim, we can get that for $\hat{x}_m$, i.e. the solution to \eqref{eq:erm}, satisfying $L(\hat{x}_m) \leq L(x_0)$, then $\|G(\hat{x}_m) - G(x_0)\|_2  \leq \varepsilon$ holds with high probability. 

Recall that $\{(y_i,a_i)\}_{i=1}^m$ are $m$ i.i.d. copies of $(y,a)$ defined in \eqref{eq:def-y}. For abbreviation, across this section, we let 
\begin{align}
\Delta^G_{x,x_0}:= G(x) - G(x_0).
\end{align}
Then, we have the following decomposition
\begin{align*}
&L(x) - L(x_0) \nonumber\\
&= \|G(x)\|_2^2 - \|G(x_0)\|_2^2 - \frac{2\lambda}{m}\sum_{i=1}^m y_i\dotp{a_i}{\Delta^G_{x,x_0}} \nonumber \\
&=  \underbrace{\|G(x)\|_2^2 - \|G(x_0)\|_2^2 - 2\lambda\expect{y_i\dotp{a_i}{\Delta^G_{x,x_0}}}}_{\text{(I)}}  - \underbrace{ \frac{2\lambda}{m}\sum_{i=1}^m\big(y_i\dotp{a_i}{\Delta^G_{x,x_0}} - \expect{y_i\dotp{a_i}{\Delta^G_{x,x_0}}}\big)}_{\text{(II)}}.
\end{align*}
The term (I) is the bias of the expected risk, and the term (II) is the variance resulting from the empirical risk. Thus, to see whether $L(x)-L(x_0) > 0$ when $\|\Delta^G_{x,x_0}\|_2 > \varepsilon$, we focus on showing the lower bound of term (I) and the upper bound of term (II). For term (I), we give its lower bound according to the following lemma.
\begin{lemma}\label{lem:expected-risk}
Letting $K_{a,\xi,R} = \|a\|_{\psi_1}R + \| \xi\|_{\psi_1}$, there exists an absolute constant $c_1>0$ such that
\begin{align*}
\l|\expect{y_i\dotp{a_i}{\Delta^G_{x,x_0}}} - \lambda^{-1}  \dotp{G(x_0)}{\Delta^G_{x,x_0}} \r| \leq \sqrt{c_1 K_{a,\xi,R}}(\sqrt{2(\lambda+1)}+2)e^{-\lambda/(2K_{a,\xi,R})} \|\Delta^G_{x,x_0}\|_2.
\end{align*}
Moreover, $\forall\varepsilon\in(0,1)$, if $\lambda\geq4C_{a,\xi,R}\cdot\log(64C_{a,\xi,R}\cdot\varepsilon^{-1})$ with $C_{a,\xi,R} = \max \{c_1 K_{a,\xi,R}, 1\}$, 
and $\|\Delta^G_{x,x_0}\|_2 >\varepsilon$, then
\[
\|G(x)\|_2^2 - \|G(x_0)\|_2^2 - 2\lambda\expect{y_i\l \langle a_i, \Delta^G_{x,x_0}\r\rangle}
\geq \frac{1}{2}\|\Delta^G_{x,x_0}\|_2^2.
\]
\end{lemma}
It shows that $\text{term (I)} \geq \frac{1}{2}\|\Delta^G_{x,x_0}\|_2^2$ when $\|\Delta^G_{x,x_0}\|_2 >\varepsilon$. This lemma is proved via the ingredient of dithering, i.e., artificially adding the noise smooths the $\sign(\cdot)$ function. To see this, for a fixed $V$, it holds  that 
$$\EE_\tau[\sign(V+\tau)] =  \frac{V}{\lambda} \mathbf{1}_{\{|V|\leq  \lambda\}} + \mathbf{1}_{\{V> \lambda\}} - \mathbf{1}_{\{V< -\lambda\}},$$ 
where the dithering noise $\tau\sim\text{Unif}[-\lambda,+\lambda]$, and $\mathbf{1}_{\{\cdot\}}$ is an indicator function. As a consequence, $\expect{y_i|a_i, \xi_i}=(\dotp{a_i}{G(x_0)} + \xi_i)/\lambda$ given that 
$|\dotp{a_i}{G(x_0)} + \xi_i|$ is not too large, 
and then Lemma \ref{lem:expected-risk} follows. Detailed proof can be found in Supplement \S\ref{sec:proof-bias}.

Next, we present the analysis for showing the upper bound of the term (II), which is the key to proving Theorem \ref{thm:main-1}. To give the upper bound of term (II), it suffices to bound the following supremum over all $x\in\mathbb{R}^k$ and all $x_0$ satisfying $x_0\in\mathbb{R}^k,~\|G(x_0)\|_2\leq R$:
\begin{equation}\label{eq:variance-bound-1}
\sup \frac{\l|\frac1m\sum_{i=1}^my_i\dotp{a_i}{\Delta^G_{x,x_0}} -\expect{y_i\dotp{a_i}{\Delta^G_{x,x_0}}}\r|}{\|\Delta^G_{x,x_0}\|_2}.
\end{equation}
Recall that $y_i = \sign(\dotp{a_i}{G(x_0)}+\xi_i+\tau_i)$. By symmetrization inequality (Lemma \ref{lemma:symmetry} in the supplement), the following lemma readily implies the similar bound for \eqref{eq:variance-bound-1}.
\begin{lemma}\label{lem:main-bound-on-sup}
Suppose Assumption \ref{as:moments} holds and the number of samples $m\geq c_2\lambda^2\log^2(\lambda m) \cdot [kn\log(ed) + k\log(2R) + k\log m+ u]/\varepsilon^2$ for some absolute constant $c_2$ large enough, then, 
with probability at least $1-c\exp(-u)$,
\[
\sup_{x_0\in\mathbb{R}^k,~\|G(x_0)\|_2\leq R, x\in\mathbb{R}^k}
\frac{\l|\frac1m\sum_{i=1}^m\varepsilon_i y_i \dotp{a_i}{\Delta^G_{x,x_0}}\r|}{\|\Delta^G_{x,x_0}\|_2}
\leq \frac{\varepsilon}{16\lambda},
\]
where $\{\varepsilon_i\}_{i=1}^m$ are i.i.d. Rademacher random variables and $c>0$ is an absolute constant.
\end{lemma}
We provide a proof sketch for Lemma \ref{lem:main-bound-on-sup} as below. Details can be found in Supplement \S \ref{sec:proof-var}. The main difficulty is the simultaneous supremum over both $x_0$ and $x$, whereas in ordinary uniform concentration bounds (e.g. in non-uniform recovery), one only requires to bound a supremum over $x$. The idea is to consider a $\delta$-covering net over the set 
$G(\mathbb{R}^k)\cap \mathbb{B}_2^d(R)$, namely $\mathcal{N}(G(\mathbb{R}^k)\cap \mathbb{B}_2^d(R),~\delta)$, and bounding the supremum over each individual covering ball.
The $\delta$ value has to be carefully chosen so as to achieve the following goals:
\begin{itemize}[leftmargin=*] 
\item We replace each $\sign(\dotp{a_i}{G(x_0)} + \xi_i + \tau_i)$ by $\sign(\dotp{a_i}{G(v)} + \xi_i + \tau_i)$, where $G(v)$ is the nearest point to $G(x_0)$ in the $\delta$-net, and show that this supremum when fixing $G(v)$ is small. This is done via a ``one-step chaining'' argument making use of the piecewise linearity structure of $G$.
\item We consider the gap of such a replacement, i.e., the sign changes when replacing $G(x_0)$ by $G(v)$, and show that $d_{H}(G(x_0),G(v)):= \frac1m\sum_{i=1}^m \mathbf{1}_{\{\sign(\dotp{a_i}{G(x_0)} + \xi_i + \tau_i)  \neq  \sign(\dotp{a_i}{G(v)} + \xi_i + \tau_i)\}}$, which is the fraction of sign changes, is uniformly small for all $G(x_0)$ and $G(v)$ pairs. This can be rephrased as the \textit{uniform hyperplane tessellation problem}: Given an accuracy level 
$\varepsilon>0$, for any two points 
$\theta_1,\theta_2\in G(\mathbb{R}^k)\cap \mathbb{B}_2^d(R)$ such that $\|\theta_1-\theta_2\|_2\leq \delta$, what is the condition on $m$ and $\delta$ such that 
$d_H(\theta_1, \theta_2)\leq \|\theta_1-\theta_2\|_2 + \varepsilon$? We answer this question with a tight sample bound on $m$ in terms of $\varepsilon$ by counting the number of linear pieces in $G(\cdot)$ with a VC dimension argument.
\item We bound the error regarding a small portion of the indices $\{1,2,\cdots,m\}$ for which the signs do change in the previous replacement, and take a union bound over the $\delta$-net.
\end{itemize}


\begin{proof}[Proof of Theorem \ref{thm:main-1}]
By Lemma \ref{lem:main-bound-on-sup} and symmetrization inequality (Lemma \ref{lemma:symmetry} in the supplement), one readily gets that \eqref{eq:variance-bound-1} is bounded by $\varepsilon/8\lambda$ with probability at least $1-c_3\exp(-u)$. This further implies the following bound
\[
\frac{2\lambda}{m}\sum_{i=1}^m\big(y_i\dotp{a_i}{\Delta^G_{x,x_0}} - \expect{y_i\dotp{a_i}{\Delta^G_{x,x_0}}}\big)\leq \frac{\varepsilon}{4}
\|\Delta^G_{x,x_0}\|_2.
\]
Thus, when $\|\Delta^G_{x,x_0}\|_2 > \varepsilon$, the left-hand side of the above inequality is further bounded by $\|\Delta^G_{x,x_0}\|_2^2/4$. Combining with Lemma \ref{lem:expected-risk}, we finally obtain 
$L(x) -L(x_0) = \text{(I)} - \text{(II)} \geq \|\Delta^G_{x,x_0}\|_2^2/4>0$, if $\|\Delta^G_{x,x_0}\|_2 > \varepsilon$. Note that with high probability, this inequality holds for any $x\in\mathbb{R}^k$ and $x_0\in\mathbb{R}^k$ satisfying $\|G(x_0)\|_2\leq R$. This further implies $\|G(\hat x_m)-G(x_0)\|_2\leq \varepsilon$, which finishes the proof.
\end{proof}

\subsection{Proof of Theorem \ref{thm:lower-bound}} \label{sec:proof_lower_bound}

The key to proving Theorem \ref{thm:lower-bound} is to build a connection between our problem and the sparse recovery problem. Then we can further analyze the lower bound by employing tools from the area of sparse recovery. The detailed proofs for this subsection are presented in Supplement \S\ref{sec:detailed_proof_lower_bound}. 

\begin{definition}
A vector $ v\in\mathbb R^d$ is $k$-group sparse if, when dividing $ v$ into $k$ blocks of sub-vectors of size $d/k$,\footnote{We assume WLOG that $d/k$ is an integer.} each block has exactly one non-zero entry. 
\end{definition}
We establish the following proposition to build a connection between the ReLU network and the group sparse vector.
\begin{proposition}\label{prop:transform-main}
Any nonnegative $k$-group sparse vector in $\mathbb B_2^d(1)$ can be generated by a ReLU network of the form \eqref{eq:relu} with a $k+1$ dimensional input and depth $n=3$.
\end{proposition}

  The idea is to map each of the first $k$ input entries into one block in $\mathbb{R}^d$ of length $d/k$ respectively, and use the remaining one entry to construct proper offsets. We construct this mapping via a ReLU network \emph{with no offset} as follows:
         
        Consider a three-layer ReLU network, which has $k+1$ dimensional input of the form: $[x_1,\cdots,x_k,z]^\top \allowbreak \in\R^{k+1}$. The first hidden layer has the width of $(k+2d/k)$ whose first $k$ nodes outputs $\sigma(x_i)$, $\forall i\in[k]$, and the next $2d/k$ nodes output $\sigma(r\cdot z), \forall r\in [2d/k]$, which become the offset terms for the second layer.
        Then, with $\sigma(x_i)$ and $\sigma(r\cdot z)$ from the first layer, the second hidden layer will output the values of $\Upsilon_r(x_i,z)=\sigma(\sigma(x_i)-2\sigma(r\cdot z))$ and $\Upsilon'_r(x_i,z) = \sigma(\sigma(x_i)-2\sigma(r\cdot z) - \sigma(z))$, $\forall i\in [k]$ and $\forall r\in [d/k]$. 
        Finally, by constructing the third layer, we have the following mapping: $\forall i\in [k]$ and $\forall r\in [d/k]$, $\Gamma_r(x_i ,z) := \sigma\big(\Upsilon_r(x_i,z) - 2\Upsilon'_r(x_i,z)\big)$.
        
         Note that $\Gamma_r(x_i ,z)$ fires only when $x_i\geq0$, in which case we have $\sigma(x_i)=x_i$. Letting $z$ always equal to 1, we can observe that $\{\Gamma_r(x_i,1)\}_{r=1}^{d/k}$ is a sequence of $d/k$ non-overlapping triangle functions on the positive real line with width $2$ and height $1$. Therefore, the function $\Gamma_r(x_i,1)$ can generate the value of the $r$-th entry in the $i$-th block of a nonnegative $k$-group sparse vector in $\mathbb{B}_2^d(1)$.

The above proposition implies that the set of nonnegative $k$-group sparse vectors in $\mathbb{B}^d_2(1)$ is the subset of $G(\R^{k+1})\cap \mathbb{B}^d_2(1)$ where $G(\cdot)$ is defined by the mapping $\Gamma$.

\begin{lemma}\label{lem:linear-model-lower-bound-final}
Assume that $\theta_0\in K \subseteq \mathbb{B}^d_2(1)$ where $K$ is a set containing any $k$-group sparse vectors in $\mathbb{B}^d_2(1)$, and $K$ satisfies that $\forall v\in K$ then $ \lambda v \in K , \forall \lambda \in [0,1)$. Assume that $\check{y} = \dotp{a}{\theta_0} + \xi$ with $\xi\sim\mathcal{N}(0,\sigma^2)$ and $ a\sim\mathcal N(0, \mf I_d)$. Then, there exist absolute constants $c_1, c_2>0$ such that any estimator $\widehat\theta$ which depends only on $m$ observations of  $(a, \check{y})$ satisfies that when $m \geq c_1 k\log(d/k)$, there is 
$$\sup_{\theta_0\in K} \EE \|\widehat{\theta} - \theta_0\|_2 \geq c_2\sqrt{\frac{k\log(d/k)}{m}}.$$
\end{lemma}

Then, we are ready to show the proof of Theorem \ref{thm:lower-bound}.
\begin{proof}[Proof of Theorem \ref{thm:lower-bound}] According to Proposition \ref{prop:transform-main}, let $G(\cdot)$ be defined by the mapping $\Gamma$. One can verify that $G(\lambda x) = \lambda G(x),~\forall \lambda\geq0$, by the positive homogeneity of ReLU network with no offsets. Letting $K = G(\mathbb R^{k+1})\cap \mathbb B_2^d(1)$ and then by Lemma \ref{lem:linear-model-lower-bound-final}, we can obtain Theorem \ref{thm:lower-bound}, which completes the proof.
\end{proof}

\subsection{Proof Outline of Theorem \ref{thm:determine_converge} and Theorem \ref{thm:optimum_compare} } 

The key to proving Theorems \ref{thm:determine_converge} and \ref{thm:optimum_compare} lies in understanding the concentration of $L(x)$ and $\nabla L(x)$. We prove two critical lemmas, Lemmas \ref{lem:comp-expect} and \ref{lem:comp-variance} in Supplement \S \ref{sec:proof_thm_22}, to show that, with high probability, when $\lambda$ and $m$ are sufficiently large, for any $x$, $z$ and $x_0$ such that $|G(x_0)|\leq R$, the following holds 
\begin{align}
\bigg| \bigg\langle  \frac{\lambda }{m}\sum_{i=1}^m y_i a_i - G(x_0), H_x(z) \bigg\rangle \bigg| \leq \varepsilon \|H_x(z)\|_2,   \label{eq:one-bit-rc}
\end{align}
where $H_x(z) := \prod_{j=1}^n W_{j, +, x}  z$ and $G(x) = H_x(x)$. 
In particular, this replaces the range restricted isometry condition (RRIC) adopted in previous works \citep{hand2018global}. Under the conditions of Theorems \ref{thm:determine_converge} and \ref{thm:optimum_compare}, the inequality \eqref{eq:one-bit-rc} essentially implies 
\begin{align*}
\frac{\lambda}{m} \sum_{i=1}^m y_i \dotp{a_i}{H_x(z)} \approx \dotp{G(x_0)}{H_x(z)}, \forall x, z.
\end{align*} 
Therefore, by definition of $L(x)$ in \eqref{eq:erm}, we can approximate $\nabla L(x)$ and $L(x)$ as follows:
\begin{align}
\langle\nabla L(x),z\rangle &\approx 2 \langle G(x), H_x(z) \rangle - 2 \langle G(x_0), H_x(z) \rangle, 
 \label{eq:approx_grad} \\
 L(x)  &\approx \|G(x_0)\|_2^2 - 2 \langle G(x_0), G(x)\rangle. 
 \label{eq:approx_erm}
\end{align}
We give a sketch proof of Theorem \ref{thm:determine_converge} as follows. Please see Supplement \S\ref{sec:proof_thm_22} for proof details.
\begin{itemize}[leftmargin=*]
\item We show that $\forall x, z$, $\langle G(x), H_x(z) \rangle -  \langle G(x_0), H_x(z) \rangle \approx \langle  h_{x,x_0},z \rangle$, where we define a certain approximation function $h_{x,x_0} :=2^{-n} x - 2^{-n} \big[ \big(  \prod_{i=0}^{n-1} \frac{\pi-\overline{\varrho}_i}{\pi}  \big) x_0 + \sum_{i=0}^{n-1} \frac{\sin \overline{\varrho}_i}{\pi} \big( \prod_{j=i+1}^{d-1}\frac{\pi - \overline{\varrho}_j}{\pi} \big) \frac{\|x_0\|_2}{\|x\|_2} x \big]$ with $\overline{\varrho}_i = g(\overline{\varrho}_{i-1})$, $\overline{\varrho}_0 = \angle (x,x_0)$, and $g(\varrho):=\cos^{-1} \l( \frac{(\pi-\varrho)\cos\varrho + \sin \varrho}{\pi} \r)$ as shown in Lemmas \ref{lem:converge_set} and  \ref{lem:wdc_ineq}. Combining with \eqref{eq:approx_grad}, we obtain $\langle\nabla L(x),z\rangle \approx 2 \langle h_{x,x_0}, z \rangle$. 
\item With $v_x$ being defined in Theorem \ref{thm:determine_converge}, the directional derivative along the direction $v_x$ is approximated as $D_{-v_x} L(x) \|v_x\|_2 \approx - 4 \| h_{x,x_0}\|^2_2$ following 
the previous step. Particularly, $\| h_{x,x_0}\|_2$ being small implies $x$ is close to $x_0$ or $-\rho_n x_0$ by Lemma \ref{lem:converge_set} and $\| h_{x,x_0}\|_2$ gets small as $\|x_0\|_2$ approaches $0$.
\item We consider the error of approximating $D_{-v_x} L(x) \|v_x\|_2$ by $- 4 \| h_{x,x_0}\|^2_2$. When $\|x_0\|_2$ is not small, and $x\neq0$, one can show the error is negligible compared to $- 4 \| h_{x,x_0}\|^2_2$, so that by the previous step, one finishes the proof of Case 1 when $x\neq 0$. On the other hand, for Case 2, when $\|x_0\|_2$ approaches 0, such an error is decaying slower than $-4\| h_{x,x_0}\|^2_2$ itself and eventually dominates it. As a consequence, one can only conclude that $\hat{x}_m$ is around the origin.
\item For Case 1 when $x = 0$, we can show $D_w L(0) \cdot \|w\|_2\leq   |\langle G(x_0), H_{x_N}(w) \rangle - \lambda / m \sum_{i=1}^m y_i  \langle a_i, H_{x_N}(w)\rangle | \allowbreak - \langle G(x_0), H_{x_N}(w)\rangle$ with  $x_N\rightarrow 0$. By giving the upper bound of the first term and the lower bound of the second term according to \eqref{eq:one-bit-rc} and Lemma \ref{lem:wdc_ineq}, we  obtain $D_{w} L(0) < 0, \forall w\neq 0$.  
\end{itemize}

Theorem \ref{thm:optimum_compare} is proved in Supplement \S\ref{sec:proof_thm_23}. We have the following proof sketch. We show by \eqref{eq:approx_erm} that 
$L(x) \approx 2 \langle  h_{x, x_0},x \rangle - \|G(x)\|_2^2$ for any $x$.  With such approximation, by Lemmas \ref{lem:loss_error_bound}, \ref{lem:rho_lower_bound} in the supplement, under certain conditions, we have that if $x$ and $z$ are around $x_0$ and  $-\rho_n x_0$ respectively, $L(x) < L(z)$ holds.

\section{Conclusion}

We consider the problem of one-bit compressed sensing via ReLU generative networks, in which $G:\mathbb{R}^k\rightarrow\mathbb{R}^d$ is an $n$-layer ReLU generative network with a low dimensional representation $x_0$ to $G(x_0)$.  We propose to recover the target $G(x_0)$ solving an unconstrained empirical risk minimization problem. Under a weak sub-exponential measurement assumption, we establish a joint statistical and computational analysis. We prove that the ERM estimator in this new framework achieves a statistical rate of $m=\tilde{\mathcal{O}}(kn \log d /\varepsilon^2)$ recovering any $G(x_0)$ uniformly up to an error $\varepsilon$. When the network is shallow, this rate matches the information-theoretic lower bound up to logarithm factors of $\varepsilon^{-1}$. Computationally, we prove that under proper conditions on the network weights, the proposed empirical risk has no stationary point outside of small neighborhoods around the true representation $x_0$ and its negative multiple. Under further assumptions on weights, we show that the global minimizer of the empirical risk stays within the neighborhood around $x_0$ rather than its negative multiple.

\bibliographystyle{ims}
\bibliography{bibliography}

\begin{thebibliography}{35}
\expandafter\ifx\csname natexlab\endcsname\relax\def\natexlab#1{#1}\fi
\expandafter\ifx\csname url\endcsname\relax
  \def\url#1{\texttt{#1}}\fi
\expandafter\ifx\csname urlprefix\endcsname\relax\def\urlprefix{}\fi

\bibitem[{Ai et~al.(2014)Ai, Lapanowski, Plan and Vershynin}]{ai2014one}
\text{Ai, A.}, \text{Lapanowski, A.}, \text{Plan, Y.} and \text{Vershynin, R.}
  (2014).
\newblock One-bit compressed sensing with non-gaussian measurements.
\newblock \textit{Linear Algebra and its Applications}, \textbf{441} 222--239.

\bibitem[{Angluin and Valiant(1979)}]{angluin1979fast}
\text{Angluin, D.} and \text{Valiant, L.~G.} (1979).
\newblock Fast probabilistic algorithms for {H}amiltonian circuits and
  matchings.
\newblock \textit{Journal of Computer and system Sciences}, \textbf{18}
  155--193.

\bibitem[{Arora et~al.(2015)Arora, Liang and Ma}]{arora2015deep}
\text{Arora, S.}, \text{Liang, Y.} and \text{Ma, T.} (2015).
\newblock Why are deep nets reversible: A simple theory, with implications for
  training.
\newblock \textit{arXiv preprint arXiv:1511.05653}.

\bibitem[{Aubin et~al.(2019)Aubin, Loureiro, Maillard, Krzakala and
  Zdeborov{\'a}}]{aubin2019spiked}
\text{Aubin, B.}, \text{Loureiro, B.}, \text{Maillard, A.}, \text{Krzakala, F.}
  and \text{Zdeborov{\'a}, L.} (2019).
\newblock The spiked matrix model with generative priors.
\newblock \textit{arXiv preprint arXiv:1905.12385}.

\bibitem[{Bora et~al.(2017)Bora, Jalal, Price and Dimakis}]{bora2017compressed}
\text{Bora, A.}, \text{Jalal, A.}, \text{Price, E.} and \text{Dimakis, A.~G.}
  (2017).
\newblock Compressed sensing using generative models.
\newblock \textit{arXiv preprint arXiv:1703.03208}.

\bibitem[{Dirksen and Mendelson(2018{\natexlab{a}})}]{dirksen2018non}
\text{Dirksen, S.} and \text{Mendelson, S.} (2018{\natexlab{a}}).
\newblock Non-gaussian hyperplane tessellations and robust one-bit compressed
  sensing.
\newblock \textit{arXiv preprint arXiv:1805.09409}.

\bibitem[{Dirksen and Mendelson(2018{\natexlab{b}})}]{dirksen2018robust-2}
\text{Dirksen, S.} and \text{Mendelson, S.} (2018{\natexlab{b}}).
\newblock Robust one-bit compressed sensing with partial circulant matrices.
\newblock \textit{arXiv preprint arXiv:1812.06719}.

\bibitem[{Gilbert et~al.(2017)Gilbert, Zhang, Lee, Zhang and
  Lee}]{gilbert2017towards}
\text{Gilbert, A.~C.}, \text{Zhang, Y.}, \text{Lee, K.}, \text{Zhang, Y.} and
  \text{Lee, H.} (2017).
\newblock Towards understanding the invertibility of convolutional neural
  networks.
\newblock \textit{arXiv preprint arXiv:1705.08664}.

\bibitem[{Goldstein et~al.(2018)Goldstein, Minsker and
  Wei}]{goldstein2018structured}
\text{Goldstein, L.}, \text{Minsker, S.} and \text{Wei, X.} (2018).
\newblock Structured signal recovery from non-linear and heavy-tailed
  measurements.
\newblock \textit{IEEE Transactions on Information Theory}, \textbf{64}
  5513--5530.

\bibitem[{Goldstein and Wei(2018)}]{10.1093/imaiai/iay006}
\text{Goldstein, L.} and \text{Wei, X.} (2018).
\newblock {Non-Gaussian observations in nonlinear compressed sensing via Stein
  discrepancies}.
\newblock \textit{Information and Inference: A Journal of the IMA}, \textbf{8}
  125--159.
\newline\urlprefix\url{https://doi.org/10.1093/imaiai/iay006}

\bibitem[{Hammernik et~al.(2018)Hammernik, Klatzer, Kobler, Recht, Sodickson,
  Pock and Knoll}]{hammernik2018learning}
\text{Hammernik, K.}, \text{Klatzer, T.}, \text{Kobler, E.}, \text{Recht,
  M.~P.}, \text{Sodickson, D.~K.}, \text{Pock, T.} and \text{Knoll, F.} (2018).
\newblock Learning a variational network for reconstruction of accelerated mri
  data.
\newblock \textit{Magnetic resonance in medicine}, \textbf{79} 3055--3071.

\bibitem[{Hand and Joshi(2019)}]{hand2019global}
\text{Hand, P.} and \text{Joshi, B.} (2019).
\newblock Global guarantees for blind demodulation with generative priors.
\newblock \textit{arXiv preprint arXiv:1905.12576}.

\bibitem[{Hand et~al.(2018)Hand, Leong and Voroninski}]{hand2018phase}
\text{Hand, P.}, \text{Leong, O.} and \text{Voroninski, V.} (2018).
\newblock Phase retrieval under a generative prior.
\newblock In \textit{Advances in Neural Information Processing Systems}.

\bibitem[{Hand and Voroninski(2018)}]{hand2018global}
\text{Hand, P.} and \text{Voroninski, V.} (2018).
\newblock Global guarantees for enforcing deep generative priors by empirical
  risk.
\newblock In \textit{Conference On Learning Theory}.

\bibitem[{Huang et~al.(2018)Huang, Hand, Heckel and
  Voroninski}]{huang2018provably}
\text{Huang, W.}, \text{Hand, P.}, \text{Heckel, R.} and \text{Voroninski, V.}
  (2018).
\newblock A provably convergent scheme for compressive sensing under random
  generative priors.
\newblock \textit{arXiv preprint arXiv:1812.04176}.

\bibitem[{Jacques et~al.(2013)Jacques, Laska, Boufounos and
  Baraniuk}]{jacques2013robust}
\text{Jacques, L.}, \text{Laska, J.~N.}, \text{Boufounos, P.~T.} and
  \text{Baraniuk, R.~G.} (2013).
\newblock Robust 1-bit compressive sensing via binary stable embeddings of
  sparse vectors.
\newblock \textit{IEEE Transactions on Information Theory}, \textbf{59}
  2082--2102.

\bibitem[{Kamath et~al.(2019)Kamath, Karmalkar and Price}]{kamath2019lower}
\text{Kamath, A.}, \text{Karmalkar, S.} and \text{Price, E.} (2019).
\newblock Lower bounds for compressed sensing with generative models.
\newblock \textit{arXiv preprint arXiv:1912.02938}.

\bibitem[{Ledig et~al.(2017)Ledig, Theis, Husz{\'a}r, Caballero, Cunningham,
  Acosta, Aitken, Tejani, Totz, Wang et~al.}]{ledig2017photo}
\text{Ledig, C.}, \text{Theis, L.}, \text{Husz{\'a}r, F.}, \text{Caballero,
  J.}, \text{Cunningham, A.}, \text{Acosta, A.}, \text{Aitken, A.},
  \text{Tejani, A.}, \text{Totz, J.}, \text{Wang, Z.} \text{et~al.} (2017).
\newblock Photo-realistic single image super-resolution using a generative
  adversarial network.
\newblock In \textit{2017 IEEE Conference on Computer Vision and Pattern
  Recognition (CVPR)}. IEEE.

\bibitem[{Lei et~al.(2018)Lei, Luo, Yau and Gu}]{lei2018geometric}
\text{Lei, N.}, \text{Luo, Z.}, \text{Yau, S.-T.} and \text{Gu, D.~X.} (2018).
\newblock Geometric understanding of deep learning.
\newblock \textit{arXiv preprint arXiv:1805.10451}.

\bibitem[{Liu and Scarlett(2019)}]{liu2019information}
\text{Liu, Z.} and \text{Scarlett, J.} (2019).
\newblock Information-theoretic lower bounds for compressive sensing with
  generative models.
\newblock \textit{arXiv preprint arXiv:1908.10744}.

\bibitem[{Manoel et~al.(2017)Manoel, Krzakala, M{\'e}zard and
  Zdeborov{\'a}}]{manoel2017multi}
\text{Manoel, A.}, \text{Krzakala, F.}, \text{M{\'e}zard, M.} and
  \text{Zdeborov{\'a}, L.} (2017).
\newblock Multi-layer generalized linear estimation.
\newblock In \textit{2017 IEEE International Symposium on Information Theory
  (ISIT)}. IEEE.

\bibitem[{Pandit et~al.(2020)Pandit, Sahraee-Ardakan, Rangan, Schniter and
  Fletcher}]{pandit2020inference}
\text{Pandit, P.}, \text{Sahraee-Ardakan, M.}, \text{Rangan, S.},
  \text{Schniter, P.} and \text{Fletcher, A.~K.} (2020).
\newblock Inference with deep generative priors in high dimensions.
\newblock \textit{IEEE Journal on Selected Areas in Information Theory}.

\bibitem[{Plan and Vershynin(2013)}]{plan2013robust}
\text{Plan, Y.} and \text{Vershynin, R.} (2013).
\newblock Robust 1-bit compressed sensing and sparse logistic regression: A
  convex programming approach.
\newblock \textit{IEEE Transactions on Information Theory}, \textbf{59}
  482--494.

\bibitem[{Plan and Vershynin(2014)}]{plan2014dimension}
\text{Plan, Y.} and \text{Vershynin, R.} (2014).
\newblock Dimension reduction by random hyperplane tessellations.
\newblock \textit{Discrete \& Computational Geometry}, \textbf{51} 438--461.

\bibitem[{Plan et~al.(2016)Plan, Vershynin and Yudovina}]{plan2016high}
\text{Plan, Y.}, \text{Vershynin, R.} and \text{Yudovina, E.} (2016).
\newblock High-dimensional estimation with geometric constraints.
\newblock \textit{Information and Inference: A Journal of the IMA}, \textbf{6}
  1--40.

\bibitem[{S{\o}nderby et~al.(2016)S{\o}nderby, Caballero, Theis, Shi and
  Husz{\'a}r}]{sonderby2016amortised}
\text{S{\o}nderby, C.~K.}, \text{Caballero, J.}, \text{Theis, L.}, \text{Shi,
  W.} and \text{Husz{\'a}r, F.} (2016).
\newblock Amortised map inference for image super-resolution.
\newblock \textit{arXiv preprint arXiv:1610.04490}.

\bibitem[{Thrampoulidis and Rawat(2018)}]{thrampoulidis2018generalized}
\text{Thrampoulidis, C.} and \text{Rawat, A.~S.} (2018).
\newblock The generalized lasso for sub-gaussian measurements with dithered
  quantization.
\newblock \textit{arXiv preprint arXiv:1807.06976}.

\bibitem[{Wei et~al.(2019)Wei, Yang and Wang}]{wei2019statistical}
\text{Wei, X.}, \text{Yang, Z.} and \text{Wang, Z.} (2019).
\newblock On the statistical rate of nonlinear recovery in generative models
  with heavy-tailed data.
\newblock In \textit{International Conference on Machine Learning}.

\bibitem[{Wellner et~al.(2013)}]{wellner2013weak}
\text{Wellner, J.} \text{et~al.} (2013).
\newblock \textit{Weak convergence and empirical processes: with applications
  to statistics}.
\newblock Springer Science \& Business Media.

\bibitem[{Winder(1966)}]{winder1966partitions}
\text{Winder, R.} (1966).
\newblock Partitions of n-space by hyperplanes.
\newblock \textit{SIAM Journal on Applied Mathematics}, \textbf{14} 811--818.

\bibitem[{Xu and Jacques(2018)}]{xu2018quantized}
\text{Xu, C.} and \text{Jacques, L.} (2018).
\newblock Quantized compressive sensing with rip matrices: The benefit of
  dithering.
\newblock \textit{arXiv preprint arXiv:1801.05870}.

\bibitem[{Yang et~al.(2018)Yang, Yu, Dong, Slabaugh, Dragotti, Ye, Liu,
  Arridge, Keegan, Guo et~al.}]{yang2018dagan}
\text{Yang, G.}, \text{Yu, S.}, \text{Dong, H.}, \text{Slabaugh, G.},
  \text{Dragotti, P.~L.}, \text{Ye, X.}, \text{Liu, F.}, \text{Arridge, S.},
  \text{Keegan, J.}, \text{Guo, Y.} \text{et~al.} (2018).
\newblock Dagan: Deep de-aliasing generative adversarial networks for fast
  compressed sensing mri reconstruction.
\newblock \textit{IEEE transactions on medical imaging}, \textbf{37}
  1310--1321.

\bibitem[{Yeh et~al.(2017)Yeh, Chen, Yian~Lim, Schwing, Hasegawa-Johnson and
  Do}]{yeh2017semantic}
\text{Yeh, R.~A.}, \text{Chen, C.}, \text{Yian~Lim, T.}, \text{Schwing, A.~G.},
  \text{Hasegawa-Johnson, M.} and \text{Do, M.~N.} (2017).
\newblock Semantic image inpainting with deep generative models.
\newblock In \textit{Proceedings of the IEEE Conference on Computer Vision and
  Pattern Recognition}.

\bibitem[{Zhang et~al.(2014)Zhang, Yi and Jin}]{zhang2014efficient}
\text{Zhang, L.}, \text{Yi, J.} and \text{Jin, R.} (2014).
\newblock Efficient algorithms for robust one-bit compressive sensing.
\newblock In \textit{International Conference on Machine Learning}.

\bibitem[{Zhu and Gu(2015)}]{zhu2015towards}
\text{Zhu, R.} and \text{Gu, Q.} (2015).
\newblock Towards a lower sample complexity for robust one-bit compressed
  sensing.
\newblock In \textit{International Conference on Machine Learning}.

\end{thebibliography}

\newpage
\appendix{}

\onecolumn
\section{Proof of Theorem \ref{thm:main-1}} \label{sec:detailed_proof_main-1}

In this section, we provide the proofs of the two key lemmas, i.e., Lemma \ref{lem:expected-risk} and Lemma \ref{lem:main-bound-on-sup} as well as other supporting lemmas. The proof of Theorem \ref{thm:lower-bound} is immediately obtained by following Lemma \ref{lem:expected-risk} and Lemma \ref{lem:main-bound-on-sup} as shown in Section \S\ref{sec:proof_main}.

\subsection{Bias of the Expected Risk} 
\label{sec:proof-bias}
We prove Lemma \ref{lem:expected-risk} in this subsection.
\begin{lemma}[Lemma \ref{lem:expected-risk}]\label{lem:expected-risk-supp}
There exists an absolute constant $c_1>0$ such that the following holds:
\begin{align*}
&\l|\expect{y_i\dotp{a_i}{G(x) - G(x_0)}} - \frac{1}{\lambda}\dotp{G(x_0)}{G(x) - G(x_0)} \r|\\
&\qquad \leq \sqrt{c_1(\|a\|_{\psi_1}R + \| \xi\|_{\psi_1})}(\sqrt{2(\lambda+1)}+2)e^{-\lambda/2(\|a\|_{\psi_1}R + \| \xi\|_{\psi_1})} 
\|G(x) - G(x_0)\|_2.
\end{align*}
Furthermore, for any $\varepsilon\in(0,1)$, if $\lambda\geq4C_{a,\xi,R}\cdot\log(64C_{a,\xi,R}\cdot\varepsilon^{-1})$ where $C_{a,\xi,R}=\max\{c_1(R\|a\|_{\psi_1}+\|\xi\|_{\psi_1}),1\}$, and $\|G(x) - G(x_0)\|_2 > \varepsilon$, then, we have 
\[
\|G(x)\|_2^2 - \|G(x_0)\|_2^2 - 2\lambda\expect{y_i\dotp{a_i}{G(x) - G(x_0)}}
\geq \frac12\|G(x) - G(x_0)\|_2^2.
\]
\end{lemma}
\begin{proof}[Proof of Lemma \ref{lem:expected-risk}]
Recall that $y_i= \sign(\dotp{a_i}{G(x_0)} + \xi_i + \tau_i)$. For simplicity of notations, we set $V_i = \dotp{a_i}{G(x_0)} + \xi_i$ and 
$Z_i =\dotp{a_i}{G(x) - G(x_0)} $. Note first that due to the independence between $V_i$ and $\tau_i$, we have
\begin{align*}
\expect{\sign(V_i+\tau_i)|V_i} =& \frac{V_i}{\lambda}\mathbf{1}_{\{|V_i|\leq \lambda\}} + \mathbf{1}_{\{V_i> \lambda\}} - \mathbf{1}_{\{V_i< -\lambda\}}\\
=&\frac{V_i}{\lambda} -  \frac{V_i}{\lambda} \mathbf{1}_{\{|V_i|> \lambda\}} + \mathbf{1}_{\{V_i> \lambda\}} - \mathbf{1}_{\{V_i< -\lambda\}}.
\end{align*}
Thus, we have
\begin{align}
\l|\expect{Z_i\sign(V_i+\tau_i)} - \frac{\expect{Z_iV_i}}{\lambda}  \r|
&= \l|-\expect{\frac{Z_iV_i}{\lambda}\mathbf{1}_{\{|V_i|> \lambda\}}}  + \expect{Z_i\mathbf{1}_{\{V_i> \lambda\}}} -  \expect{Z\mathbf{1}_{\{V_i> \lambda\}}}\r|
\nonumber\\
&  \leq \l|\expect{\frac{Z_iV_i}{\lambda}\mathbf{1}_{\{|V_i|> \lambda\}}}\r| + 2\l| \expect{Z_i\mathbf{1}_{\{|V_i|> \lambda\}}}\r|\nonumber\\
&  \leq \frac{\|Z_i\|_{L_2}\cdot\|V_i\mathbf{1}_{\{|V_i|> \lambda\}}\|_{L_2}}{\lambda} + 2\|Z_i\|_{L_2} \mathrm{Pr}(|V_i|> \lambda)^{1/2},
\label{eq:inter-exp-analysis} 
\end{align}
where the last line follows from Cauchy-Schwarz inequality. Now we bound these terms respectively. First of all, by the isotropic assumption of $a_i$, we have
\[
\|Z_i\|_{L_2} = \l\{\expect{|\dotp{a_i}{G(x) - G(x_0)}|^2}\r\}^{1/2} = \|G(x) - G(x_0)\|_2.
\]
Next, we have
\begin{align*}
\|V_i\mathbf{1}_{\{|V_i|> \lambda\}}\|_{L_2} &= \expect{V_i^2\mathbf{1}_{\{|V_i|> \lambda\}}}^{1/2}
= \l(\int_{\lambda}^{\infty}w^2dP(w)\r)^{1/2} \\
&= \l(2\int_{\lambda}^{\infty}wP(|V_i|> w)dw\r)^{1/2}
\leq \l(2c_1\int_{\lambda}^{\infty}we^{-w/\|\dotp{a_i}{G(x_0)} + \xi_i\|_{\psi_1}}dw\r)^{1/2}\\
&\leq \sqrt{2c_1(\lambda+1)\|\dotp{a_i}{G(x_0)} + \xi_i\|_{\psi_1}}e^{-\lambda/2\|\dotp{a_i}{G(x_0)} + \xi_i\|_{\psi_1}},
\end{align*}
where the second  from the last inequality follows from sub-exponential assumption of $\dotp{a_i}{G(x_0)} + \xi_i$ and $c_1>0$ is an absolute constant. 
Note that 
\[
\|\dotp{a_i}{G(x_0)} + \xi_i\|_{\psi_1}\leq \|\dotp{a_i}{G(x_0)}\|_{\psi_1} + \| \xi_i\|_{\psi_1} 
\leq\|a\|_{\psi_1}\|G(x_0)\|_2 + \| \xi\|_{\psi_1} \leq \|a\|_{\psi_1}R + \| \xi\|_{\psi_1},
\]
where we use the assumption that $\|G(x_0)\|_2\leq R$. Substituting this bound into the previous one gives 
\[
\|V_i\mathbf{1}_{\{|V_i|> \lambda\}}\|_{L_2}\leq 
\sqrt{2c_1(\lambda+1)(\|a\|_{\psi_1}R + \| \xi\|_{\psi_1})}e^{-\lambda/2(\|a\|_{\psi_1}R + \| \xi\|_{\psi_1})}.
\]
Furthermore, 
\[
\mathrm{Pr}(|V_i|> \lambda)^{1/2}\leq \sqrt{c_1(\|a\|_{\psi_1}R + \| \xi\|_{\psi_1})}e^{-\lambda/2(\|a\|_{\psi_1}R + \| \xi\|_{\psi_1})}.
\]
Overall, substituting the previous computations into \eqref{eq:inter-exp-analysis}, we obtain
\begin{align*}
&\l|\expect{Z_i\sign(V_i+\tau_i)} - \frac{\expect{Z_iV_i}}{\lambda}  \r|\\
&\qquad \leq \sqrt{c_1(\|a\|_{\psi_1}R + \| \xi\|_{\psi_1})}(\sqrt{2(\lambda+1)}/\lambda+2)e^{-\lambda/2(\|a\|_{\psi_1}R + \| \xi\|_{\psi_1})} 
\|G(x) - G(x_0)\|_2,
\end{align*}
finishing the first part of the proof.

To prove the second part, we need to compute 
$$
\Big|2\lambda\expect{y_i\dotp{a_i}{G(x) - G(x_0)}} - 2\dotp{G(x_0)}{G(x) - G(x_0)}\Big|
=2\Big|\lambda\expect{Z_i\sign(V_i+\tau_i)} - \expect{Z_iV_i}  \Big|.
$$
Note that when $\varepsilon<1$ and
$$\lambda\geq4\max\{c_1(R\|a\|_{\psi_1}+\|\xi\|_{\psi_1}),1\}\log(64\max\{c_1(R\|a\|_{\psi_1} + \|\xi\|_{\psi_1}),1\}/\varepsilon).$$
One can check that
\begin{align*}
&\l|\lambda\expect{Z_i\sign(V_i+\tau_i)} - \expect{Z_iV_i}  \r|\\
& \qquad \leq \sqrt{c_1(\|a\|_{\psi_1}R + \| \xi\|_{\psi_1})}(\sqrt{2(\lambda+1)}+2\lambda)e^{-\lambda/2(\|a\|_{\psi_1}R + \| \xi\|_{\psi_1})}
\|G(x) - G(x_0)\|_2\\
& \qquad \leq\frac14\varepsilon\|G(x) - G(x_0)\|_2.
\end{align*}
Thus, it follows
\begin{align*}
&\|G(x)\|_2^2 - \|G(x_0)\|_2^2 - 2\lambda\expect{y_i\dotp{a_i}{G(x) - G(x_0)}}\\
& \qquad \geq  \|G(x)\|_2^2 - \|G(x_0)\|_2^2 - 2\dotp{G(x_0)}{G(x) - G(x_0)}  - \frac12\varepsilon\|G(x) - G(x_0)\|_2\\
& \qquad = \|G(x) - G(x_0)\|_2^2 - \frac12\varepsilon\|G(x) - G(x_0)\|_2.
\end{align*}
Thus, when $\|G(x) - G(x_0)\|_2>\varepsilon$ the second claim holds.
\end{proof}

\subsection{Analysis of Variances: Uniform Bounds of An Empirical Process}\label{sec:proof-var}
Our goal in this subsection is to prove Lemma \ref{lem:main-bound-on-sup}.
Note that one can equivalently write the $\{G(x_0): \|G(x_0)\|_2\leq R,~x_0\in\mathbb{R}^k\}$ as $G(\mathbb{R}^k)\cap \mathbb{B}_2^d(R)$, where $\mathbb{B}_2^d(R)$ denotes the $\ell_2$-ball of radius $R$. 
The strategy of bounding this supremum is as follows: Consider a $\delta$-covering net over the set 
$G(\mathbb{R}^k)\cap \mathbb{B}_2^d(R)$, namely $\mathcal{N}(G(\mathbb{R}^k)\cap \mathbb{B}_2^d(R),~\delta)$, and bounding the supremum over each individual covering ball.
The $\delta$ value will be decided later.

\subsubsection{Bounding Supremum Under Fixed Signs: A Covering Net Argument}
First of all, since for any point $\theta\in G(\mathbb{R}^k)\cap \mathbb{B}_2^d(R)$, there exists a $v\in\mathbb{R}^k$ such that 
$\theta = G(v)$, we use $G(v)$ to denote any point in the net $\mathcal{N}(G(\mathbb{R}^k)\cap \mathbb{B}_2^d(R),~\delta)$. 
We replace each $\sign(\dotp{a_i}{G(x_0)} + \xi_i + \tau_i)$ by $\sign(\dotp{a_i}{G(v)} + \xi_i + \tau_i)$ and have the following lemma regarding the supremum for each fixed $G(v)$.

\begin{lemma}\label{lem:ball-supremum}
Let $c,c_1>0$ be some absolute constants. For any $u\geq0$ and fixed $G(v)$, the following holds with probability at least $1-2\exp(-u-c_1 kn\log ed)$,
\begin{align*}
&\sup_{x\in\mathbb{R}^k,~x_0\in\mathbb{R}^k} \frac{\l|\frac1m\sum_{i=1}^m\varepsilon_i\sign(\dotp{a_i}{G(v)} + \xi_i + \tau_i)\dotp{a_i}{G(x) - G(x_0)}\r|}{\|G(x) - G(x_0)\|_2} \\
&\qquad \leq \sqrt{\frac{8(u+ckn\log(ed))}{m}} + \frac{2\|a\|_{\psi_1}(u+ckn\log(ed))}{m}.
\end{align*}
\end{lemma}

\begin{proof}[Proof of Lemma \ref{lem:ball-supremum}]
First of all, since $v$ is fixed and $\varepsilon_i$ is independent of $ \sign(\dotp{a_i}{G(v)} + \xi_i + \tau_i)$, it follows the distribution of 
$\varepsilon_i$ is the same as the distribution of $\varepsilon_i \sign(\dotp{a_i}{G(v)} + \xi_i + \tau_i)$. Thus, it is enough to work with the following supremum:
\[
\sup_{x\in\mathbb{R}^k,~x_0\in\mathbb{R}^k} \frac{\l|\frac1m\sum_{i=1}^m\varepsilon_i\dotp{a_i}{G(x) - G(x_0)}\r|}{\|G(x) - G(x_0)\|_2}.
\]
To this point, we will then use the piecewise linear structure of the ReLU function. Note that the ReLU network has $n$ layers with each layer having at most $d$ nodes, where each layer of the network is a linear transformation followed by at most $d$ pointwise nonlinearities. Consider any node in the first layer, which can be written as $\max\{\dotp{w}{x},0\}$ with a weight vector $w$ and an input vector $x$, splits the input space $\mathbb{R}^k$ into two disjoint pieces, namely 
$\mathcal{P}_1$ and $\mathcal{P}_2$, where for any input in $\mathcal{P}_1$, the node is a linear mapping $\dotp{w}{x}$ and for any input in $\mathcal{P}_2$ is the other linear mapping $\dotp{0}{x}$. 

Thus, each node in the first layer corresponds to a splitting hyperplane in $\mathbb R^k$. We have the following claim on the number of possible pieces split by $d$ hyperplanes:

\noindent \textit{Claim 1:} The maximum number of pieces when splitting $\mathbb R^k$ with $d$ hyperplanes, denoted as 
$\mathcal{C}(d,k)$, is
\[
\mathcal{C}(d,k) = {d\choose 0} + {d\choose 1} + \cdots + {d\choose k}.
\]
The proof of this claim, which follows from, for example \cite{winder1966partitions}, is based on an induction argument on both $d$ and $k$ and omitted here for brevity. Note that $\mathcal{C}(d,k)\leq d^k+1$.
For the second layer, we can consider each piece after the first layer, which is a subset of $\mathbb R^k$ and will then be further split into at most $d^k+1$ pieces. Thus, we will get at most $(d^k+1)^2$ pieces after the second layer. Continuing this argument through all $n$ layers and we have the input space $\mathbb R^k$ is split into at most 
$(d^{k}+1)^n\leq (2d)^{kn}$ pieces, where within each piece the function $G(\cdot)$ is simply a linear transformation from 
$\mathbb R^k$ to $\mathbb R^d$. 

Now, we consider any two pieces, namely $\mathcal{P}_1,~\mathcal{P}_2\subseteq\mathbb R^k$, from the aforementioned collection of pieces, and aim at bounding the following quantity:
\begin{align*}
\sup_{t_1\in\mathcal P_1,t_2\in\mathcal P_2}
\frac{\l|\frac1m\sum_{i=1}^m\varepsilon_i\dotp{a_i}{G(t_1)-G(t_2)}\r|}{\|G(t_1)-G(t_2)\|_2}.
\end{align*}
By the previous argument, we know that within $\mathcal P_1$ and $\mathcal P_2$, the function $G(\cdot)$ can simply be represented by some fixed linear maps $W_1$ and $W_2$, respectively. As a consequence, it suffices to bound
\begin{align*}
&\sup_{t_1\in\mathcal P_1,t_2\in\mathcal P_2}
\frac{\l|\frac1m\sum_{i=1}^m\varepsilon_i\dotp{a_i}{W_1t_1-W_2t_2}\r|}{\|W_1t_1-W_2t_2\|_2}\\
& \qquad \leq 
\sup_{t_1,~t_2\in\mathbb{R}^k}
\frac{\l|\frac1m\sum_{i=1}^m\varepsilon_i\dotp{a_i}{W_1t_1-W_2t_2}\r|}{\|W_1t_1-W_2t_2\|_2}\\
& \qquad \leq\sup_{t\in\mathbb{R}^{2k}}
\frac{\l|\frac1m\sum_{i=1}^m\varepsilon_i\dotp{a_i}{W_0t}\r|}{\|W_0t\|_2},
\end{align*}
where $W_0:=[W_1,~-W_2]$, and the last inequality follows from concatenating $t_1$ and $t_2$ to form a vector $t \in\mathbb{R}^{2k}$ and then expanding the set to take supremum over $t\in\mathbb{R}^{2k}$. Let $\mathcal{E}_{2k}$ be the subspace in $\mathbb R^d$ spanned by the $2k$ columns of $W_0$, then, the above supremum can be rewritten as
\[
E_m:=\sup_{b\in\mathcal{E}^{2k}\cap\mathcal{S}^{d-1}}\l|\frac1m\sum_{i=1}^m\varepsilon_i\dotp{a_i}{b}\r|.
\]
To bound the supremum, we consider a $1/2$-covering net of the set $\mathcal{E}^{2k}\cap\mathcal{S}^{d-1}$, namely, 
$\mathcal{N}(\mathcal{E}^{2k}\cap\mathcal{S}^{d-1},1/2)$. A simple volume argument shows that the cardinality
$|\mathcal{N}(\mathcal{E}^{2k}\cap\mathcal{S}^{d-1},1/2)|\leq 3^{2k}$.

By Bernstein's inequality (Lemma \ref{Bernstein}), we have for any fixed $b\in\mathcal{N}(\mathcal{E}^{2k}\cap\mathcal{S}^{d-1},1/2)$,
\[
\mathrm{Pr}\l( \l|\frac1m\sum_{i=1}^m\varepsilon_i\dotp{a_i}{b}\r|\geq \sqrt{\frac{2u'}{m}} + \frac{\|a\|_{\psi_1}u'}{m} \r)\leq 2e^{-u'}.
\]
Taking $u' = u+ckn\log(ed)$ for some $c>6$, we have with probability at least $1-2\exp(-u-ckn\log(ed))$, 
\[
\l|\frac1m\sum_{i=1}^m\varepsilon_i\dotp{a_i}{b}\r|\leq \sqrt{\frac{2(u+ckn\log(ed))}{m}} + \frac{\|a\|_{\psi_1} (u+ckn\log(ed))}{m}.
\]
Taking a union bound over all $b\in\mathcal{N}(\mathcal{E}^{2k}\cap\mathcal{S}^{d-1},1/2)$, we have with probability at least $1-2\exp(-u-ckn\log(ed))\cdot 3^{2k}\geq 1-2\exp(-u-c_1kn\log(ed))$ for some absolute constant $c_1>2$.
\begin{equation}\label{eq:covering-bound-1}
\sup_{b\in\mathcal{N}(\mathcal{E}^{2k}\cap\mathcal{S}^{d-1},1/2)}\l|\frac1m\sum_{i=1}^m\varepsilon_i\dotp{a_i}{b}\r|\leq \sqrt{\frac{2(u+ckn\log(ed))}{m}} + \frac{\|a\|_{\psi_1} (u+ckn\log(ed))}{m}.
\end{equation}
Let $P_{\mathcal N}(\cdot)$ be the projection of any point in $\mathcal{E}^{2k}\cap\mathcal{S}^{d-1}$ onto $\mathcal{N}(\mathcal{E}^{2k}\cap\mathcal{S}^{d-1},1/2)$. we have
\begin{align}
E_m \leq& \sup_{b\in\mathcal{N}(\mathcal{E}^{2k}\cap\mathcal{S}^{d-1},1/2)}\l|\frac1m\sum_{i=1}^m\varepsilon_i\dotp{a_i}{b}\r|  
+ \sup_{b\in\mathcal{E}^{2k}\cap\mathcal{S}^{d-1}}\l|\frac1m\sum_{i=1}^m\varepsilon_i\dotp{a_i}{b- P_{\mathcal N}(b)}\r|  \nonumber\\
\leq&\sup_{b\in\mathcal{N}(\mathcal{E}^{2k}\cap\mathcal{S}^{d-1},1/2)}\l|\frac1m\sum_{i=1}^m\varepsilon_i\dotp{a_i}{b}\r|  
+ \frac12\sup_{b\in\mathcal{E}^{2k}\cap\mathcal{S}^{d-1}}\l|\frac1m\sum_{i=1}^m\frac{\varepsilon_i\dotp{a_i}{b - P_{\mathcal N}(b)}}{\|b - P_{\mathcal N}(b)\|_2}\r|   \nonumber\\
\leq& \sup_{b\in\mathcal{N}(\mathcal{E}^{2k}\cap\mathcal{S}^{d-1},1/2)}\l|\frac1m\sum_{i=1}^m\varepsilon_i\dotp{a_i}{b}\r| + \frac12 E_m, \label{eq:1-step-chaining}
\end{align}
where the second inequality follows from the homogeneity of the set $\mathcal{E}^{2k}\cap\mathcal{S}^{d-1}$ under constant scaling. 
Combining \eqref{eq:covering-bound-1} and \eqref{eq:1-step-chaining} gives
\[
\sup_{b\in\mathcal{E}^{2k}\cap\mathcal{S}^{d-1}}\l|\frac1m\sum_{i=1}^m\varepsilon_i\dotp{a_i}{b}\r|\leq 
 2\sqrt{\frac{2(u+ckn\log(ed))}{m}} + \frac{2\|a\|_{\psi_1} (u+ckn\log(ed))}{m}.
\]
Taking a further union bound over at most $(2d)^{kn}$ different pair of subspaces $\mathcal P_1,~\mathcal P_2$ finishes the proof.
\end{proof}

\subsubsection{Counting the Sign Differences: A VC-Dimension Bound}
In this section, we consider all possible sign changes replacing each $\sign(\dotp{a_i}{G(x_0)} + \xi_i + \tau_i)$ by $\sign(\dotp{a_i}{G(v)} + \xi_i + \tau_i)$, where we recall $G(v)$ is a nearest point to $G(x_0)$ in $\mathcal{N}(G(\mathbb{R}^k)\cap \mathbb{B}_2^d(R),~\delta)$.

First of all, since $\tau_i\sim \text{Unif}[-\lambda,+\lambda]$, for any $\eta>0$, defining a new random variable $X_i:=\dotp{a_i}{G(v)} + \xi_i$ which is thus independent of $\tau_i$, for all $i=1,2,\cdots,m$, we have
\begin{align*}
&\mathrm{Pr}(| \dotp{a_i}{G(v)} + \xi_i+\tau_i| \leq \eta) =\mathrm{Pr}( -\eta\leq  X_i +\tau_i \leq \eta) \leq \frac{\eta}{\lambda},
\end{align*}
by computing the integral of the probability density functions of $X_i$ and $\tau_i$ in $-\eta\leq  X_i +\tau_i \leq \eta$. It is sufficient to calculate the above probability bound with only knowing the distribution of $\tau_i$. Using Chernoff bound (Lemma \ref{lem:chernoff}), one has with probability at least $1-\exp(-\eta m/3\lambda)$,
\begin{equation}\label{eq:count-1}
\sum_{i=1}^m\mathbf{1}_{\{| \dotp{a_i}{G(v)} + \xi_i+\tau_i|\geq\eta\}}\geq \l(1-\frac{2\eta}{\lambda}\r)m.
\end{equation}
Next, we prove the following lemma:
\begin{lemma}\label{lem:counting-process}
Let $\eta,\delta>0$ be chosen parameters. For any $u\geq0$ and fixed $G(v)$, the following holds with probability at least $1-2\exp(-u)$,
\[
\sup_{x_0\in\mathbb{R}^k, \|G(x_0)-G(v)\|_2\leq \delta}\sum_{i=1}^m\mathbf{1}_{\{|\dotp{a_i}{G(x_0) -G(v)}| \geq \eta\}}
\leq m\cdot \mathrm{Pr}(|\dotp{a_i}{z}|\geq \eta/\delta) + L\sqrt{(kn\log(ed)+u)m},
\] 
where $z$ is any fixed vector in $\mathbb{B}_2^d(1)$ and $L>1$ is an absolute constant.
\end{lemma}

This lemma implies that the counting process $\{\mathbf{1}_{\{|\dotp{a_i}{G(x_0) -G(v)}| \geq \eta\}}\}_{i=1}^m$ enjoys a tight sub-Gaussian uniform concentration. The proof relies on a book-keeping VC dimension argument.

\begin{proof}[Proof of Lemma \ref{lem:counting-process}]
First of all, let $T = G(\mathbb{R}^k)$, and it suffices to bound the following supremum:
\[
\sup_{t\in (T-T)\cap\mathbb{B}_2^d(\delta)}\sum_{i=1}^m\mathbf{1}_{\{|\dotp{a_i}{t}| \geq \eta\}}.
\]
Let $\mathcal{T}$ be the set of all distinctive pieces split by $G(\cdot)$. By the same argument as that of Lemma \ref{lem:ball-supremum}, the cardinality of $\mathcal{T}$ is at most 
$(d^{k}+1)^n\leq (2d)^{kn}$, and we have
\begin{align*}
&\sup_{t\in (T-T)\cap\mathbb{B}_2^d(\delta)}\sum_{i=1}^m\mathbf{1}_{\{|\dotp{a_i}{t}| \geq \eta\}}\\
& \qquad \leq \sup_{\mathcal{P}_1,~\mathcal{P}_2\in\mathcal{T}, t\in (\mathcal{P}_1-\mathcal{P}_2)\cap\mathbb{B}_2^d(\delta)}\sum_{i=1}^m\mathbf{1}_{\{|\dotp{a_i}{t}| \geq \eta\}}\\
& \qquad \leq \sup_{\mathcal{P}_1,~\mathcal{P}_2\in\mathcal{T},~ t\in \text{affine}(\mathcal P_1-\mathcal P_2)\cap\mathbb{B}_2^d(\delta)}
\sum_{i=1}^m\mathbf{1}_{\{|\dotp{a_i}{t}| \geq \eta\}}\\
& \qquad = \sup_{\mathcal{P}_1,~\mathcal{P}_2\in\mathcal{T},~ t\in \text{affine}(\mathcal P_1-\mathcal P_2)\cap\mathbb{B}_2^d(1)}
\sum_{i=1}^m\mathbf{1}_{\{|\dotp{a_i}{t}| \geq \eta/\delta\}},
\end{align*}
where $\text{affine}(\mathcal P_1-\mathcal P_2)$ denotes the affine subspace spanned by $\mathcal P_1-\mathcal P_2$, which is of dimension at most $2k$. To this point, define the set
\begin{equation}\label{eq:the-set-1}
\mathcal{C} := \{t: t\in \text{affine}(\mathcal P_1-\mathcal P_2)\cap\mathbb{B}_2^d(1), \mathcal{P}_1,~\mathcal{P}_2\in\mathcal{T}\},
\end{equation}
and define an empirical process
\[
\mathcal R(\{a_i\}_{i=1}^m,t):= \frac1m\sum_{i=1}^m\l(\mathbf{1}_{\{|\dotp{a_i}{t}| \geq \eta/\delta\}} - \expect{\mathbf{1}_{\{|\dotp{a_i}{t}| \geq \eta/\delta\}}}\r).
\]
Our goal is to bound
\[
\sup_{t\in\mathcal{C}}
|\mathcal R(\{a_i\}_{i=1}^m, t)|.
\]
By symmetrization inequality (Lemma \ref{lemma:symmetry}) it suffices to bound 
\[
\sup_{t\in\mathcal{C}}
\l|\frac1m\sum_{i=1}^m\varepsilon_i\mathbf{1}_{\{|\dotp{a_i}{t}| \geq \eta/\delta\}}\r|,
\]
where $\{\varepsilon\}_{i=1}^m$ are i.i.d. Rademacher random variables.
Define the set of indicator functions:
$$\mathcal{F}:= \{\mathbf{1}_{\{|\dotp{\cdot}{t}| \geq \eta/\delta\}}:
~t\in\mathcal{C}\}.$$
By Hoeffding's inequality, the stochastic process $m^{-1/2}\sum_{i=1}^m\varepsilon_i\mathbf{1}_{\{|\dotp{a_i}{t}| \geq \eta/\delta\}}$ parametrized by $\mathcal{F}$ when fixing $\{a_i\}_{i=1}^m$ is a sub-Gaussian process with respect to the empirical $L_2$ metric:
\[
\|f-g\|_{L_2(\mu_m)}:= \sqrt{\frac1m\sum_{i=1}^m(f(a_i)-g(a_i))^2},~\forall f,g\in\mathcal{F}.
\]
By Lemma \ref{lem:entropy-integral}, one can easily derive the following bound:
\begin{equation}\label{eq:dudley-bound}
\expect{\sup_{t\in\mathcal{C}}
|\mathcal R(\{a_i\}_{i=1}^m,t)|}
\leq \frac{C_0}{\sqrt{m}}\int_0^2\sqrt{\log|\mathcal N(\varepsilon, ~\mathcal{F},~\|\cdot\|_{L_2(\mu_m)})|}d\varepsilon,
\end{equation}
where $\mathcal N(\varepsilon, ~\mathcal{F},~\|\cdot\|_{L_2(\mu_m)})$ is the $\varepsilon$-covering net of $\mathcal{F}$ under the empirical $L_2$-metric. By Haussler's inequality (Theorem 2.6.4 of \citet{wellner2013weak}), 
\[
|\mathcal N(\varepsilon, ~\mathcal{F},~\|\cdot\|_{L_2(\mu_m)})|
\leq C_1 V(\mathcal{F})(4e)^{V(\mathcal F)}\l(\frac{1}{\varepsilon}\r)^{2V(\mathcal{F})},
\]
where $V(\mathcal{F})$ is the VC dimension of the class $\mathcal{F}$ and $C_1$ is an absolute constant. To compute $V(\mathcal{F})$, note first that for any fixed $\mathcal P_1,~\mathcal P_2\in\mathcal{T}$ and any fixed constant $c$, the VC dimension of the class of half-spaces defined as
\[
\mathcal{H}' := \{\dotp{\cdot}{t}\geq c:~t\in\text{affine}(\mathcal{P}_1-\mathcal{P}_2)\}
\]  
is bounded by $2k$. Thus, for any $p$ points on $\mathbb{R}^k$ and the number of different subsets of these points picked by $\mathcal{H}'$ is bounded by $(p+1)^{2k}$.
Next, note that any element in the class 
\[
\mathcal{H} := \{|\dotp{\cdot}{t}|\geq c:~t\in\text{affine}(\mathcal{P}_1-\mathcal{P}_2)\}
\]
is the intersection of two halfspaces in $\mathcal{H}'$. Thus, the number of different subsets of $p$ points picked by $\mathcal{H}$ is bounded by 
\[
{(p+1)^{2k}\choose 2}\leq e^2(p+1)^{4k}/4\leq 2(p+1)^{4k}.
\]
Taking into account that the class $\mathcal{F}$ is the union of at most $(2d)^{2kn}$ different classes of the form
\[
 \{\mathbf{1}_{\{|\dotp{\cdot}{t}| \geq \eta/\delta\}}:~t\in\text{affine}(\mathcal{P}_1-\mathcal{P}_2)\},
\]
we arrive at the conclusion that the number of distinctive mappings in $\mathcal{F}$ from any $p$ points in $\mathbb{R}^k$ to $\{0,1\}^p$
is bounded by $2d^{2kn}(p+1)^{4k}$. To get the VC dimension of $\mathcal{F}$, we try to find the smallest $p$ such that 
\[
2d^{2kn}(p+1)^{4k}<2^p.
\]
A sufficient condition is to have $2kn\log_2(d) + 4k\log_2(p+1) + 1 < p$, which holds when $p>c_0kn\log(ed)-1$ for some absolute constant $c_0$ large enough. Thus, $V(\mathcal{F})\leq c_0kn\log(ed)$. Thus, it follows
\begin{align*}
\log|\mathcal N(\varepsilon, ~\mathcal{F},~\|\cdot\|_{L_2(\mu_m)})|
&  \leq  \log C_1 + \log V(\mathcal{F}) + V(\mathcal{F})\log(4e) +  2V(\mathcal F)\log(1/\varepsilon)\\
&   \leq  c_1kn\log(ed)(\log(1/\varepsilon)+1),
\end{align*}
for some absolute constant $c_1>0$.  Substituting this bound into \eqref{eq:dudley-bound}, and we obtain 
\[
\expect{\sup_{t\in\mathcal{C}}
|\mathcal R(\{a_i\}_{i=1}^m, t)|}\leq c_2\sqrt{\frac{kn\log(ed)}{m}},
\]
for some absolute constant $c_2$. Finally, by bounded difference inequality, we obtain with probability at least $1-2e^{-u}$, 
\begin{equation*}
\sup_{t\in\mathcal{C}}
|\mathcal R(\{a_i\}_{i=1}^m, t)| \leq \expect{\sup_{t\in\mathcal{C}}
|\mathcal R(\{a_i\}_{i=1}^m, t)|} + \sqrt{\frac um}\leq L\sqrt{\frac{kn\log(ed)+u}{m}},
\end{equation*}
finishing the proof.
\end{proof}

Combining Lemma \ref{lem:counting-process} and \eqref{eq:count-1} we have the following bound on the number of sign differences:
\begin{lemma}\label{lem:main-count}
Let $u>0$ be any constant. Suppose $
m\geq c_2\lambda^2(kn\log(ed) + k\log(2R) + u)/\varepsilon^2
$ with $\varepsilon < 1$ for some absolute constant $c_2$ large enough and $\lambda\geq1$. 
Define the following parameters
\begin{align}
&\delta := \frac{\eta}{\|a\|_{\psi_1}}\log(c_1\lambda/\eta),\label{eq:delta-1}\\
&\eta := (\lambda+ \|a\|_{\psi_1}) L\sqrt{\frac{kn\log(ed)  + u' }{m}}, \label{eq:eta-1}
\end{align} 
and $u'>0$ satisfying
\begin{equation}\label{eq:u-prime-bound}
u' = u+kn\log(ed) + k\log(2R) + Ck\log\l(\frac{m}{k\log(ed) +u'}\r).
\end{equation}
We have with probability at least $1-\exp(-c_0 u)-2\exp(-u)$, 
\begin{align*}
\sup~
\frac{1}{m}\sum_{i=1}^m\mathbf{1}_{\{ \sign(\dotp{a_i}{G(v)} + \xi_i+\tau_i) \neq \sign(\dotp{a_i}{G(x_0)} + \xi_i+\tau_i)\}}
\leq 
\frac{4\eta}{\lambda},
\end{align*}
where the supremum is taken over $x_0\in\mathbb{R}^k, \|G(x_0)-G(v)\|_2\leq \delta, G(v)\in\mathcal{N}(G(\mathbb{R}^k)\cap \mathbb{B}_2^d(R),~\delta)$ and $c_0\geq 1, c_1, c_2, C, L>0$ are absolute constants.
\end{lemma}

\begin{proof}[Proof of Lemma \ref{lem:main-count}]
We compute $\mathrm{Pr}(|\dotp{a_i}{z}|\geq \eta/\delta)$. By the fact that $\dotp{a_i}{z}$ is a sub-exponential random variable, 
\[
\mathrm{Pr}(|\dotp{a_i}{z}|\geq \eta/\delta)\leq c_1\exp\l(-\frac{\eta}{\delta\|a\|_{\psi_1}}\r),
\]
where $c_1>0$ is an absolute constant.
We choose $\delta$ according to \eqref{eq:delta-1}, 
which implies
\[
\mathrm{Pr}(|\dotp{a_i}{z}|\geq \eta/\delta)\leq \frac{\eta}{\lambda}.
\]
From Lemma \ref{lem:counting-process}, we readily obtain with probability at least $1-2\exp(-u')$,
\begin{equation}\label{eq:sup-1}
\sup_{x_0\in\mathbb{R}^k, \|G(x_0)-G(v)\|_2\leq \delta}\sum_{i=1}^m\mathbf{1}_{\{|\dotp{a_i}{G(x_0) -G(v)}| \geq \eta\}}\leq 
\l(\frac{\eta}{\lambda} + L\sqrt{\frac{kn\log(ed)  + u'}{m}}\r)m.
\end{equation}
We will then take a further supremum over all $G(v)\in \mathcal{N}(G(\mathbb{R}^k)\cap \mathbb{B}_2^d(R),~\delta)$.
Note that by a simple volume argument, $\mathcal{N}(G(\mathbb{R}^k)\cap \mathbb{B}_2^d(R),~\delta)$ satisfies
\[
\log|\mathcal{N}(G(\mathbb{R}^k)\cap \mathbb{B}_2^d(R),~\delta)|\leq kn\log(ed) + 
k\log(2R/\delta).
\]
Choose $\eta$ according to \eqref{eq:eta-1}. Then, 
By the aforementioned choices of $\eta$ and $\delta$ in \eqref{eq:eta-1} and \eqref{eq:delta-1}, we obtain
\[
\log(1/\delta)\leq C\log\l(\frac{m}{k\log(ed)  + u'}\r),
\]
where $C$ is an absolute constant. Thus,
\begin{equation}\label{eq:net-size}
\log|\mathcal{N}(G(\mathbb{R}^k)\cap \mathbb{B}_2^d(R),~\delta)|\leq
kn\log(ed) + k\log(2R) + Ck\log\l(\frac{m}{k\log(ed) +u'}\r).
\end{equation}
Finally, for any $u>0$, take $u'$ so that it satisfies \eqref{eq:u-prime-bound}.
By \eqref{eq:sup-1}, we obtain that,   with probability at least 
$$1-2\exp\l(-u - kn\log(ed) - k\log(2R) - Ck\log\l(\frac{m}{k\log(ed) +u'}\r)\r),$$
the following holds
\begin{align*}
\sup_{x_0\in\mathbb{R}^k, \|G(x_0)-G(v)\|_2\leq \delta}\sum_{i=1}^m\mathbf{1}_{\{|\dotp{a_i}{G(x_0) -G(v)}| \geq \eta\}}
\leq 
\l(\frac{\eta}{\lambda} + L\sqrt{\frac{kn\log(ed) + u'}{m}}\r)m.
\end{align*}
Taking a union bound over all $G(v)\in \mathcal{N}(G(\mathbb{R}^k)\cap \mathbb{B}_2^d(R),~\delta)$, we get with probability at least 
$1-2\exp(-u)$,
\begin{align*}
\sup_{x_0\in\mathbb{R}^k, \|G(x_0)-G(v)\|_2\leq \delta, G(v)\in\mathcal{N}(G(\mathbb{R}^k)\cap \mathbb{B}_2^d(R),~\delta)}~~\sum_{i=1}^m\mathbf{1}_{\{|\dotp{a_i}{G(x_0) -G(v)}| \geq \eta\}} \leq 
\l(\frac{\eta}{\lambda} + L\sqrt{\frac{kn\log(ed) + u'}{m}}\r)m.
\end{align*}
Note that by definition of $\eta$ in \eqref{eq:eta-1}, $L\sqrt{(kn\log(ed) + u')/m}\leq \eta/\lambda$, and this readily implies with probability at least $1-2\exp(-u)$,
\begin{align*}
\sup_{x_0\in\mathbb{R}^k, \|G(x_0)-G(v)\|_2\leq \delta, G(v)\in\mathcal{N}(G(\mathbb{R}^k)\cap \mathbb{B}_2^d(R),~\delta)}~~\sum_{i=1}^m\mathbf{1}_{\{|\dotp{a_i}{G(x_0) -G(v)}| \geq \eta\}}
\leq 
\frac{2\eta}{\lambda}m,
\end{align*}
or equivalently
\begin{align}\label{eq:small-count}
\inf_{x_0\in\mathbb{R}^k, \|G(x_0)-G(v)\|_2\leq \delta, G(v)\in\mathcal{N}(G(\mathbb{R}^k)\cap \mathbb{B}_2^d(R),~\delta)}~~\sum_{i=1}^m\mathbf{1}_{\{|\dotp{a_i}{G(x_0) -G(v)}| < \eta\}}
\geq 
\l(1-\frac{2\eta}{\lambda}\r)m.
\end{align}
Moreover, taking a union bound over all $G(v)\in\mathcal{N}(G(\mathbb{R}^k)\cap \mathbb{B}_2^d(R),~\delta)$ in \eqref{eq:count-1}, we have with probability at least 
\begin{align*}
&1-\exp(\log|\mathcal{N}(G(\mathbb{R}^k)\cap \mathbb{B}_2^d(R),~\delta)|-\eta m/3\lambda)\\
&\qquad \geq 1-\exp\l(kn\log(ed) + k\log(2R) + Ck\log\l(\frac{m}{k\log(ed) +u'}\r) -\frac{\eta m}{3\lambda} \r),
\end{align*}
one has
\begin{equation}\label{eq:major-count}
\inf_{G(v)\in\mathcal{N}(G(\mathbb{R}^k)\cap \mathbb{B}_2^d(R),~\delta)}~~\sum_{i=1}^m\mathbf{1}_{\{| \dotp{a_i}{G(v)} + \xi_i+\tau_i|\geq\eta\}}\geq \Big(1-\frac{2\eta}{\lambda} \Big)m.
\end{equation}
Note that by assumption, we have 
$m\geq c_2\lambda^2(kn\log(ed) + k\log(2R) + u)/\varepsilon^2$ for some $\varepsilon<1$ and
some absolute constant $c_2$ large enough. Thus, it follows
\begin{align*}
\frac{\eta m}{3\lambda} \geq& \frac L3\sqrt{(kn\log(ed)+u')m}\\
\geq& \frac L3\sqrt{\l(u+kn\log(ed) + k\log(2R) + Ck\log\frac{m}{kn\log(ed) + u'}\r)m}\\
\geq& \frac{L}{3\sqrt{2}}\l( \sqrt{c_2} \big(u+kn\log(ed) + k\log(2R)\big)  + \sqrt{Ckm\log\frac{m}{kn\log(ed) + u'}}\r)\\
\geq& c_0\l( u+kn\log(ed) + k\log(2R)  + k\log\frac{m}{k\log(ed) +u'}\r),
\end{align*}
where $c_0$ is an absolute constant with $c_0 = L\sqrt{c_2} \min\{\sqrt{C}, 1\}/ (3\sqrt{2})$, and
the last inequality follows from the assumption that $m\geq \sqrt{c_2 km}\geq\sqrt{c_2 k\log m}$ for any $\varepsilon < 1$ and $m > 1$. When $c_2$ is large enough such that $c_0\geq\max\{C, 1\}$, we have
\begin{align*}
\frac{\eta m}{3\lambda} - \l( kn\log(ed) + k\log(2R) + Ck\log\l(\frac{m}{k\log(ed) +u'}\r) \r) \geq c_0 u.
\end{align*}
Then, we have \eqref{eq:major-count} holds with probability at least $1-\exp(-c_0u)$. 

Furthermore, $|\dotp{a_i}{G(x_0) -G(v)}| < \eta$ and $|\dotp{a_i}{G(v)} + \xi_i+\tau_i|\geq\eta$ with $\eta > 0$ will lead to $\sign(\dotp{a_i}{G(v)} + \xi_i+\tau_i) = \sign(\dotp{a_i}{G(x_0)} + \xi_i+\tau_i)$. The inequality \eqref{eq:small-count} implies that with high probability, there are at least $(1-2\eta/\lambda)m$ vectors $a_i$ (where $i\in \{1,\ldots,m\}$) satisfying $|\dotp{a_i}{G(x_0) -G(v)}| < \eta$. In addition, the inequality \eqref{eq:major-count} implies that with high probability, there are at least $(1-2\eta/\lambda)m$ vectors $a_i$ satisfying $|\dotp{a_i}{G(v)} + \xi_i+\tau_i|\geq\eta$. Then, we know that there must be at least $(1-4\eta/\lambda)m$ vectors $a_i$ satisfying both $|\dotp{a_i}{G(x_0) -G(v)}| < \eta$ and $|\dotp{a_i}{G(v)} + \xi_i+\tau_i|\geq\eta$ with high probability. 

Thus, combining \eqref{eq:small-count} and \eqref{eq:major-count}, under the conditions of this lemma, we have with probability at least $1-\exp(-c_0u)-2\exp(-u)$,
\begin{align*}
\inf ~
\sum_{i=1}^m \mathbf{1}_{\{ \sign(\dotp{a_i}{G(v)} + \xi_i+\tau_i) = \sign(\dotp{a_i}{G(x_0)} + \xi_i+\tau_i)\}}
\geq 
\l(1-\frac{4\eta}{\lambda}\r)m,
\end{align*}
which is equivalent to
\begin{align*}
\sup ~ \frac{1}{m}
\sum_{i=1}^m \mathbf{1}_{\{ \sign(\dotp{a_i}{G(v)} + \xi_i+\tau_i) \neq \sign(\dotp{a_i}{G(x_0)} + \xi_i+\tau_i)\}}
\leq
\frac{4\eta}{\lambda} ,
\end{align*}
where the infimum and supremum are taken over $x_0\in\mathbb{R}^k, \|G(x_0)-G(v)\|_2\leq \delta, G(v)\in\mathcal{N}(G(\mathbb{R}^k)\cap \mathbb{B}_2^d(R),~\delta)$. This completes the proof.
\end{proof}

\subsubsection{Putting Bounds Together: Proof of Lemma \ref{lem:main-bound-on-sup}}\label{sec:proof-main-lemma}

\begin{lemma}[Lemma \ref{lem:main-bound-on-sup}]\label{lem:main-bound}
Suppose Assumption \ref{as:moments} holds and 
\begin{equation}\label{eq:m-bound-supp}
m\geq c_2\lambda^2\log^2(\lambda m)(kn\log(ed) + k\log(2R) + k\log m+ u)/\varepsilon^2,
\end{equation}
for some absolute constant $c_2$ large enough, then, 
with probability at least $1-c_1\exp(-u)$,
\[
\sup_{x_0\in\mathbb{R}^k,~\|G(x_0)\|_2\leq R,~x\in\mathbb{R}^k}
\frac{\l|\frac1m\sum_{i=1}^m\varepsilon_i\sign(\dotp{a_i}{G(x_0)} + \xi_i+\tau_i)\dotp{a_i}{G(x) - G(x_0)}\r|}{\|G(x) - G(x_0)\|_2}
\leq \frac{\varepsilon}{16\lambda},
\]
where $\{\varepsilon_i\}_{i=1}^m$ are i.i.d. Rademacher random variables and $c>0$ is an absolute constant.
\end{lemma}
\begin{proof}[Proof of Lemma \ref{lem:main-bound-on-sup}]
Let $\mathcal I$ be the set of indices such that 
$\sign(\dotp{a_i}{G(v)} + \xi_i+\tau_i) \neq \sign(\dotp{a_i}{G(x_0)} + \xi_i+\tau_i)$. By Lemma \ref{lem:main-count}, we know that 
$|\mathcal I|\leq 4m\eta/\lambda$. Then, we have with probability at least $1-\exp(-c_0u)-2\exp(-u)$,
\begin{align*}
&\sup_{x_0\in\mathbb{R}^k,~\|G(x_0)\|_2\leq R,~x\in\mathbb{R}^k}
\frac{\l|\frac1m\sum_{i=1}^m\varepsilon_i\sign(\dotp{a_i}{G(x_0)} + \xi_i+\tau_i)\dotp{a_i}{G(x) - G(x_0)}\r|}{\|G(x) - G(x_0)\|_2}\\
& \qquad \leq 
\sup_{x\in\mathbb{R}^k,~x_0\in\mathbb{R}^k,~\|G(x_0)-G(v)\|_2\leq \delta, ~G(v)\in\mathcal{N}(G(\mathbb{R}^k)\cap \mathbb{B}_2^d(R),~\delta)} \frac{\l|\frac1m\sum_{i=1}^m\varepsilon_i y_i^v \dotp{a_i}{G(x) - G(x_0)}\r|}{\|G(x) - G(x_0)\|_2}
\\
&\qquad \quad + 
\sup_{x\in\mathbb{R}^k, ~x_0\in\mathbb{R}^k, ~\|G(x_0)-G(v)\|_2\leq \delta, ~G(v)\in\mathcal{N}(G(\mathbb{R}^k)\cap \mathbb{B}_2^d(R),~\delta)}
\frac{\l|\frac1m\sum_{i=1}^m\varepsilon_i (y_i - y_i^v) \dotp{a_i}{G(x) - G(x_0)}\r|}{\|G(x) - G(x_0)\|_2}\\
& \qquad \leq \underbrace{
\sup_{x\in\mathbb{R}^k,~x_0\in\mathbb{R}^k,~G(v)\in\mathcal{N}(G(\mathbb{R}^k)\cap \mathbb{B}_2^d(R),~\delta)} \frac{\l|\frac1m\sum_{i=1}^m\varepsilon_i\sign(\dotp{a_i}{G(v)} + \xi_i + \tau_i)\dotp{a_i}{G(x) - G(x_0)}\r|}{\|G(x) - G(x_0)\|_2}
}_{\text{(I)}}\\
&\qquad \quad +\underbrace{
\sup_{x\in\mathbb{R}^k, ~x_0\in\mathbb{R}^k}~
\max_{|\mathcal I|\leq 4m\eta/\lambda}
\frac{2}{m}\sum_{i\in\mathcal I}\frac{|\dotp{a_i}{G(x) - G(x_0}|}{\|G(x) - G(x_0)\|_2}
}_{\text{(II)}},
\end{align*}
where, for simplicity, we let $y_i^v := \sign(\dotp{a_i}{G(v)} + \xi_i+\tau_i)$ be the sign function associated with $G(v)$ in the net in the first inequality, and the second inequality is by Lemma \ref{lem:main-count} for term (II) and dropping the constraint $\|G(x_0)-G(v)\|_2\leq \delta$ for term (I).

For the rest of the proof, we will bound (I) and (II) respectively.  
To bound (I), take $u$ in  Lemma \ref{lem:ball-supremum} to be
$
kn\log(ed) + k\log(2R) + Ck\log m + u
$, we have with probability at 
 $1- 2\exp(-c' kn\log(ed) - k\log(2R) - Ck\log m-u)$, for a fixed $G(v)$, any $x\in\mathbb{R}^k,~x_0\in\mathbb{R}^k$,
\begin{align*}
&\frac{\l|\frac1m\sum_{i=1}^m\varepsilon_i\sign(\dotp{a_i}{G(v)} + \xi_i + \tau_i)\dotp{a_i}{G(x) - G(x_0)}\r|}{\|G(x) - G(x_0)\|_2}\\
& \leq \sqrt{\frac{8(ckn\log(ed) + k\log(2R) + Ck\log m + u)}{m}} + \frac{2\|a\|_{\psi_1}(ckn\log(ed) + k\log(2R) + Ck\log m + u)}{m},
\end{align*}
where $c,c',C>0$ are absolute constants. Take a further union bound over all $G(v)\in\mathcal{N}(G(\mathbb{R}^k)\cap \mathbb{B}_2^d(R),~\delta)$ with the net size satisfying \eqref{eq:net-size}, we have with probability at least $1-2\exp(-u)$,
\begin{align}
\begin{aligned}\label{eq:bound-i}
\text{(I)}&\leq \sqrt{\frac{8(ckn\log(ed) + k\log(2R) + Ck\log m + u)}{m}} \\
&\qquad + \frac{2\|a\|_{\psi_1}(ckn\log(ed) + k\log(2R) + Ck\log m + u)}{m}.
\end{aligned}
\end{align}
Next, we will bound  the term (II). Let $t = (G(x)-G(x_0))/\|G(x)-G(x_0)\|_2$ and it is enough to bound
\begin{align}
\sup_{x_0\in\mathbb{R}^k, x_0\in\mathbb{R}^k}~~
\max_{|\mathcal I|\leq 4m\eta/\lambda}\frac1m\sum_{i\in\mathcal I}|\dotp{a_i}{t}| - \expect{|\dotp{a_i}{t}|} + \expect{|\dotp{a_i}{t}|}. \label{eq:sep-bound-overall} 
\end{align}
It is obvious that $|\dotp{a_i}{t}| - \expect{|\dotp{a_i}{t}|}$ is also a sub-exponential random variable with sub-exponential norm bounded by $2\|a\|_{\psi_1}$, and $\expect{|\dotp{a_i}{t}|}\leq 1$. Thus, by Bernstein's inequality,
\[
\frac{1}{|\mathcal I|}\sum_{i\in\mathcal I}|\dotp{a_i}{t}| - \expect{|\dotp{a_i}{t}|}\leq \frac{2\sqrt{u_2}}{\sqrt{|\mathcal I|}} 
+ \frac{2\|a\|_{\psi_1}u_2}{|\mathcal I|},
\]
with probability at least $1-2\exp(-u_2)$. 
Thus,
\[
\frac1m\sum_{i\in\mathcal I}|\dotp{a_i}{t}| - \expect{|\dotp{a_i}{t}|} \leq 
\frac1m (2\sqrt{u_2|\mathcal I|}  + 2\|a\|_{\psi_1}u_2).
\]
Here we take 
$$
u_2 = C_1\log
\l(\frac{\lambda m}{kn\log(ed)+u'}\r)\Big(u+2kn\log(ed) + k\log(2R) + k\log\frac{m}{kn\log(ed)+u'}\Big)^{1/2}\sqrt{m},
$$
where $C_1$ is an absolute constant large enough and $u'$ satisfies \eqref{eq:u-prime-bound}.
Using the fact that 
\[
|\mathcal I|\leq \frac{4m\eta}{\lambda}\leq 2L
\Big(u+2kn\log(ed) + k\log(2R) + k\log\frac{m}{kn\log(ed)+u'}\Big)^{1/2}\sqrt{m},
\]
we have with probability at least 
\[
1-2\exp\l(-C_1\log
\l(\frac{\lambda m}{kn\log(ed)+u'}\r)\Big(u+2kn\log(ed) + k\log(2R) + k\log\frac{m}{kn\log(ed)+u'}\Big)^{1/2}\sqrt{m} \r),
\]
the following holds
\begin{align}
\begin{aligned}\label{eq:inter-bound}
&\l|\frac1m\sum_{i\in\mathcal I}|\dotp{a_i}{t}| - \expect{|\dotp{a_i}{t}|}\r|  \\
&\qquad \leq C_1\|a\|_{\psi_1}\log\Big(\frac{\lambda m}{kn\log(ed)+u'}\Big)\sqrt{\frac{u+2kn\log(ed) + k\log(2R) + k\log\frac{m}{kn\log(ed)+u'}}{m}}.
\end{aligned}
\end{align}
To bound the maximum over $|\mathcal I|\leq 4m\eta/\lambda$, we take a union bound over all ${m\choose 4\eta m/\lambda}$ possibilities, where
\[
{m\choose 4\eta m/\lambda}\leq \l(\frac{em}{4\eta m/\lambda}\r)^{4\eta m/\lambda}
 = \l(\frac\lambda \eta\r)^{4\eta m/\lambda}.
\]
Thus, it follows from the definition of $\eta$ in terms of $\lambda$ in Lemma \ref{lem:main-count},
\begin{align*}
\log{m\choose 4\eta m/\lambda} &\leq \frac{4\eta m}{\lambda}\log\frac{\lambda}{\eta} \\
&\leq L\Big(u+2kn\log(ed) + k\log(2R) + \log\frac{m}{kn\log(ed)+u'}\Big)^{1/2}\cdot\sqrt{m}\log\Big(\frac{\lambda m}{kn\log(ed)+u'}\Big),
\end{align*}
and when $C_1>L$, the union bound gives, with probability at least 
\[
1-2\exp\l(-C_2\log
\l(\frac{\lambda m}{kn\log(ed)+u'}\r)\Big(u+2kn\log(ed) + k\log(2R) + \log\frac{m}{kn\log(ed)+u'}\Big)^{1/2}\sqrt{m} \r),
\]
the quantity 
\[
\max_{|\mathcal I|\leq 4\eta/\lambda}\l|\frac1m\sum_{i\in\mathcal I}|\dotp{a_i}{t}| - \expect{|\dotp{a_i}{t}|}\r|
\]
is also bounded by the right hand side of \eqref{eq:inter-bound} with a possibly different constant $C_1$, where $t = (G(x)-G(x_0))/\|G(x)-G(x_0)\|_2$. Now, 
using the same trick as that of Lemma \ref{lem:ball-supremum}, we obtain 
\[
\sup_{x\in\mathbb{R}^k, x_0\in\mathbb{R}^k}
\max_{|\mathcal I|\leq 4\eta/\lambda}\l|\frac1m\sum_{i\in\mathcal I}|\dotp{a_i}{t}| - \expect{|\dotp{a_i}{t}|}\r|
\]
is bounded by the right hand side of \eqref{eq:inter-bound} with a possibly different constant $C_1$ and 
with probability
\begin{align*}
1-2\cdot 3^{2k}(2d)^{kn}
\exp\l(-C_2\log
\l(\frac{\lambda m}{kn\log(ed)+u'}\r)\Big(u+2kn\log(ed) + k\log(2R) + \log\frac{m}{kn\log(ed)+u'}\Big)^{1/2}\sqrt{m} \r),
\end{align*}
where $C_2$ is another absolute constant.
Note that by assumption in Theorem \ref{thm:main-1}, 
\begin{align*}
m\geq c_2\lambda^2\log^2(\lambda m)(kn\log(ed) + k\log(2R) + k\log m+ u)/\varepsilon^2,
\end{align*}
for some absolute constant $c_2$ large enough. 

On the other hand, for the extra expectation term in \eqref{eq:sep-bound-overall}, due to $\expect{|\dotp{a_i}{t}|}\leq 1$ with $t = (G(x)-G(x_0))/\|G(x)-G(x_0)\|_2$, we have
\begin{align*}
&\sup_{x_0\in\mathbb{R}^k, x_0\in\mathbb{R}^k}
\max_{|\mathcal I| \leq 4m\eta/\lambda}\frac1m\sum_{i\in\mathcal I}   \expect{|\dotp{a_i}{t}|} \leq \max_{|\mathcal I|\leq 4m\eta/\lambda} \frac{|\mathcal I|}{m}\\
&\qquad\qquad \leq 2L
\sqrt{\frac{u+2kn\log(ed) + k\log(2R) + k\log\frac{m}{kn\log(ed)+u'}}{m}}. 
\end{align*}
Combining the above results, we have
\[
\text{(II)}
\leq
C_3\log\Big(\frac{\lambda m}{kn\log(ed)+u'}\Big)\sqrt{\frac{u+2kn\log(ed) + k\log(2R) + k\log\frac{m}{kn\log(ed)+u'}}{m}},
\]
with probability at least $1- c_3\exp(-u)$, where $c_3\geq1$ is an absolute constant and $C_3$ is a constant depending on $\|a\|_{\psi_1}$, $C_1$, and $L$. Combining this bound with \eqref{eq:bound-i} and using \eqref{eq:m-bound-supp}, and letting $c_2$ be sufficiently large such that it satisfies $c_2\geq 256[(2\|a\|_{\psi_1}+3) (c+C)+C_3]^2$,  we obtain with probability $1-c_3\exp(-u)-\exp(-c_0 u)-2\exp(-u) \geq 1-c_1 \exp(-u)$ for an absolute constant $c_1 > 0$ (recalling that $c_0 \geq 1$ as shown in Lemma \ref{lem:main-count}),
\[
\sup_{x_0\in\mathbb{R}^k,~\|G(x_0)\|_2\leq R,~x\in\mathbb{R}^k}
\frac{\l|\frac1m\sum_{i=1}^m\varepsilon_i\sign(\dotp{a_i}{G(x_0)} + \xi_i+\tau_i)\dotp{a_i}{G(x) - G(x_0)}\r|}{\|G(x) - G(x_0)\|_2}
\leq \frac{\varepsilon}{16\lambda}.
\]
This finishes the proof.
\end{proof}

\subsection{Useful Probability Bounds for Proving Theorem \ref{thm:main-1}}
We recall the following well-known concentration inequality.
\begin{lemma}[Bernstein's inequality]
\label{Bernstein}
Let $X_1,\cdots,X_m$ be a sequence of independent centered random variables. 
Assume that there exist positive constants $f $ and $D$ such that for all integers $p\geq 2$
\[
\frac1m\sum_{i=1}^m\expect{|X_i|^p}\leq\frac{p!}{2}f ^2D^{p-2},
\]
then
\[
 \mathrm{Pr}\left(\left|\frac1m\sum_{i=1}^m X_i\right|\geq\frac{f }{\sqrt{m}}\sqrt{2u}+\frac{D}{m}u\right)
\leq2\exp(-u).
\]
In particular, if $X_1,\cdots,X_m$ are all sub-exponential random variables, then $f $ and $D$ can be chosen as 
$f =\frac{1}{m}\sum_{i=1}^m\|X_i\|_{\psi_1}$ and $D=\max\limits_{i=1\ldots m}\|X_i\|_{\psi_1}$. 
\end{lemma}

The following version of symmetrization inequality can be found, for example, in \citet{wellner2013weak}.
\begin{lemma}[Symmetrization inequality]\label{lemma:symmetry}
Let $\l\{Z_t(i)\r\}_{i=1}^m$ be i.i.d. copies of a mean 0 stochastic process $\l\{Z_t:~t\in T\r\}$. For every $1\leq i \leq m$, let $g_t(i):~T\rightarrow\mathbb R$ be an arbitrary function. Let $\{\varepsilon_i\}_{i=1}^m$ be a sequence of independent Rademacher random variables. Then, for every $x>0$, 
\begin{equation*}
\l(1-\frac{4m}{x^2}\sup_{t\in T} var(Z_t)\r)\cdot \mathrm{Pr}\l(\sup_{t\in T}\l| \sum_{i=1}^mZ_t(i)\r| > x\r)
\leq 2 \mathrm{Pr}\l( \sup_{t\in T}\l| \sum_{i=1}^m\varepsilon_i (Z_t(i)-g_t(i))\r| > \frac{x}{4} \r),
\end{equation*}
where $var(Z_t) = \expect{(Z_t - \expect{Z_t})^2}$.
\end{lemma}

The following classical bound can be found, for example in Proposition 2.4 of \citet{angluin1979fast}. 
\begin{lemma}[Chernoff bound]
\label{lem:chernoff}
Let $X_1,\ldots,X_n$ be a sequence of i.i.d. copies of $X$ such that $\mathrm{Pr}(X=1)=1-\mathrm{Pr}(X=0)=p\in (0,1)$, and define 
$S_n:=\sum_{i=1}^n X_i$.  
Then 
\[
\mathrm{Pr}\Big(\frac{S_n}{n} \geq (1+\tau)p \Big)\leq \inf_{\theta>0}\Big[ e^{-\theta np(1+\tau)}\mathbb{E} e^{\theta S_n} \Big]\leq
\begin{cases}
e^{-\frac{\tau^2 np}{2+\tau}}, & \tau>1, \\
e^{-\frac{\tau^2 np}{3}}, & 0<\tau\leq 1.
\end{cases}
\]
\end{lemma}

The following bound is the well-known Dudley's entropy estimate which can be found, for example, in Corollary 2.2.8 of \citet{wellner2013weak}.
\begin{lemma}[Dudley's entropy bound]\label{lem:entropy-integral}
Let $(T,d)$ be an arbitrary semi-metric space, and let $\{X_t,~t\in T\}$ be a separable sub-Gaussian stochastic process with 
\footnote{For a sub-Gaussian random variable $X$, the $\psi_2$-norm is defined as $\sup_{p\geq1}p^{-1/2}\|X\|_{L_p}$.}
\[
\|X_s-X_t\|_{\psi_2}\leq Cd(s,t),~\forall s,t\in T,
\]
for some constant $C>0$. Then, for every $r>0$,
\[
\expect{\sup_{d(s,t)\leq r}|X_s-X_t|}\leq C_0\int_0^r\sqrt{\log\mathcal{N}(\varepsilon,d)}d\varepsilon, 
\]
where $\mathcal{N}(\varepsilon,d)$ is the $\varepsilon$ covering number of the set $T$ and $C_0$ is an absolute constant.
\end{lemma}

\section{Proof of Theorem \ref{thm:lower-bound}} \label{sec:detailed_proof_lower_bound}

We provide detailed proofs of Proposition \ref{prop:transform-main} and Lemma \ref{lem:linear-model-lower-bound-final} in this section. As shown in Section \S\ref{sec:proof_lower_bound}, Theorem \ref{thm:lower-bound} can be proved immediately following Proposition \ref{prop:transform-main} and Lemma \ref{lem:linear-model-lower-bound-final}.

\begin{definition}
A vector $ v\in\mathbb R^d$ is $k$-group sparse if, when dividing $ v$ into $k$ blocks of sub-vectors of size $d/k$,\footnote{We assume WLOG that $d/k$ is an integer.} each block has exactly one non-zero entry. 
\end{definition}

\begin{proposition}[Proposition \ref{prop:transform-main}]\label{prop:transform}
Any nonnegative $k$-group sparse vector in $\mathbb B_2^d(1)$ can be generated by a ReLU network of the form \eqref{eq:relu} with a $k+1$ dimensional input and and depth $n=3$.
\end{proposition}
\begin{proof}[Proof of Proposition \ref{prop:transform-main}]
Consider an $k+1$ dimensional input of a network. The idea is to map each of the first $k$ entries of the input into a block in $\mathbb{R}^d$ of length $d/k$, respectively, and use one another input entry to construct proper offsets.

We first construct a single-hidden-layer ReLU network (i.e. $n=2$) \emph{with offsets} and $k$ dimensional input $[x_1,~\cdots, x_k]^T$ that can generate all positive $k$-group sparse signals. 
For each entry $x_i$ of $x\in\mathbb R^k$,  we consider a sequence of functions of the form:
\begin{equation}\label{eq:construct-1}
\tilde{\Gamma}_r(x_i):= \sigma(\sigma(x_i-2r) - 2\sigma(x_i-2r-1)),~~r\in\left\{1,~2,~\cdots,~\frac{d}{k}\right\}.
\end{equation}
Graphically, it is a sequence of $d/k$ non-overlapping triangle functions on the positive real line with width 2 and height 1. We use outputs of $\tilde{\Gamma}_r(x_i)$ over all $r$ as the output of the $i$-th block in $\mathbb{R}^d$. It then follows that for any $x_i\in\mathbb{R}$, there is only one of $\tilde{\Gamma}_r(x_i)$ that can be nonzero. Furthermore, the nonzero entry can take any value in $[0,1]$. Thus, lining up all $k$ blocks constructed in such a way, we have any positive $k$-group sparse vector in $\mathbb B_\infty^d(1)$ can be generated by this network, and so does any vector in $\mathbb B_2^d(1)$.

To represent such a network above using a ReLU network \emph{with no offset}, we add another hidden layer of width $(k+2d/k)$ before passing to $\tilde{\Gamma}_r(\cdot)$ and make use of the additional $k+1$ entries. The proposed network with a $k+1$ dimensional input of the form: $[x_1,\cdots,x_k,z]^T$ can be constructed as follows. The first $k$ nodes are:
\[
\sigma(x_i),~~i\in\{1,2,\cdots,k\}.
\]
The next $2d/k$ nodes are used to construct the offsets:
\[
\sigma(r\cdot z),~~r\in\left\{1,2,\cdots, \frac{2d}{k}\right\}.
\]
The second and the third hidden layers are almost the same as \eqref{eq:construct-1} mapping each $\sigma(x_i)$ into a block in $\mathbb{R}^d$ of length $d/k$, except that we replace the offsets $2r$ and $2r+1$ by the output computed in the first hidden layer, i.e., $\sigma(r\cdot z)$. Then, we construct the second layer that can output the following results for all $i\in \{ 1,2,...,k \}$ and $r\in \{1,2,...,d/k \}$: 
\begin{align*}
 \Upsilon_r(x_i,z)=\sigma(\sigma(x_i)-2\sigma(r\cdot z)) \quad \text{ and } \quad  \Upsilon'_r(x_i,z) = \sigma(\sigma(x_i)-2\sigma(r\cdot z) - \sigma(z)). 
\end{align*}        
Finally, by constructing the third layer, we have for all $i\in \{ 1,2,...,k \}$ and $r\in \{1,2,...,d/k \}$
\begin{equation}\label{eq:construct-2}
\Gamma_r(x_i ,z) := \sigma\big(\Upsilon_r(x_i,z) - 2\Upsilon'_r(x_i,z)\big).
\end{equation}
Note that \eqref{eq:construct-1} fires only when $x_i\geq0$, on which case we have $\sigma(x_i)=x_i$. Finally, we take $z$ always equal to 1 and obtain $\tilde{\Gamma}_r(x_i) = \Gamma_r( x_i, 1)$. Thus, the proposed network \eqref{eq:construct-2} can generate all nonnegative $k$-group sparse signals in $\mathbb{B}_2^d(1)$.
\end{proof}

Furthermore, based on the next two lemmas, we give the proof of Lemma \ref{lem:linear-model-lower-bound-final}.
\begin{lemma}[Theorem 4.2 of \citet{plan2016high}]\label{lem:linear-model-lower-bound}
Assume that $\theta_0\in K$ where $K \subseteq \mathbb R^d$ satisfies $\lambda v\in K$ for any $v\in K$ and $\lambda\in[0,1)$. Assume that $\check{y} = \dotp{a}{\theta_0} + \xi$ with $\xi\sim\mathcal{N}(0,\sigma^2)$ and $ a\sim\mathcal N(0, \mf I_d)$. Let 
\begin{align*}
\delta_* := \inf_{t>0}\l\{ t+\frac{\sigma}{\sqrt m}\l(1+\sqrt{\log P_t}\r) \r\}, 
\end{align*}
where $P_t$ with $t>0$ is the packing number of $K\cap \mathbb{B}_2^d(t)$ with balls of radius $t/10$. Then, there exists an absolute constant $c>0$ such that any estimator $\widehat\theta$ which depends only on $m$ observations of  $(a, \check{y})$ satisfies
\[
\sup_{\theta_0\in K} \expect{\|\widehat{\theta} - \theta_0\|_2}\geq c\min\{\delta_*,\text{diam}(K)\}.
\]
\end{lemma}

\begin{lemma}
\label{lem:packing}
When $k\leq d/4$, for any $t\leq 1$, we have $P_t\geq \exp\l(ck\log d/k \r)$, where $P_t$ is defined as in Lemma \ref{lem:linear-model-lower-bound} with letting $K \subseteq \mathbb{B}^d_2(1)$ being a set containing all $k$ group sparse vectors  in $\mathbb B_2^d(1)$.  Here $c>0$ is an absolute constant. 
\end{lemma}
\begin{proof}[Proof of Lemma \ref{lem:packing}]
The proof of this lemma follows from the idea of randomized packing construction in Section 4.3 of \citet{plan2016high}.  
For any $t$, since $P_t$ is defined as the packing number with balls of radius scaling as $t$, which is the radius of the set $K \cap \mathbb{B}_2^d(t)$,  then we have $P_t  = P_1$. Thus, we only need to consider the lower bound of $P_1$.  Furthermore, since $\mc S^{d-1} \subseteq \mathbb{B}_2^d(1)$, where $\mc S^{d-1}$ is the unit sphere in $\mathbb R^d$ of radius $1$, the packing number of $K \cap \mathbb{B}_2^d(1)$ is larger $K \cap \mc S^{d-1}$. Thus, we consider $1/10$ packing of the set 
$K\cap \mc S^{d-1}$ to obtain the lower bound of $P_1$. Consider a subset $K\cap \mc S^{d-1}$ such that it contains all nonnegative $k$ group sparse signals in $\mathbb{R}^d$ where each non-zero entry equals $1/\sqrt{k}$. This is possible due to Proposition \ref{prop:transform}. Then, we have $|\mathcal C| = (d/k)^k$. 
We will show that there exists a large enough subset $\mathcal X\subseteq\mathcal C$ such that $\forall x, y \in\mathcal X$, $\|x-y\|_2>1/10$. 
Consider picking vectors $x,y\in\mathcal C$ uniformly at random and computing the probability of the event
$
\|x-y\|_2^2\leq 1/100
$. When the event happens, it requires $x$ and $y$ to have at least $0.99k$ matching non-zero coordinates. Assume without loss of generality that $0.01k$ is an integer, this event happens with probability
\[
\l.{k\choose 0.99k}{d-0.99k\choose 0.01k}\r/(d/k)^k.
\]
Using Stirling's approximation and $k\leq d/4$, we have $\mathrm{Pr}(\|x-y\|_2^2\leq 1/100)\leq \exp(-c'k\log (d/k))$, where $c'>0$ is an absolute constant. This implies the claim that when choosing $\mathcal X$ to have $\exp(ck\log (d/k))$ uniformly chosen vectors from $\mathcal C$, which satisfies $\forall x, y \in\mathcal X$, $\|x-y\|_2>1/10$ with a constant probability.
\end{proof}

\begin{lemma}[Lemma \ref{lem:linear-model-lower-bound-final}]
Assume that $\theta_0\in K \subseteq \mathbb{B}^d_2(1)$ where $K$ is a set containing any $k$-group sparse vectors in $\mathbb{B}^d_2(1)$, and $K$ satisfies that $\forall v\in K$ then $ \lambda v \in K , \forall \lambda \in [0,1)$. Assume that $\check{y} = \dotp{a}{\theta_0} + \xi$ with $\xi\sim\mathcal{N}(0,\sigma^2)$ and $ a\sim\mathcal N(0, \mf I_d)$. Then, there exist absolute constants $c_1, c_2>0$ such that any estimator $\widehat\theta$ which depends only on $m$ observations of  $(a, \check{y})$ satisfies that when $m \geq c_1 k\log(d/k)$, there is 
\[
\sup_{\theta_0\in K} \EE \|\widehat{\theta} - \theta_0\|_2\geq c_2\sqrt{\frac{k\log(d/k)}{m}}.
\]
\end{lemma}

\begin{proof}[Proof of Lemma \ref{lem:linear-model-lower-bound-final}] 
Since $K$ satisfies $\lambda v \in K$ for any $ v \in K$ and $\lambda \in [0,1)$. Thus, by  Lemma \ref{lem:linear-model-lower-bound},  we have
\begin{align*}
\delta_* = \inf_{t>0}\l\{ t+\frac{\sigma}{\sqrt m}\l(1+\sqrt{\log P_t}\r) \r\}.
\end{align*}
Consider that for any $t > 1$, we can observe that 
\begin{align*}
t+\frac{\sigma}{\sqrt m}\l(1+\sqrt{\log P_t}\r) > 1.
\end{align*} 
On the other hand, for any $t \leq 1$, then we have
\begin{align*}
&\inf_{0 <  t \leq 1}\l\{ t+\frac{\sigma}{\sqrt m}\l(1+\sqrt{\log P_t}\r) \r\}\\
&\quad = \inf_{0 <  t \leq 1}\l\{ t+\frac{\sigma}{\sqrt m}\l(1+\sqrt{\log P_1}\r) \r\}\\
&\quad  = \frac{\sigma}{\sqrt m}\l(1+\sqrt{\log P_1}\r),
\end{align*}
where the first equality is due to Lemma \ref{lem:packing}, the second equality is by taking $\inf$ over $t$. If $m \geq \sigma^2 (1+\sqrt{\log P_1})^2$, we have 
\begin{align*}
\inf_{0 <  t \leq 1} \{ t+\tfrac{\sigma}{\sqrt m} (1+\sqrt{\log P_t} )  \} \leq 1.
\end{align*} 
Comparing the cases $t > 1$ and $t\leq 1$,  we get that, if $m \geq \sigma^2 (1+\sqrt{\log P_1})^2$, then 
\begin{align*}
\delta_0 = \inf_{t>0}\l\{ t+\frac{\sigma}{\sqrt m}\l(1+\sqrt{\log P_t}\r) \r\}  = \inf_{0 <  t \leq 1} \l\{ t+\frac{\sigma}{\sqrt m} \l(1+\sqrt{\log P_t} \r)  \r\} = \frac{\sigma}{\sqrt m}\l(1+\sqrt{\log P_1}\r).
\end{align*}
Moreover, since the $\text{diam}(K) \leq 1$, then by  Lemma \ref{lem:linear-model-lower-bound}, we have
\begin{align*}
\sup_{\theta_0\in K} \EE  \|\widehat{\theta} - \theta_0\|_2 \geq c\min\{\delta_*,\text{diam}(K)\} = \frac{c\sigma}{\sqrt{m}} \l( 1+\sqrt{\log P_1} \r) \geq \frac{c\sigma}{\sqrt{m}} \sqrt{\log P_1},
\end{align*}
by letting $m\geq \sigma^2 (1+\sqrt{\log P_1})^2$. Furthermore, according to Lemma \ref{lem:packing}, we know $\log P_1 \geq c' k\log(d/k)$ with $c'$ being an absolute constant. Then, there exists a sufficient large absolute constant $c_1$ such that when $m\geq c_1 k \log (d/k)$, we have
\begin{align*}
\sup_{\theta_0\in K} \EE [\|\widehat{\theta} - \theta_0\|_2]\geq c_2\sqrt{\frac{k\log(d/k)}{m}}.
\end{align*}
\end{proof}
\section{Proof of Theorem \ref{thm:determine_converge}} \label{sec:proof_thm_22}

Before presenting the proof of Theorem \ref{thm:determine_converge}, we first introduce some notations and definitions used hereafter. These notations and definitions will also be used in the proof of Theorem \ref{thm:optimum_compare}
in Section \S\ref{sec:proof_thm_23}.
According to the definition of $W_{i,+,x}$ in the paper, we can know that $G(x)$ can be represented as 
\[
G(x) =  \l( \prod_{i=1}^n W_{i, +, x}\r)  x  = (W_{n, +, x} W_{n-1, +, x} \cdots W_{1,+,x}) x.
\]
We therefore further define a more general form $H_x(z)$ as follows, 
\begin{align*}
H_x(z) := \l( \prod_{i=1}^n W_{i, +, x} \r)  z = (W_{n, +, x} W_{n-1, +, x}\cdots W_{1, +, x})  z,
\end{align*}
by which we can see that $H_x(x) = G(x)$.
 
Recall that as shown in the main body of the paper, for any $x$ such that $L(x)$ is differentiable, we can write the gradient of $L(x)$ w.r.t. $x$ as follows 
\begin{align*}
\nabla L(x) = 2 \l(\prod_{j=1}^n W_{j, +, x}\r)^\top \l(\prod_{j=1}^n W_{j, +, x}\r) x - \frac{2\lambda}{m} \sum_{i=1}^m y_i \l( \prod_{j=1}^n W_{j, +, x} \r)^\top a_i, 
\end{align*}
by which we further have
\begin{align*}
\dotp{\nabla L(x)}{z}  = 2 \dotp{G(x)}{H_x(z)} - \frac{2\lambda}{m} \sum_{i=1}^m y_i \dotp{a_i}{H_x(z)},
\end{align*}
for any $x$ and $z$.

We then let
\begin{align}
&h_{x,x_0} := \frac{1}{2^n} x - \frac{1}{2^n} \l[ \l(  \prod_{i=0}^{n-1} \frac{\pi-\overline{\varrho}_i}{\pi}  \r) x_0 + \sum_{i=0}^{n-1} \frac{\sin \overline{\varrho}_i}{\pi} \l( \prod_{j=i+1}^{d-1}\frac{\pi - \overline{\theta}_j}{\pi} \r) \frac{\|x_0\|_2}{\|x\|_2} x \r], \label{eq:h} \\
&S_{\varepsilon,x_0} := \{ x \neq 0 : \|h_{x, x_0}\|_2 \leq \frac{1}{2^n} \varepsilon \max (\|x\|_2, \|x_0\|_2) \}. \label{eq:S}
\end{align}
where $\overline{\varrho}_0 = \angle (x,x_0) $ and $\overline{\varrho}_i = g(\overline{\varrho}_{i-1})$, and $g(\varrho):=\cos^{-1} \l( \frac{(\pi-\varrho)\cos\varrho + \sin\varrho}{\pi} \r)$ as defined in Lemma~\ref{lem:converge_set}. In the following subsections, we provides key lemmas for the proof of Theorem \ref{thm:determine_converge}, and then a proof sketch of this theorem, followed by a detailed proof.

\subsection{Lemmas for Theorem \ref{thm:determine_converge}}

\begin{lemma}\label{lem:comp-expect}
Define $H_x(z) = \prod_{j=1}^n W_{j, +, x}  z $. Suppose that $G(x_0)$ satisfies $|G(x_0)| \leq R$. There exists an absolute constant $c_1>0$ such that for any $z$ and any $x$, when $$\lambda\geq4\max\{c_1(R\|a\|_{\psi_1}+\|\xi\|_{\psi_1}),1\}\log(64\max\{c_1(R\|a\|_{\psi_1} + \|\xi\|_{\psi_1}),1\}/\varepsilon),$$ the following holds:
\begin{align*}
\l|\lambda \expect{y_i\dotp{a_i}{H_x(z)}} - \dotp{G(x_0)}{H_x(z)} \r|\leq \frac{1}{4} \varepsilon  \|H_x(z)\|_2.
\end{align*}
\end{lemma}

\begin{proof}[Proof of Lemma \ref{lem:comp-expect}]
Recall that $y_i= \sign(\dotp{a_i}{G(x_0)} + \xi_i + \tau_i)$. We let $V_i = \dotp{a_i}{G(x_0)} + \xi_i$ and 
$Z_i =\dotp{a_i}{H_z(x)} $. Still, we assume $V_i$ and $\tau_i$ are independent. Thus, there is
\begin{align*}
\expect{\sign(V_i+\tau_i)|V_i} = \frac{V_i}{\lambda} -  \frac{V_i}{\lambda} \mathbf{1}_{\{|V_i|> \lambda\}} + \mathbf{1}_{\{V_i> \lambda\}} - \mathbf{1}_{\{V_i< -\lambda\}}.
\end{align*}
Therefore, we have
\begin{align*}
&\l|\expect{Z_i\sign(V_i+\tau_i)} - \frac{\expect{Z_i V_i}}{\lambda}  \r| \\
& \qquad =  \l|-\expect{\frac{Z_i V_i}{\lambda}\mathbf{1}_{\{|V_i|> \lambda\}}}  + \expect{Z_i\mathbf{1}_{\{V_i> \lambda\}}} -  \expect{Z\mathbf{1}_{\{V_i> \lambda\}}}\r|\\
& \qquad \leq \frac{\|Z_i\|_{L_2}\cdot\|V_i\mathbf{1}_{\{|V_i|> \lambda\}}\|_{L_2}}{\lambda} + 2\|Z_i\|_{L_2} \mathrm{Pr}(|V_i|> \lambda)^{1/2}, 
\end{align*}
where the last line follows from Cauchy-Schwarz inequality. 

First, by the isotropic assumption of $a_i$, we have
\[
\|Z_i\|_{L_2} = \l \{ \expect{|\dotp{a_i}{H_x(z)}|^2} \r\}^{1/2}  =  \|H_x(z)\|_2.
\]
Next, same to Lemma \ref{lem:expected-risk}, we have
\begin{align*}
\|V_i \mathbf{1}_{\{|V_i|> \lambda\}}\|_{L_2} \leq &\sqrt{2c_1(\lambda+1)\|\dotp{a_i}{G(x_0)} + \xi_i\|_{\psi_1}}e^{-\lambda/2\|\dotp{a_i}{G(x_0)} + \xi_i\|_{\psi_1}}\\
\leq & \sqrt{2c_1(\lambda+1)(\|a\|_{\psi_1}R + \| \xi\|_{\psi_1})}e^{-\lambda/2(\|a\|_{\psi_1}R + \| \xi\|_{\psi_1})}.
\end{align*}
due to our assumption that $\|G(x_0)\|_2\leq R$ and $V_i$ is sub-gaussian. Moreover, we also have 
\[
\mathrm{Pr}(|V_i|> \lambda)^{1/2}\leq \sqrt{c_1(\|a\|_{\psi_1}R + \| \xi\|_{\psi_1})}e^{-\lambda/2(\|a\|_{\psi_1}R + \| \xi\|_{\psi_1})}.
\]
Overall, we can obtain
\begin{align*}
\l|\lambda \expect{Z_i\sign(V_i+\tau_i)} - \expect{Z_i V_i}  \r|\leq \sqrt{c_1(\|a\|_{\psi_1}R + \| \xi\|_{\psi_1})}(\sqrt{2(\lambda+1)}+2\lambda)e^{-\lambda/2(\|a\|_{\psi_1}R + \| \xi\|_{\psi_1})}
\|H_x(z)\|_2.
\end{align*}
When 
$$\lambda\geq4\max\{c_1(R\|a\|_{\psi_1}+\|\xi\|_{\psi_1}),1\}\log(64\max\{c_1(R\|a\|_{\psi_1} + \|\xi\|_{\psi_1}),1\}/\varepsilon),$$
it is immediate that 
\begin{equation*}
\sqrt{c_1(\|a\|_{\psi_1}R + \| \xi\|_{\psi_1})}(\sqrt{2(\lambda+1)}+2\lambda)e^{-\lambda/2(\|a\|_{\psi_1}R + \| \xi\|_{\psi_1})}
\leq \frac14\varepsilon.
\end{equation*}
As a consequence, we have
\[
\l|\lambda \expect{y_i\dotp{a_i}{H_x(z)}} - \dotp{G(x_0)}{H_x(z)} \r| \leq \frac14\varepsilon\|H_x(z)\|_2,
\]
which finishes the proof.
\end{proof}

\begin{lemma}  \label{lem:comp-variance}
Define $H_x(z) := \prod_{j=1}^n W_{j, +, x}  z $. Suppose that $G(x_0)$ satisfies $|G(x_0)| \leq R$.  Then, with probability at least $1-c_4\exp(-u)$ where $c_4 > 0$ is an absolute constant,
\begin{align*}
\sup_{x\in \R^k, z\in \R^k, ~x_0 \in \R^k, ~|G(x_0)|\leq R } \frac{\l|\frac{\lambda }{m}\sum_{i=1}^m y_i \dotp{a_i}{H_x(z)} - \lambda \expect{y_i\dotp{a_i}{H_x(z)}}\r|}{\|H_x(z)\|_2} \leq \frac{\varepsilon}{8},
\end{align*}
where the sample complexity is
\begin{align*}
m\geq c_2\lambda^2\log^2(\lambda m)(kn\log(ed) + k\log(2R) + k\log m+ u)/\varepsilon^2,
\end{align*}
for some absolute constant $c_2$ large enough.
\end{lemma}

\begin{proof} [Proof of Lemma \ref{lem:comp-variance}]
The proof of Lemma \ref{lem:comp-variance} is very similar to the proofs shown in the previous subsection. Therefore, we only outline the main proof steps here but ignore detailed calculation for some inequalities. We aim to bound the following term
\begin{align*}
\sup_{x\in \R^k, z\in \R^k, x_0 \in \R^k, |G(x_0)|\leq R } \frac{\l|\frac{1}{m}\sum_{i=1}^m y_i \dotp{a_i}{H_x(z)} - \expect{y_i\dotp{a_i}{H_x(z)}}\r|}{\|H_x(z)\|_2}.
\end{align*}
By symmetrization inequality in Lemma \ref{lemma:symmetry}, it suffices to bound 
\begin{align*}
\sup_{x\in \R^k, z\in \R^k, x_0 \in \R^k, |G(x_0)|\leq R } \frac{\l|\frac{1}{m}\sum_{i=1}^m \varepsilon_i y_i \dotp{a_i}{H_x(z)}\r|}{\|H_x(z)\|_2}
\end{align*}
where $\{\varepsilon_i\}$ are i.i.d. Rademacher random variables that are independent of other random variables.

We rewrite the set $\{G(x_0): \|G(x_0)\|_2\leq R,~x_0\in\mathbb{R}^k\}$ as $G(\mathbb{R}^k)\cap \mathbb{B}_2^d(R)$. To bound the supremum above is based on building a $\delta$-covering net over the set $G(\mathbb{R}^k)\cap \mathbb{B}_2^d(R)$, namely $\mathcal{N}(G(\mathbb{R}^k)\cap \mathbb{B}_2^d(R),~\delta)$. The $\delta$ value should be carefully chosen. For a simply notation, we let $y_i^v := \sign (\dotp{a_i}{G(v)} + \xi_i + \tau_i)$ be the sign function associated with $G(v)$ in the net. We begin our proof by bounding the supremum term as follows, with probability at least $1-\exp(-c_0u)-2\exp(-u)$,

\begin{align*}
&\sup_{x,z,x_0\in\mathbb{R}^k,~\|G(x_0)\|_2\leq R}
\frac{\l|\frac1m\sum_{i=1}^m\varepsilon_i y_i \dotp{a_i}{H_x(z)}\r|}{\|H_x(z)\|_2}\\
& \qquad \leq
\sup_{x, x_0, z\in\mathbb{R}^k,~\|G(x_0)-G(v)\|_2\leq \delta,~G(v)\in\mathcal{N}(G(\mathbb{R}^k)\cap \mathbb{B}_2^d(R),~\delta)} \frac{\l|\frac1m\sum_{i=1}^m\varepsilon_i y_i^v \dotp{a_i}{H_x(z)}\r|}{\|H_x(z)\|_2}\\
& \qquad \quad + \sup_{x,z, x_0\in\mathbb{R}^k, ~\|G(x_0)-G(v)\|_2\leq \delta, ~G(v)\in\mathcal{N}(G(\mathbb{R}^k)\cap \mathbb{B}_2^d(R),~\delta)} \frac{\l|\frac1m\sum_{i=1}^m\varepsilon_i (y_i- y_i^v) \dotp{a_i}{H_x(z)}\r|}{\|H_x(z)\|_2} \\
& \qquad \leq\underbrace{
\sup_{x,z\in\mathbb{R}^k,G(v)\in\mathcal{N}(G(\mathbb{R}^k)\cap \mathbb{B}_2^d(R),~\delta)} \frac{\l|\frac1m\sum_{i=1}^m\varepsilon_i\sign(\dotp{a_i}{G(v)} + \xi_i + \tau_i)\dotp{a_i}{H_x(z)}\r|}{\|H_x(z)\|_2}
}_{\text{(I)}}\\
&\qquad \quad +\underbrace{
\sup_{x,z\in \R^k}
\max_{|\mathcal I|\leq 4\eta/\lambda}
\frac{2}{m}\sum_{i\in\mathcal I}\frac{|\dotp{a_i}{H_x(z)}|}{\|H_x(z)\|_2}
}_{\text{(II)}},
\end{align*}
where the first inequality is due to decomposition of the supremum term, and the second inequality is by Lemma \ref{lem:main-count} for term (II) and dropping the constraint $\|G(x_0)-G(v)\|_2\leq \delta$ for term (I).

\noindent \textbf{Bounding Term (I):} We first show the bound based on fixed $G(v)$. Then we give a uniform bound for any $G(v)$ in the $\delta$-net.
For a fixed $G(v)$, we have
\begin{align*}
&\sup_{x, z\in\mathbb{R}^k} \frac{\l|\frac1m\sum_{i=1}^m\varepsilon_i \sign(\dotp{a_i}{G(v)} + \xi_i + \tau_i)  \dotp{a_i}{H_x(z)}\r|}{\|H_x(z)\|_2} = \sup_{x, z\in\mathbb{R}^k} \frac{\l|\frac1m\sum_{i=1}^m\varepsilon_i  \dotp{a_i}{H_x(z)}\r|}{\|H_x(z)\|_2}.
\end{align*} 
For the function $H_x(z)$, we can see that as $x$ varies, $H_x(z)$ can be different linear functions, which constructs at most $[\mathcal{C}(d, k)]^n = [ {d\choose 0} + {d\choose 1} + \cdots + {d\choose k}]^n \leq  (d^k + 1)^n \leq (2d)^{kn}$ hyperplanes that split the whole $\R^{k}$ space.

Now, we consider any one piece $H_{\tilde{x}}$ where $\tilde{x} \in \mathcal{P}\subseteq\mathbb R^k$ and bound the following quantity:
\begin{align*}
\sup_{z\in \R^k}
\frac{\l|\frac1m\sum_{i=1}^m\varepsilon_i\dotp{a_i}{H_{\tilde{x}}(z)}\r|}{\|H_{\tilde{x}}(z)\|_2} \leq & \sup_{z\in\mathbb{R}^{k}}
\frac{\l|\frac1m\sum_{i=1}^m\varepsilon_i\dotp{a_i}{W_{\tilde{x}} z}\r|}{\|W_{\tilde{x}} z\|_2} \\
\leq&\sup_{b\in\mathcal{E}^{k}\cap\mathcal{S}^{d-1}}\l|\frac1m\sum_{i=1}^m\varepsilon_i\dotp{a_i}{b}\r|,
\end{align*}
where we let $W_{\tilde{x}} = \prod_{j=1}^n W_{j, +, \tilde{x}}  $ be the linear function at $\tilde{x}$ such that $H_{\tilde{x}}(z) = \l(\prod_{j=1}^n W_{j, +, \tilde{x}} \r) z $, and let $\mathcal{E}_{k}$ be the subspace in $\mathbb R^d$ spanned by the $k$ columns of $W_{\tilde{x}}$. We also define $b = W_{\tilde{x}} z/\|W_{\tilde{x}} z\| $ in the above formulation.

It suffices to bound the last term in the above formulation. We consider a $1/2$-covering net of the set $\mathcal{E}^{k}\cap\mathcal{S}^{d-1}$, namely, 
$\mathcal{N}(\mathcal{E}^{k}\cap\mathcal{S}^{d-1},1/2)$. A simple volume argument shows that the cardinality
$|\mathcal{N}(\mathcal{E}^{k}\cap\mathcal{S}^{d-1},1/2)|\leq 3^{k}$.

By Bernstein's inequality in Lemma \ref{Bernstein}, we have for any fixed $v\in\mathcal{N}(\mathcal{E}^{k}\cap\mathcal{S}^{d-1},1/2)$,
\[
\mathrm{Pr}\l( \l|\frac1m\sum_{i=1}^m\varepsilon_i\dotp{a_i}{b}\r|\geq \sqrt{\frac{2u'}{m}} + \frac{\|a\|_{\psi_1}u'}{m} \r)\leq 2e^{-u'}.
\]
Taking $u' = u+ckn\log(ed)$ for some $c>6$, we have with probability at least $1-2\exp(-u-ckn\log(ed))$, 
\[
\l|\frac1m\sum_{i=1}^m\varepsilon_i\dotp{a_i}{v}\r|\leq \sqrt{\frac{2(u+ckn\log(ed))}{m}} + \frac{\|a\|_{\psi_1} (u+ckn\log(ed))}{m}.
\]
Taking a union bound over all $v\in\mathcal{N}(\mathcal{E}^{k}\cap\mathcal{S}^{d-1},1/2)$, we have with probability at least $1-2\exp(-u-ckn\log(ed))\cdot 3^{k}\geq 1-2\exp(-u-c_1kn\log(ed))$ for some absolute constant $c_1>2$.
\begin{equation}
\sup_{v\in\mathcal{N}(\mathcal{E}^{k}\cap\mathcal{S}^{d-1},1/2)}\l|\frac1m\sum_{i=1}^m\varepsilon_i\dotp{a_i}{b}\r|\leq \sqrt{\frac{2(u+ckn\log(ed))}{m}} + \frac{\|a\|_{\psi_1} (u+ckn\log(ed))}{m}.
\end{equation}
Therefore, we will have 
\begin{align*}
&\sup_{b\in\mathcal{E}^{k}\cap\mathcal{S}^{d-1}}\l|\frac1m\sum_{i=1}^m\varepsilon_i\dotp{a_i}{b}\r| \\
& \qquad \leq  \sup_{v\in\mathcal{N}(\mathcal{E}^{k}\cap\mathcal{S}^{d-1},1/2)}\l|\frac1m\sum_{i=1}^m\varepsilon_i\dotp{a_i}{v}\r|  
+ \sup_{b\in\mathcal{E}^{k}\cap\mathcal{S}^{d-1}, v\in\mathcal{N}(\mathcal{E}^{k}\cap\mathcal{S}^{d-1},1/2), \|b-v\|_2 \leq 1/2}\l|\frac1m\sum_{i=1}^m\varepsilon_i\dotp{a_i}{b- v}\r|  \\
& \qquad \leq \sup_{v\in\mathcal{N}(\mathcal{E}^{k}\cap\mathcal{S}^{d-1},1/2)}\l|\frac1m\sum_{i=1}^m\varepsilon_i\dotp{a_i}{v}\r|  
+ \frac12\sup_{b\in\mathcal{E}^{k}\cap\mathcal{S}^{d-1}, v\in\mathcal{N}(\mathcal{E}^{k}\cap\mathcal{S}^{d-1},1/2), \|b-v\|_2 \leq 1/2}\l|\frac1m\sum_{i=1}^m\frac{\varepsilon_i\dotp{a_i}{b - v}}{\|b -v\|_2}\r|   \\
& \qquad \leq  \sup_{v\in\mathcal{N}(\mathcal{E}^{k}\cap\mathcal{S}^{d-1},1/2)}\l|\frac1m\sum_{i=1}^m\varepsilon_i\dotp{a_i}{v}\r| + \frac{1}{2}\sup_{b\in\mathcal{E}^{k}\cap\mathcal{S}^{d-1}}\l|\frac1m\sum_{i=1}^m\varepsilon_i\dotp{a_i}{b}\r|.
\end{align*}
Now we can obtain 
\begin{align*}
\sup_{z\in \R^k}\frac{\l|\frac1m\sum_{i=1}^m\varepsilon_i\dotp{a_i}{H_{\tilde{x}}(z)}\r|}{\|H_{\tilde{x}}(z)\|_2} \leq & \sup_{b\in\mathcal{E}^{k}\cap\mathcal{S}^{d-1}}\l|\frac1m\sum_{i=1}^m\varepsilon_i\dotp{a_i}{b}\r| \\
\leq & 2\sqrt{\frac{2(u+ckn\log(ed))}{m}} + \frac{2\|a\|_{\psi_1} (u+ckn\log(ed))}{m}.
\end{align*}
Taking a further union bound over at most $(2d)^{kn}$ linear functions, we have
\begin{align*}
\text{(I)}  \leq \sqrt{\frac{8(u+ckn\log(ed))}{m}} + \frac{2\|a\|_{\psi_1} (u+ckn\log(ed))}{m}
\end{align*}
with probability at least $1-2\exp(-u-c_1kn\log(ed))\cdot (2d)^{kn}\geq 1-2\exp(-u-c'  kn\log(ed))$ where $c' > 1$.

\noindent \textbf{Bounding Term (II):}
Now we bound the term 
\begin{align*}
\text{(II)} = \sup_{x,z\in \R^k}
\max_{|\mathcal I|\leq 4m\eta/\lambda}
\frac{2}{m}\sum_{i\in\mathcal I}\frac{|\dotp{a_i}{H_x(z)}|}{\|H_x(z)\|_2}
\end{align*}
Let $t = H_x(z) / \| H_x(z) \|_2$ and it is enough to bound
\begin{align} \label{eq:split-decomp-term2}
\sup_{t \in \mathcal{E}^k \cap \mathcal{S}^{d-1}}
\max_{|\mathcal I|\leq 4m\eta/\lambda}\frac1m\sum_{i\in\mathcal I}   \l( |\dotp{a_i}{t}| - \expect{|\dotp{a_i}{t}|} + \expect{|\dotp{a_i}{t}|} \r).
\end{align}
Note that $|\dotp{a_i}{t}| - \expect{|\dotp{a_i}{t}|}$ is also a sub-exponential random variable with sub-exponential norm bounded by $2\|a\|_{\psi_1}$, and $\expect{|\dotp{a_i}{t}|}\leq 1$. Given $x$, $H_x(z)$ is a linear function and there are at most $(2d)^{kn}$ different linear function for different $x$. 

First, we bound the term $\sup_{t \in \mathcal{E}^k \cap \mathcal{S}^{d-1}}
\max_{|\mathcal I|\leq 4m\eta/\lambda}\frac1m\sum_{i\in\mathcal I}   \l( |\dotp{a_i}{t}| - \expect{|\dotp{a_i}{t}|}\r)$. 
We have
\begin{align*}
&\sup_{t \in \mathcal{E}^k \cap \mathcal{S}^{d-1}}
\max_{|\mathcal I|\leq 4m\eta/\lambda}\frac1m\sum_{i\in\mathcal I}   \l( |\dotp{a_i}{t}| - \expect{|\dotp{a_i}{t}|}\r) \\
& \qquad = \sup_{t \in \mathcal{E}^k \cap \mathcal{S}^{d-1}}
\max_{|\mathcal I|\leq 4m\eta/\lambda}   \frac1m\sum_{i=1}^m  \mathbf{1}_{\{i\in \mathcal I\}} \l( |\dotp{a_i}{t}| - \expect{|\dotp{a_i}{t}|}\r) \\
& \qquad \leq \sup_{t \in \mathcal{E}^k \cap \mathcal{S}^{d-1}}
\max_{|\mathcal I|\leq 4m\eta/\lambda}  \l| \frac1m\sum_{i=1}^m  \mathbf{1}_{\{i\in \mathcal I\}} \l( |\dotp{a_i}{t}| - \expect{|\dotp{a_i}{t}|}\r) \r|
\end{align*}
By symmetrization inequality, it suffices to bound 
\begin{align*}
 \sup_{t \in \mathcal{E}^k \cap \mathcal{S}^{d-1}}
\max_{|\mathcal I|\leq 4m\eta/\lambda}  \l| \frac1m\sum_{i=1}^m  \varepsilon_i   \mathbf{1}_{\{i\in \mathcal I\}}  |\dotp{a_i}{t}| \r|  =  \sup_{t \in \mathcal{E}^k \cap \mathcal{S}^{d-1}, |\mathcal I|\leq 4\eta/\lambda}
\l| \frac1m\sum_{i\in \mathcal I}  \varepsilon_i |\dotp{a_i}{t}| \r| 
\end{align*}
We consider a $1/2$-covering net of the set $\mathcal{E}^{k}\cap\mathcal{S}^{d-1}$, namely, 
$\mathcal{N}(\mathcal{E}^{k}\cap\mathcal{S}^{d-1},1/2)$. A simple volume argument shows that the cardinality
$|\mathcal{N}(\mathcal{E}^{k}\cap\mathcal{S}^{d-1},1/2)|\leq 3^{k}$.
Therefore, we will have 
\begin{align*}
&\sup_{t \in \mathcal{E}^k \cap \mathcal{S}^{d-1}, |\mathcal I|\leq 4\eta/\lambda}
\l| \frac1m\sum_{i\in \mathcal I}  \varepsilon_i |\dotp{a_i}{t}| \r|  \\
& \qquad \leq  \sup_{v\in\mathcal{N}(\mathcal{E}^{k}\cap\mathcal{S}^{d-1},1/2),  |\mathcal I|\leq 4\eta/\lambda} \l| \frac1m\sum_{i\in \mathcal I}  \varepsilon_i |\dotp{a_i}{v}| \r|  
+ \sup_{\substack{t\in\mathcal{E}^{k}\cap\mathcal{S}^{d-1}, v\in\mathcal{N}(\mathcal{E}^{k}\cap\mathcal{S}^{d-1},1/2), \\ \|t-v\|_2 \leq 1/2,  |\mathcal I|\leq 4\eta/\lambda}} \l| \frac1m\sum_{i\in \mathcal I}  \varepsilon_i |\dotp{a_i}{t-v}| \r|  \\
& \qquad \leq \sup_{v\in\mathcal{N}(\mathcal{E}^{k}\cap\mathcal{S}^{d-1},1/2), |\mathcal I|\leq 4\eta/\lambda }\l|\frac1m\sum_{i\in \mathcal I}\varepsilon_i |\dotp{a_i}{v}| \r|  
+ \frac12\sup_{\substack{t\in\mathcal{E}^{k}\cap\mathcal{S}^{d-1}, v\in\mathcal{N}(\mathcal{E}^{k}\cap\mathcal{S}^{d-1},1/2), \\ \|t-v\|_2 \leq 1/2, |\mathcal I|\leq 4\eta/\lambda}}\l|\frac1m\sum_{i\in \mathcal I}\frac{\varepsilon_i |\dotp{a_i}{t - v}|}{\|t -v\|_2}\r|   \\
& \qquad \leq  \sup_{v\in\mathcal{N}(\mathcal{E}^{k}\cap\mathcal{S}^{d-1},1/2), |\mathcal I|\leq 4\eta/\lambda}\l|\frac1m\sum_{i\in \mathcal I}\varepsilon_i | \dotp{a_i}{v}|\r| + \frac{1}{2}\sup_{t\in\mathcal{E}^{k}\cap\mathcal{S}^{d-1}, |\mathcal I|\leq 4\eta/\lambda}\l|\frac1m\sum_{i\in \mathcal I}\varepsilon_i|\dotp{a_i}{t}| \r|,
\end{align*}
which implies
\begin{align*}
\sup_{t \in \mathcal{E}^k \cap \mathcal{S}^{d-1}, |\mathcal I|\leq 4\eta/\lambda}
\l| \frac1m\sum_{i\in \mathcal I}  \varepsilon_i |\dotp{a_i}{t}| \r| \leq 2  \sup_{v\in\mathcal{N}(\mathcal{E}^{k}\cap\mathcal{S}^{d-1},1/2), |\mathcal I|\leq 4\eta/\lambda}\l|\frac1m\sum_{i\in \mathcal I}\varepsilon_i | \dotp{a_i}{v}|\r|.
\end{align*}
For any fixed $v$ in the $1/2$-net and a fixed $\mathcal I$, by Bernstein's inequality, we have
\begin{align*}
\l|\frac1m\sum_{i\in \mathcal I} \varepsilon_i |\dotp{a_i}{v} | \r| = \frac{|\mathcal I|}{m} \l|\frac{1}{|\mathcal I |}\sum_{i\in \mathcal I} \varepsilon_i\dotp{a_i}{v}\r|\leq \frac1m (2\sqrt{u_2|\mathcal I|}  + 2\|a\|_{\psi_1}u_2).
\end{align*}
We take 
$$
u_2 = C_1\log
\l(\frac{\lambda m}{kn\log(ed)+u'}\r)\Big(u+2kn\log(ed) + k\log(2R) + k\log\frac{m}{kn\log(ed)+u'}\Big)^{1/2}\sqrt{m},
$$
where $C_1$ is an absolute constant large enough and $u'$ satisfies \eqref{eq:u-prime-bound}.
Using the fact that 
\[
|\mathcal I|\leq \frac{4\eta}{\lambda} m \leq 2L
\Big(u+2kn\log(ed) + k\log(2R) + k\log\frac{m}{kn\log(ed)+u'}\Big)^{1/2}\sqrt{m},
\]
we have with probability at least 
\[
1-2\exp\l(-C_1\log
\l(\frac{\lambda m}{kn\log(ed)+u'}\r)\Big(u+2kn\log(ed) + k\log(2R) + k\log\frac{m}{kn\log(ed)+u'}\Big)^{1/2}\sqrt{m} \r),
\]
the following holds
\begin{align*}
\l|\frac1m\sum_{i\in \mathcal I} \varepsilon_i | \dotp{a_i}{v}| \r|  \leq C_2\log\Big(\frac{\lambda m}{kn\log(ed)+u'}\Big)\sqrt{\frac{u+2kn\log(ed) + k\log(2R) + k\log\frac{m}{kn\log(ed)+u'}}{m}},
\end{align*}
where $C_2$ is a constant depending on $\|a\|_{\psi_1}$, $C_1$, and $L$. To bound the maximum over $|\mathcal I|\leq 4\eta/\lambda$, we take a union bound over all ${m\choose 4\eta m/\lambda}$ possibilities,
\[
{m\choose 4\eta m/\lambda}\leq \l(\frac{em}{4\eta m/\lambda}\r)^{4\eta m/\lambda}
 = \l(\frac\lambda \eta\r)^{4\eta m/\lambda}.
\]
Thus, it follows
\begin{align*}
\log{m\choose 4\eta m/\lambda}&\leq \frac{4\eta m}{\lambda}\log\frac{\lambda}{\eta}\\
&\leq L\Big(u+2kn\log(ed) + k\log(2R) + \log\frac{m}{kn\log(ed)+u'}\Big)^{1/2}\sqrt{m}\log\Big(\frac{\lambda m}{kn\log(ed)+u'}\Big),
\end{align*}
and when $C_1>L$, taking the union bound gives, with probability at least 
\[
1-2\exp\l(-C_3\log
\l(\frac{\lambda m}{kn\log(ed)+u'}\r)\Big(u+2kn\log(ed) + k\log(2R) + \log\frac{m}{kn\log(ed)+u'}\Big)^{1/2}\sqrt{m} \r),
\]
we have 
\begin{align*}
& \max_{|\mathcal I|\leq 4m\eta/\lambda}\l|\frac1m\sum_{i\in\mathcal I} \varepsilon_i |\dotp{a_i}{v}|\r|  \\
&\qquad  \leq C_2 \log\Big(\frac{\lambda m}{kn\log(ed)+u'}\Big)\sqrt{\frac{u+2kn\log(ed) + k\log(2R) + k\log\frac{m}{kn\log(ed)+u'}}{m}}.
\end{align*}
Furthermore, taking the union bound on all the $1/2$-net, we obtain 
\begin{align*}
&\sup_{t \in \mathcal{E}^k \cap \mathcal{S}^{d-1}}
\max_{|\mathcal I|\leq 4m\eta/\lambda}\l|\frac1m\sum_{i\in\mathcal I} \varepsilon_i|\dotp{a_i}{t}|\r|\\
& \qquad  \leq 2 \sup_{v \in \mathcal N(\mathcal{E}^k \cap \mathcal{S}^{d-1},1/2)}
\max_{|\mathcal I|\leq 4m\eta/\lambda}\l|\frac1m\sum_{i\in\mathcal I}|\dotp{a_i}{v}|\r| \\
&\qquad \leq 2 C_2 \log\Big(\frac{\lambda m}{kn\log(ed)+u'}\Big)\sqrt{\frac{u+2kn\log(ed) + k\log(2R) + k\log\frac{m}{kn\log(ed)+u'}}{m}}.
\end{align*}
with probability
\begin{align*}
1-2\cdot 3^{k}(2d)^{kn} \exp\l(-C_3\log
\Big(\frac{\lambda m}{kn\log(ed)+u'}\Big)\Big(u+2kn\log(ed) + k\log(2R) + \log\frac{m}{kn\log(ed)+u'}\Big)^{1/2}\sqrt{m} \r).
\end{align*}
Thus, we have with the probability above, the following holds
\begin{align*}
&\sup_{t \in \mathcal{E}^k \cap \mathcal{S}^{d-1}}
\max_{|\mathcal I|\leq 4m\eta/\lambda}\frac1m\sum_{i\in\mathcal I}   \l( |\dotp{a_i}{t}| - \expect{|\dotp{a_i}{t}|}\r)\\
&\qquad \leq C_2 \log\Big(\frac{\lambda m}{kn\log(ed)+u'}\Big)\sqrt{\frac{u+2kn\log(ed) + k\log(2R) + k\log\frac{m}{kn\log(ed)+u'}}{m}}.
\end{align*}
On the other hand, for the extra expectation term in \eqref{eq:split-decomp-term2}, we have
\begin{align*}
&\sup_{t \in \mathcal{E}^k \cap \mathcal{S}^{d-1}}
\max_{|\mathcal I|\leq 4m\eta/\lambda}\frac1m\sum_{i\in\mathcal I}   \expect{|\dotp{a_i}{t}|} \leq \max_{|\mathcal I|\leq 4m\eta/\lambda} \frac{|\mathcal I|}{m} \\
&\qquad \leq  2L
\sqrt{\frac{u+2kn\log(ed) + k\log(2R) + k\log\frac{m}{kn\log(ed)+u'}}{m}},
\end{align*}
where the first inequality is due to $\expect{|\dotp{a_i}{t}|}  \leq 1$.
 
Therefore, combining the above results, if we set
\begin{align*}
m\geq c_2\lambda^2\log^2(\lambda m)(kn\log(ed) + k\log(2R) + k\log m+ u)/\varepsilon^2,
\end{align*}
for some absolute constant $c_2$ large enough, we have
\[
\text{(II)}
\leq
C_4\log\Big(\frac{\lambda m}{kn\log(ed)+u'}\Big)\sqrt{\frac{u+2kn\log(ed) + k\log(2R) + k\log\frac{m}{kn\log(ed)+u'}}{m}},
\]
with probability at least $1- c_3\exp(-u)$, where $c_3\geq1$ is an absolute constant and $C_4$ is a constant depending on $C_3$ and $L$.

\noindent \textbf{Combining (I) and (II):} 
Combining all the results above, letting $c_2\geq 256[(2\|a\|_{\psi_1}+3)(c+c')+C_4]^2$, we obtain with probability $1-c_3\exp(-u)-\exp(-c_0u)-4\exp(-u)$ (recalling that $c_0 \geq 1$ as shown in Lemma \ref{lem:main-count}),
\begin{align*}
\sup_{x\in \R^k, z\in \R^k, x_0 \in \R^k, |G(x_0)|\leq R } \frac{\l|\frac{1}{m}\sum_{i=1}^m \varepsilon_i y_i \dotp{a_i}{H_x(z)} \r|}{\|H_x(z)\|_2} \leq \frac{\varepsilon}{16\lambda}.
\end{align*}
which thus means, for any $x,z,x_0$, by symmetrization, we have
 \begin{align*}
\l|\frac{\lambda }{m}\sum_{i=1}^m y_i \dotp{a_i}{H_x(z)} - \lambda \expect{y_i\dotp{a_i}{H_x(z)}}\r| \leq \frac{\varepsilon}{8}\|H_x(z)\|_2 .
 \end{align*}
with probability at least $1- c_4 \exp(-u)$. This finishes the proof.
\end{proof}


The following lemmas are some useful lemmas from previous papers. We rewrite them here for integrity.

\begin{lemma}[\citet{hand2018global}] \label{lem:converge_set}
Suppose $8\pi n^6 \sqrt{\varepsilon} \leq 1$. Let 
\begin{align*}
S_{\varepsilon, x_0} := \{x \neq 0 \in \R^k | \|h_{x,x_0} \|_2 \leq \frac{1}{2^n	} \varepsilon \max(\|x\|_2, \|x_0\|_2)  \}, 
\end{align*}
where $n$ is an integer greater than 1 and let $h_{x,x_0}$ be defined by
\begin{align*}
h_{x,x_0} :=\frac{1}{2^n} x - \frac{1}{2^n} \l[ \l(  \prod_{i=0}^{n-1} \frac{\pi-\overline{\varrho}_i}{\pi}  \r) x_0 + \sum_{i=0}^{n-1} \frac{\sin \overline{\varrho}_i}{\pi} \l( \prod_{j=i+1}^{d-1}\frac{\pi - \overline{\varrho}_j}{\pi} \r) \frac{\|x_0\|_2}{\|x\|_2} x \r]
\end{align*}
where $\overline{\varrho}_0 = \angle (x,x_0) $ and $\overline{\varrho}_i = g(\overline{\varrho}_{i-1})$. Particularly, we define
\begin{align*}
g(\varrho):=\cos^{-1} \l( \frac{(\pi-\varrho)\cos\varrho + \sin \varrho}{\pi} \r).
\end{align*}
If $x\in S_{\varepsilon, x_0}$, then we have
\begin{align*}
S_{\varepsilon, x_0} \subset \mathcal{B} (x_0, 56n\sqrt{\varepsilon}\|x_0\|_2) \cup \mathcal{B} (-\rho_n x_0, 500n^{11} \sqrt{\varepsilon} \|x_0\|_2),
\end{align*}
where $\rho_n$ is defined as
\begin{align*}
\rho_n := \sum_{i=0}^{n-1} \frac{\sin \check{\varrho}_i}{\pi} \l( \prod_{j=i+1}^{n-1} \frac{\pi - \check{\varrho}_j}{\pi}\r) \leq 1,
\end{align*}
and $\check{\varrho}_0=\pi$ and $\check{\varrho}_i = g(\check{\varrho}_{i-1})$.
\end{lemma}

\begin{lemma}[\citet{hand2018global}] \label{lem:wdc_ineq}
Fix $0 < 16\pi n^2 \sqrt{\varepsilon_\WDC} < 1$ and $n \geq  2$. Suppose that $W_i$ satisfies the WDC with constant $\varepsilon_\WDC$ for $i = 1, \ldots,n$. Define
\begin{align*}
\tilde{h}_{x,z} = \frac{1}{2^n} \l[ \l( \prod_{i=0}^{n-1}\frac{\pi-\overline{\varrho}_i}{\pi} \r) z +  \sum_{i=0}^{n-1} \frac{\sin \overline{\varrho}_i }{\pi} \l( \prod_{j=i+1}^{n-1} \frac{\pi-\bar{\varrho}_j}{\pi} \r) \frac{\|z\|_2}{\|x\|_2} x \r],
\end{align*}
where $\overline{\varrho}_i  = g(\overline{\varrho}_{i-1})$ for $g$ in Lemma \ref{lem:converge_set} and $\overline{\varrho}_0 = \angle (x,z)$. For all $x\neq 0$ and $y \neq 0$, 
\begin{align}
&\l\|\l(\prod_{i=1}^n W_{i,+,x}\r)^\top G(z) - \tilde{h}_{x,z} \r \|_2 \leq 24 \frac{n^3\sqrt{\varepsilon_\WDC}}{2^n} \|z\|_2, \label{eq:wdc_ineq_1}\\
&\dotp{ G(x)}{G(z)} \geq \frac{1}{4\pi}\frac{1}{2^n} \|x\|_2 \|z\|_2, \label{eq:wdc_ineq_2} \\
&\|W_{i,+,x}\|_2 \leq \l( \frac{1}{2} + \varepsilon_\WDC\r)^{1/2}. \label{eq:wdc_ineq_3}
\end{align}

\end{lemma}

\subsection{Proof Sketch of Theorem \ref{thm:determine_converge}}

Under the conditions of Theorem \ref{thm:determine_converge}, our proof is sketched as follows:
\begin{itemize}
\item The key to proving Theorem \ref{thm:determine_converge} lies in understanding the concentration of $L(x)$ and $\nabla L(x)$. Here we prove two critical lemmas, Lemmas \ref{lem:comp-expect} and \ref{lem:comp-variance} in this section, combining which we can show that for any $x$, $z$ and $|G(x_0)|\leq R$, when $\lambda$ and $m$ are sufficiently large, the following holds with high probability
\begin{align*}
\Big|\frac{\lambda }{m}\sum_{i=1}^m y_i \dotp{a_i}{H_x(z)} - \dotp{G(x_0)}{H_x(z)} \Big| \leq \varepsilon \|H_x(z)\|_2.
\end{align*}
which further implies 
\begin{align*}
\frac{\lambda}{m} \sum_{i=1}^m y_i \dotp{a_i}{H_x(z)} \approx \dotp{G(x_0)}{H_x(z)},
\end{align*}
for any $x, z$.

Therefore, we have $\forall z$ and $\forall x$ such that $L(x)$ is differentiable, we can approximate $\nabla L(x)$ as follows:
\begin{align*}
 \langle\nabla L(x),z\rangle &\approx 2 \langle G(x), H_x(z) \rangle - 2 \langle G(x_0), H_x(z) \rangle.
\end{align*}

\item On the other hand, we can show that $\forall x, z$, 
\begin{align*}
\langle G(x), H_x(z) \rangle -  \langle G(x_0), H_x(z) \rangle \approx \langle  h_{x,x_0},z \rangle,
\end{align*}
which therefore leads to
\begin{align*}
\langle\nabla L(x),z\rangle \approx 2 \langle h_{x,x_0}, z \rangle.
\end{align*}

\item Following 
the previous step, with $v_x$ being defined in Theorem \ref{thm:determine_converge}, the directional derivative is approximated as 
\begin{align*}
D_{-v_x} L(x) \cdot \|v_x\|_2 \approx - 4 \| h_{x,x_0}\|^2_2.
\end{align*}

\item We consider the error of approximating $D_{-v_x} L(x) \cdot \|v_x\|_2$ by $- 4 \| h_{x,x_0}\|^2_2$ in the following two cases:

\textbf{Case 1:} When $\|x_0\|_2$ is not small and $x\neq 0$, one can show the error is negligible compared to $- 4 \| h_{x,x_0}\|^2_2$, so that $D_{-v_x} L(x) < 0$ as $- 4  \| h_{x,x_0}\|^2_2$. 

\textbf{Case 2:} When $\|x_0\|_2$ approaches 0, such an error is decaying slower than $- 4  \| h_{x,x_0}\|^2_2$ itself and eventually dominates it. As a consequence, one can only conclude that $\hat{x}_m$ is around the origin.
\item To characterize the directional derivative at $0$ in Case 1, one can show 
\begin{align*}
D_w L(0) \cdot \|w\|_2\leq   \l|\langle G(x_0), H_{x_N}(w) \rangle - \frac{\lambda }{ m } \sum_{i=1}^m y_i  \langle a_i, H_{x_N}(w)\rangle \r| - \langle G(x_0), H_{x_N}(w)\rangle 
\end{align*}
 with $x_N\rightarrow 0$. By showing that the second term dominates according to \eqref{eq:one-bit-rc} and Lemma \ref{lem:wdc_ineq}, we obtain
\begin{align*}
D_{w} L(0) < 0, \forall w\neq 0.
\end{align*}   

\end{itemize}

\subsection{Detailed Proof of Theorem \ref{thm:determine_converge}}

\begin{proof}[Proof of Theorem \ref{thm:determine_converge}]
According to Theorem \ref{thm:determine_converge}, we define a non-zero direction as follows:  
\begin{align*}
v_x := \begin{cases}
\nabla L(x) , &\text{ if $L(x)$ is differentiable at } x ,\\
\lim_{N\rightarrow +\infty} \nabla L(x_N), &\text{ otherwise, } 
\end{cases}
\end{align*}
where $\{x_N\}$ is a sequence such that $\nabla L(x)$ is differentiable at all point $x_N$ in the sequence because of the piecewise linearity of $G(x)$. 

On the other hand, by our definition of directional derivative, we have
\begin{align*}
D_{-v_x} L(x) = \begin{cases}
\big \langle  \nabla L(x), - \frac{v_x}{\|v_x\|_2}\big \rangle , &\text{ if $L(x)$ is differentiable at } x, \\
\lim_{N\rightarrow +\infty} \big \langle \nabla L(\tilde{x}_N),  -v_x / \|v_x\|_2 \big \rangle, &\text{ otherwise, } 
\end{cases}
\end{align*}
where $\{\tilde{x}_N\}$ is also a sequence with $\nabla L(\tilde{x}_N)$ existing for all $\tilde{x}_N$. Here we use $\tilde{x}_N$ only in order to distinguish from the sequence of $x_N$ in the definition of $v_x$ above.
We give the proof as follows:

\noindent \textbf{Approximation of $\langle\nabla L(x),z\rangle$:}
The proof is mainly based on the two critical lemmas, i.e., Lemma \ref{lem:comp-expect} and Lemma \ref{lem:comp-variance}.

First by \eqref{eq:wdc_ineq_3} in Lemma \ref{lem:wdc_ineq}, we have 
\begin{align}
\|G(x)\|_2 = (\prod_{i=1}^n W_{i,+,x}) x \leq (1/2 + \varepsilon_\WDC)^{n/2} \|x\|_2, \label{eq:proof22_0}
\end{align}
for any $x$. Thus, due to the assumption $\|x_0\|_2 \leq R (1/2 + \varepsilon_\WDC)^{-n/2} $ in Theorem \ref{thm:determine_converge} and $\|G(x_0)\|_2 \leq (1/2 + \varepsilon_\WDC)^{n/2} \|x_0\|_2$,  we further have 
\begin{align*}
\|G(x_0)\|_2 \leq R.
\end{align*}
By Lemma \ref{lem:comp-expect} and $\|G(x_0)\|_2 \leq R$, setting 
$$\lambda\geq4\max\{c_1(R\|a\|_{\psi_1}+\|\xi\|_{\psi_1}),1\}\log(64\max\{c_1(R\|a\|_{\psi_1} + \|\xi\|_{\psi_1}),1\}/\varepsilon),$$ the following holds for any $x$:
\begin{align}
\l|\lambda\expect{y_i\dotp{a_i}{G(x)}} - \dotp{G(x_0)}{G(x)} \r|\leq \frac{1}{4} \varepsilon  \|G(x)\|_2,  \label{eq:proof22_1}
\end{align}
if we let  $z = x$ in Lemma \ref{lem:comp-expect} such that $H_x(x) = G(x)$.

On the other hand, according to Lemma \ref{lem:comp-variance} and $|G(x_0)| \leq R$, we have that with probability at least $1-c_4 \exp(-u)$, for any $x$, the following holds:
\begin{align}
\l|\frac{\lambda }{m}\sum_{i=1}^m y_i \dotp{a_i}{G(x)} - \lambda \expect{y_i\dotp{a_i}{G(x)}}\r| \leq \frac{\varepsilon}{8}\|G(x)\|_2, \label{eq:proof22_2}
\end{align}
with sample complexity being
\begin{align*}
m\geq c_2\lambda^2\log^2(\lambda m)(kn\log(ed) + k\log(2R) + k\log m+ u)/\varepsilon^2,
\end{align*}
where we set $z = x$ in Lemma \ref{lem:comp-variance} with $H_x(x) = G(x)$.

Combining \eqref{eq:proof22_1} and \eqref{eq:proof22_2}, we will have that with probability at least $1-c_4 \exp(-u)$, for any $x$, setting 
\begin{align*}
\lambda\geq& 4\max\{c_1(R\|a\|_{\psi_1}+\|\xi\|_{\psi_1}),1\}\log(64\max\{c_1(R\|a\|_{\psi_1} + \|\xi\|_{\psi_1}),1\}/\varepsilon), \text{ and }\\
m \geq& c_2\lambda^2\log^2(\lambda m)(kn\log(ed) + k\log(2R) + k\log m+ u)/\varepsilon^2,
\end{align*}
the following holds
\begin{align}
\l|\frac{\lambda }{m}\sum_{i=1}^m y_i \dotp{a_i}{G(x)} - \dotp{G(x_0)}{G(x)} \r| \leq \varepsilon\|G(x)\|_2, \label{eq:proof22_3}
\end{align}
which leads to
\begin{align}
\l|\frac{1}{2}\langle\nabla L(x),z\rangle - ( \langle G(x), H_x(z) \rangle -  \langle G(x_0), H_x(z) \rangle) \r| \leq \varepsilon\|G(x)\|_2. \label{eq:proof22_4}
\end{align}

\noindent \textbf{Approximating $D_{-v_x} L(x) \cdot \|v_x\|_2$ and Bounding Errors:}
Without loss of generality, we directly prove the case where $L(x)$ is not differentiable at $x$. Then there exists a sequence $\tilde{x}_N \rightarrow x$ such that $\nabla L(\tilde{x}_N)$ exists for all $\tilde{x}_N$. The proof is the same when $x$ is the point such that $L(x)$ is differentiable. Therefore, we consider
\begin{align}
D_{-v_x}L(x) \|v_x\|_2 = \lim_{\tilde{x}_N\rightarrow x} \big \langle \nabla L(\tilde{x}_N),  -v_x  \big \rangle. \label{eq:proof22_5} 
\end{align}
When $L(x)$ is not differentiable, $v_x$ is defined as $\lim_{x_M \rightarrow x} \nabla L(x_M)$, where $\{x_M\} $ could be another sequence such that $\nabla L(x_M)$ exists for all $x_M$.
We decompose $D_{-v_x}L(x) \|v_x\|_2$ as follows
\begin{align*}
D_{-v_x}L(x) \|v_x\|_2 = & \lim_{\tilde{x}_N\rightarrow x}  \big \langle \nabla L(\tilde{x}_N),  -v_x  \big \rangle =  \lim_{\tilde{x}_N\rightarrow x} \lim_{x_M\rightarrow x}  -  \big \langle \nabla L(\tilde{x}_N),  L(x_M)  \big \rangle \\
= &  \lim_{\tilde{x}_N\rightarrow x} \lim_{x_M\rightarrow x} - \big [ 4 \dotp{h_{\tilde{x}_N,x_0}}{h_{x_M,x_0}} 
+ \dotp{\nabla L(\tilde{x}_N)-2h_{\tilde{x}_N,x_0}}{2h_{x_M,x_0}}\\
 & + \dotp{2h_{\tilde{x}_N,x_0}}{\nabla L(\tilde{x}_N)  - 2 h_{x_M,x_0}} + \dotp{\nabla L(\tilde{x}_N)-2h_{\tilde{x}_N,x_0}}{\nabla L(\tilde{x}_N)  - 2 h_{x_M,x_0}} \big ] \\
 = & -4 \|h_{x,x_0}\|_2^2 - \lim_{\tilde{x}_N\rightarrow x} \lim_{x_M\rightarrow x} \big [ \dotp{\nabla L(\tilde{x}_N)-2h_{\tilde{x}_N,x_0}}{2h_{x_M,x_0}}\\
 & + \dotp{2h_{\tilde{x}_N,x_0}}{\nabla L(\tilde{x}_N)  - 2 h_{x_M,x_0}} + \dotp{\nabla L(\tilde{x}_N)-2h_{\tilde{x}_N,x_0}}{\nabla L(\tilde{x}_N)  - 2 h_{x_M,x_0}} \big ],
\end{align*}
where we regard the last term inside the limitation as approximation error term.
It is equivalent to analyze
\begin{align}
\frac{1}{4} D_{-v_x}L(x)  \|v_x\|_2 =  & -\|h_{x,x_0}\|_2^2 - \lim_{\tilde{x}_N\rightarrow x} \lim_{x_M\rightarrow x} \bigg [ \dotp{\frac12 \nabla  L(\tilde{x}_N)-h_{\tilde{x}_N,x_0}}{h_{x_M,x_0}} \nonumber \\
 & + \dotp{h_{\tilde{x}_N,x_0}}{\frac12 \nabla L(x_M)  -  h_{x_M,x_0}} \nonumber \\
 & + \dotp{\frac12 \nabla L(\tilde{x}_N)-h_{\tilde{x}_N,x_0}}{\frac12 \nabla L(x_M)  - h_{x_M,x_0}} \bigg ]. \label{eq:proof22_6}
\end{align}
For simply notation, we let 
\begin{align*}
\overline{v}_{x,x_0} = \l(  \prod_{i=1}^n W_{i,+,x} \r)^\top \l(  \prod_{i=1}^n W_{i,+,x} \r) x -  \l(  \prod_{i=1}^n W_{i,+,x} \r)^\top \l(  \prod_{i=1}^n W_{i,+,x_0} \r) x_0.
\end{align*}
Thus we have
\begin{align*}
\dotp{\overline{v}_{x,x_0}}{z} =   \langle G(x), H_x(z) \rangle -  \langle G(x_0), H_x(z) \rangle.
\end{align*}
For the term $\dotp{\frac12 \nabla  L(\tilde{x}_N)-h_{\tilde{x}_N,x_0}}{h_{x_M,x_0}}$ in \eqref{eq:proof22_6}, we have that , setting $\lambda$ and $m$ sufficiently large as shown above, with probability at least $1-c_4 \exp(-u)$, 
\begin{align*}
&\l \langle \frac12 \nabla  L(\tilde{x}_N)-h_{\tilde{x}_N,x_0}, h_{x_M,x_0} \r\rangle  \\
&\qquad = \l \langle\frac12 \nabla  L(\tilde{x}_N)-\overline{v}_{\tilde{x}_N,x_0}, h_{x_M,x_0}\r\rangle + \dotp{\overline{v}_{\tilde{x}_N,x_0}-h_{\tilde{x}_N,x_0}}{h_{x_M,x_0}} \\
&\qquad\geq  -\varepsilon \|H_{\tilde{x}_N}(h_{x_M,x_0})\|_2 - \|\overline{v}_{\tilde{x}_N,x_0}-h_{\tilde{x}_N,x_0}\|_2  \|h_{x_M,x_0}\|_2 \\
&\qquad \geq -\varepsilon\|H_{\tilde{x}_N}(h_{x_M,x_0})\|_2  - 48 \frac{n^3 \sqrt{\varepsilon_\WDC}}{2^n} \max (\|\tilde{x}_N\|_2,  \|x_0\|_2) \|h_{x_M,x_0}\|_2 \\
&\qquad \geq -\varepsilon \l(\frac{1}{2} + \varepsilon_\WDC\r)^{n/2} \|h_{x_M,x_0}\|_2  - 48 \frac{n^3 \sqrt{\varepsilon_\WDC}}{2^n} \max (\|\tilde{x}_N\|_2,  \|x_0\|_2) \|h_{x_M,x_0}\|_2, 
\end{align*}
where the first inequality is by \eqref{eq:proof22_4} and Cauchy-Schwarz inequality,and the third inequality is by \eqref{eq:wdc_ineq_1} in Lemma \ref{lem:wdc_ineq}. The second inequality above is due to
\begin{align*}
&\|\overline{v}_{\tilde{x}_N,x_0}-h_{\tilde{x}_N,x_0}\|_2 \\
&\leq  \l\|\l(  \prod_{i=1}^n W_{i,+,\tilde{x}_N} \r)^\top \l(  \prod_{i=1}^n W_{i,+,\tilde{x}_N} \r) \tilde{x}_N  -  \frac{1}{2^n} \tilde{x}_N    \r\|_2  \\
&\quad + \l\| \l(  \prod_{i=1}^n W_{i,+,\tilde{x}_N} \r)^\top \l(  \prod_{i=1}^n W_{i,+,x_0} \r) x_0 -   \frac{1}{2^n} \l[ \l(  \prod_{i=0}^{n-1} \frac{\pi-\overline{\varrho}_i}{\pi}  \r) x_0 + \sum_{i=0}^{n-1} \frac{\sin \overline{\varrho}_i}{\pi} \l( \prod_{j=i+1}^{d-1}\frac{\pi - \overline{\varrho}_j}{\pi} \r) \frac{\|x_0\|_2}{\|\tilde{x}_N\|_2} \tilde{x}_N \r]    \r\|_2 \\
&\leq  24 \frac{n^3 \sqrt{\varepsilon_\WDC}}{2^n} \|\tilde{x}_N\|_2 + 24 \frac{n^3 \sqrt{\varepsilon_\WDC}}{2^n} \|x_0\|_2 \\
&\leq  48 \frac{n^3 \sqrt{\varepsilon_\WDC}}{2^n} \max (\|\tilde{x}_N\|_2,  \|x_0\|_2),
\end{align*}
where the second inequality is by \eqref{eq:wdc_ineq_1} in Lemma \ref{lem:wdc_ineq}.

Similarly, for the terms $\dotp{h_{\tilde{x}_N,x_0}}{\frac12 \nabla L(x_M)  -  h_{x_M,x_0}} $ in \eqref{eq:proof22_6}, we have that, setting $m$ and $\lambda$ sufficiently large as above, with probability at least $1-c_4 \exp(-u)$, the following holds:
\begin{align*}
\l\langle h_{\tilde{x}_N,x_0}, \frac12 \nabla L(x_M)  -  h_{x_M,x_0} \r\rangle \geq & -\varepsilon\l(\frac{1}{2} + \varepsilon_\WDC\r)^{n/2} \|h_{\tilde{x}_N, x_0}\|_2   - 48 \frac{n^3 \sqrt{\varepsilon_\WDC}}{2^n} \max (\|x_M\|_2,  \|x_0\|_2) \|h_{\tilde{x}_N, x_0}\|_2.
\end{align*} 
For the terms $\dotp{\frac12 \nabla L(\tilde{x}_N)-h_{\tilde{x}_N,x_0}}{\frac12 \nabla L(x_M)  - h_{x_M,x_0}}$ in \eqref{eq:proof22_6}, we have that, setting $m$ and $\lambda$ sufficiently large as above, with probability at least $1-2c_4 \exp(-u)$, the following holds:
\begin{align*}
&\l \langle \frac12 \nabla L(\tilde{x}_N)-h_{\tilde{x}_N,x_0}, \frac12 \nabla L(x_M)  - h_{x_M,x_0}\r \rangle  \\
&\qquad \geq   -\l[ \varepsilon\l(\frac{1}{2} + \varepsilon_\WDC\r)^{n/2} + 48 \frac{n^3 \sqrt{\varepsilon_\WDC}}{2^n} \max (\|x_M\|_2,  \|x_0\|_2) \r]   \\
&\qquad \quad \cdot \l[ \varepsilon\l(\frac{1}{2} + \varepsilon_\WDC\r)^{n/2} + 48 \frac{n^3 \sqrt{\varepsilon_\WDC}}{2^n} \max (\|\tilde{x}_N\|_2,  \|x_0\|_2) \r].
\end{align*}
Combining the above together, plugging in \eqref{eq:proof22_6} and taking limit on both sides, we have
\begin{align*}
-\frac{1}{4} D_{-v_x}L(x)  \|v_x\|_2 &\geq  \frac{1}{2}\| h_{x, x_0}\|_2 \l[ \| h_{x, x_0}\|_2  -2 \l(  \varepsilon\l(\frac{1}{2} + \varepsilon_\WDC\r)^{n/2}  + 48 \frac{n^3 \sqrt{\varepsilon_\WDC}}{2^n} \max (\|x\|_2,  \|x_0\|_2) \r) \r]  \\
&\quad + \frac{1}{2} \l[ \| h_{x, x_0}\|_2^2 - 2 \l( \varepsilon\l(\frac{1}{2} + \varepsilon_\WDC\r)^{n/2} + 48 \frac{n^3 \sqrt{\varepsilon_\WDC}}{2^n} \max (\|x\|_2,  \|x_0\|_2) \r)^2 \r],
\end{align*} 
with probability at least $1-4c_4 \exp(-u)$ by setting $m$ and $\lambda$ sufficiently large as above. 

\noindent \textbf{Discussion of Two Cases:}
We take our discussion from two aspects: $\|x_0\|_2 > 2^{n/2}\varepsilon^{1/2}_\WDC$ and $\|x_0\|_2 \leq 2^{n/2}\varepsilon^{1/2}_\WDC$.

\noindent \textbf{Case 1:} $\|x_0\|_2 > 2^{n/2}\varepsilon^{1/2}_\WDC$, or equivalently $\varepsilon_\WDC< 2^{-n}\|x_0\|^2_2$. 
This means $\|x\|_0$ is not close to $0$. 
If we let $\varepsilon = \varepsilon_\WDC$, $4\pi n \varepsilon_\WDC \leq 1$, then we have
\begin{align*}
&  \varepsilon \l(\frac{1}{2} + \varepsilon_\WDC\r)^{n/2}  + 48 \frac{n^3 \sqrt{\varepsilon_\WDC}}{2^n} \max (\|x\|_2,  \|x_0\|_2)  \\
&\qquad \leq \frac{\|x_0\| \sqrt{\varepsilon_\WDC}}{2^n} (1+2\varepsilon_\WDC)^{n/2} + 48 \frac{n^3 \sqrt{\varepsilon_\WDC}}{2^n} \max (\|x\|_2,  \|x_0\|_2)  \\
&\qquad\leq \frac{\|x_0\|  \sqrt{\varepsilon_\WDC} }{2^n} (1+2 n \varepsilon_\WDC ) + 48 \frac{n^3 \sqrt{\varepsilon_\WDC}}{2^n} \max (\|x\|_2,  \|x_0\|_2)  \\
&\qquad\leq \frac{3 n^3 \sqrt{\varepsilon_\WDC}}{2^n}  \max (\|x\|_2,  \|x_0\|_2)  + 48 \frac{n^3 \sqrt{\varepsilon_\WDC}}{2^n} \max (\|x\|_2,  \|x_0\|_2)  \\
&\qquad\leq 51 \frac{n^3 \sqrt{\varepsilon_\WDC}}{2^n} \max (\|x\|_2,  \|x_0\|_2), 
\end{align*}
where the second inequality is due to $(1+2\varepsilon_\WDC)^{n/2} \leq e^{n\varepsilon_\WDC} \leq 1 + 2 n\varepsilon_\WDC$ when $\varepsilon_\WDC$ is sufficiently small satisfying the conditions of Theorem \ref{thm:determine_converge}.

Recall the definition of $S_{121n^4 \sqrt{\varepsilon_\WDC}, x_0}$ in \eqref{eq:S}. 
We assume $x\neq 0$ and $x \notin S_{121n^4 \sqrt{\varepsilon_\WDC} ,x_0}$, namely $\|h_{x,x_0}\|_2 > 121n^4/2^n \sqrt{\varepsilon_\WDC} \max (\|x\|_2, \|x_0\|_2)$. By Lemma \ref{lem:converge_set}, if $x \in  \mathcal{B}^c (x_0, 616 n^3 \varepsilon^{-1/4}_\WDC\|x_0\|_2) \cap \mathcal{B}^c (-\rho_n x_0, 5500n^{14} \varepsilon^{-1/4}_\WDC \|x_0\|_2) $, it is guaranteed that $x \notin S_{121n^3 \sqrt{\varepsilon_\WDC} ,x_0}$ under the condition  that $88\pi n^6 \varepsilon_\WDC^{1/4} < 1$. Then we obtain
\begin{align*}
-\frac{1}{4} D_{-v_x}L(x)  \|v_x\|_2 \geq \frac{9}{2}\| h_{x, x_0}\|_2 \frac{n^3 \sqrt{\varepsilon_\WDC}}{2^n} \max (\|x\|_2,  \|x_0\|_2)    + \frac{9439}{2} \frac{n^6 \varepsilon_\WDC}{2^{2n}} [\max (\|x\|_2,  \|x_0\|_2)]^2>0,
\end{align*}
or equivalently,
\begin{align*}
D_{-v_x}L(x)  \|v_x\|_2 < 0,
\end{align*}
with probability at least $1-4c_4 \exp(-u)$ when we set 
\begin{align}
\lambda\geq & 4\max\{c_1(R\|a\|_{\psi_1}+\|\xi\|_{\psi_1}),1\}\log(64\max\{c_1(R\|a\|_{\psi_1} + \|\xi\|_{\psi_1}),1\}/\varepsilon_\WDC), \label{eq:proof22_lambda}\\
m\geq& c_2\lambda^2\log^2(\lambda m)(kn\log(ed) + k\log(2R) + k\log m+ u)/\varepsilon_\WDC^2.\label{eq:proof22_m}
\end{align} 
Next, we need to prove that $\forall w\neq 0$, $D_w L(0) < 0$.  We compute the directional derivative as 
\begin{align*}
\frac{1}{2} D_w L(0)\cdot \|w\|_2 =&  - \lim_{x_N\rightarrow 0}\frac{\lambda }{m} \sum_{i=1}^m y_i  \dotp{a_i}{H_{x_N}(w)} \\
= & \lim_{x_N\rightarrow 0} \dotp{G(x_0)}{H_{x_N}(w)} - \frac{\lambda }{m} \sum_{i=1}^m y_i  \dotp{a_i}{H_{x_N}(w)} - \dotp{G(x_0)}{H_{x_N}(w)} \\
\leq & \lim_{x_N\rightarrow 0} \l|\dotp{G(x_0)}{H_{x_N}(w)} - \frac{\lambda }{m} \sum_{i=1}^m y_i  \dotp{a_i}{H_{x_N}(w)} \r| - \dotp{G(x_0)}{H_{x_N}(w)} \\
\leq & \varepsilon\l(\frac{1}{2} + \varepsilon_\WDC\r)^{n/2} \|w\|_2 - \frac{1}{4\pi}\frac{1}{2^n} \|w\|_2 \|x_0\|_2 \\
\leq & \frac{1}{2^{n/2}} \varepsilon (1+2n\varepsilon_\WDC) \|w\|_2 - \frac{1}{4\pi}\frac{1}{2^n} \|w\|_2 \|x_0\|_2, 
\end{align*}
where the first inequality is due to \eqref{eq:proof22_4}, and the second inequality is due to \eqref{eq:wdc_ineq_2} in Lemma \ref{lem:wdc_ineq}.
Now we still let $\varepsilon = \varepsilon_\WDC$, then $576\pi^2 n^6 \varepsilon_\WDC \leq 1$ (which is guaranteed by the condition $88\pi n^6 \varepsilon_\WDC^{1/4} < 1$).  If $w\neq0$, setting $\lambda$ and $m$ satisfying \eqref{eq:proof22_lambda} and \eqref{eq:proof22_m}, the following holds with probability at least $1-c_4 \exp(-u)$,
\begin{align*}
\frac{1}{2} D_w L(0)\cdot \|w\|_2 \leq & \frac{1}{2^{n/2}} \varepsilon_\WDC (1+2n\varepsilon_\WDC) \|w\|_2 - \frac{1}{4\pi}\frac{1}{2^n} \|w\|_2 \|x_0\|_2 \\
\leq & \frac{1}{2^n} \sqrt{\varepsilon_\WDC} (1+2n\varepsilon_\WDC) \|w\|_2 \|x_0\|_2 - \frac{1}{4\pi}\frac{1}{2^n} \|w\|_2 \|x_0\|_2 \\
\leq & \frac{1}{2^n} 3n^3\sqrt{\varepsilon_\WDC}  \|w\|_2 \|x_0\|_2 - \frac{1}{4\pi}\frac{1}{2^n} \|w\|_2 \|x_0\|_2 \\
\leq & - \frac{1}{8\pi}\frac{1}{2^n} \|w\|_2 \|x_0\|_2  <  0,
\end{align*}
where the first inequality is due to the condition that $\varepsilon_\WDC < 2^{-n} \|x_0\|_2^2$. This implies that 
\begin{align*}
D_{w} L(0) < 0, \forall w \neq 0. 
\end{align*}
Summarizing the results in Case 1, we have that, if we let $\lambda$ and $m$ satisfying \eqref{eq:proof22_lambda} and \eqref{eq:proof22_m}, the following holds with probability at least $1- 5 c_4 \exp(-u)$,
\begin{align*}
&D_{-v_x} L(x) < 0, \forall x \notin \mathcal{B} (x_0, 616 n^3 \varepsilon^{1/4}_\WDC\|x_0\|_2) \cup \mathcal{B} (-\rho_n x_0, 5500n^{14} \varepsilon^{1/4}_\WDC \|x_0\|_2) \cup \{0\}, \\
&D_{w} L(0) < 0, \forall w \neq 0. 
\end{align*}

\noindent \textbf{Case 2:} $\|x_0\|_2 \leq 2^{n/2}\varepsilon^{1/2}_\WDC$, or equivalently $\varepsilon_\WDC \geq 2^{-n}\|x_0\|^2_2$.
This condition means $\|x_0\|$ is very small and close to $0$. Then, for any $z$, we would similarly have
\begin{align*}
-\frac14 D_{-v_x} L(x) \|v_x\|_2^2\geq & \frac{1}{2}\| h_{x, x_0}\|_2 \l[ \| h_{x, x_0}\|_2  -2 \l(  \varepsilon
\l(\frac{1}{2} + \varepsilon_\WDC\r)^{n/2}  + 48 \frac{n^3 \sqrt{\varepsilon_\WDC}}{2^n} \max (\|x\|_2,  \|x_0\|_2) \r) \r]  \\
&+ \frac{1}{2} \l[ \| h_{x, x_0}\|_2^2 - 2 \l( \varepsilon\l(\frac{1}{2} + \varepsilon_\WDC\r)^{n/2} + 48 \frac{n^3 \sqrt{\varepsilon_\WDC}}{2^n} \max (\|x\|_2,  \|x_0\|_2) \r)^2 \r].
\end{align*} 
For any non-zero $x$ satisfying $x \notin S_{121n^4 \sqrt{\varepsilon_\WDC}, x_0}$, which implies that $\|h_{x,x_0}\|_2 > 121 n^4 2^{-n} \varepsilon_\WDC \max (\|x\|_2, \|x_0\|_2)$, we have
\begin{align*}
-\frac14 D_{-v_x} L(x) \|v_x\|_2^2 \geq & \frac{1}{2}\| h_{x, x_0}\|_2 \l[ 25 \frac{n^3 \sqrt{\varepsilon_\WDC}}{2^n} \max (\|x\|_2,  \|x_0\|_2)  -2  \varepsilon\l(\frac{1}{2} + \varepsilon_\WDC\r)^{n/2}  \r]  \\
&+ \frac{1}{2} \l[ 53 \frac{n^3 \sqrt{\varepsilon_\WDC}}{2^n} \max (\|x\|_2,  \|x_0\|_2) - \sqrt{2}  \varepsilon\l(\frac{1}{2} + \varepsilon_\WDC\r)^{n/2} \r] \\
& \cdot \l[ \| h_{x, x_0}\|_2 + \sqrt{2} \l( \varepsilon\l(\frac{1}{2} + \varepsilon_\WDC\r)^{n/2} + 48 \frac{n^3 \sqrt{\varepsilon_\WDC}}{2^n} \max (\|x\|_2,  \|x_0\|_2) \r) \r].
\end{align*}
Furthermore, for any $x$ satisfying $\|x\|_2 \geq 2^{n/2} \sqrt{\varepsilon_\WDC}$, we have
\begin{align*}
\|x\|_2 \geq 2^{n/2} \sqrt{\varepsilon_\WDC} \geq \|x_0\|, \text{ namely } x \notin \mathcal{B} (0, 2^{n/2} \varepsilon^{1/2}_\WDC),
\end{align*}
which leads to
\begin{align*}
-\frac14 D_{-v_x} L(x) \|v_x\|_2^2 \geq & \frac{1}{2}\| h_{x, x_0}\|_2 \l[ 25 \frac{n^3 \sqrt{\varepsilon_\WDC}}{2^n} 2^{n/2} \sqrt{\varepsilon_\WDC}  -2  \varepsilon\l(\frac{1}{2} + \varepsilon_\WDC\r)^{n/2}  \r]  \\
& + \frac{1}{2} \l[ 53 \frac{n^3 \sqrt{\varepsilon_\WDC}}{2^n} 2^{n/2} \sqrt{\varepsilon_\WDC} - \sqrt{2}  \varepsilon\l(\frac{1}{2} + \varepsilon_\WDC\r)^{n/2} \r] \\
&   \cdot \l[ \| h_{x, x_0}\|_2 + \sqrt{2} \l( \varepsilon\l(\frac{1}{2} + \varepsilon_\WDC\r)^{n/2} + 48 \frac{n^3 \sqrt{\varepsilon_\WDC}}{2^n} \max (\|x\|_2,  \|x_0\|_2) \r) \r]\\
=&\frac{1}{2}\| h_{x, x_0}\|_2 \l[ 25 \frac{n^3 \varepsilon_\WDC}{2^{n/2}} -2  \varepsilon \l(\frac{1}{2} + \varepsilon_\WDC\r)^{n/2}  \r]  \\
&  + \frac{1}{2} \l[ 53 \frac{n^3 \varepsilon_\WDC}{2^{n/2}}   - \sqrt{2}  \varepsilon \l(\frac{1}{2} + \varepsilon_\WDC\r)^{n/2} \r] \\
&  \cdot \l[ \| h_{x, x_0}\|_2 + \sqrt{2} \l( \varepsilon\l(\frac{1}{2} + \varepsilon_\WDC\r)^{n/2} + 48 \frac{n^3 \sqrt{\varepsilon_\WDC}}{2^n} \max (\|x\|_2,  \|x_0\|_2) \r) \r].
\end{align*}
We let $\varepsilon = \varepsilon_\WDC$. Then we have $\varepsilon (1/2 + \varepsilon_\WDC)^{n/2} \leq 3n\varepsilon_\WDC 2^{-n/2} $, which consequently results in
\begin{align*}
-\frac14 D_{-v_x} L(x) \|v_x\|_2^2 > 0,
\end{align*}
or equivalently,
\begin{align*}
D_{-v_x} L(x) \|v_x\|_2^2 < 0.
\end{align*}
Note that in the above results, we also apply \eqref{eq:proof22_4} in deriving the inequalities. Therefore, summarizing the above results in Case 2, we have that, if we let $\lambda$ and $m$ satisfying \eqref{eq:proof22_lambda} and \eqref{eq:proof22_m}, the following holds with probability at least $1- 4 c_4 \exp(-u)$,
\begin{align*}
&D_{-v_x} L(x) < 0, ~~ \forall x \notin \mathcal{B} (x_0, 616 n^3 \varepsilon^{1/4}_\WDC\|x_0\|_2) \cup \mathcal{B} (-\rho_n x_0, 5500n^{14} \varepsilon^{1/4}_\WDC \|x_0\|_2) \cup \mathcal{B} (0, 2^{n/2} \varepsilon^{1/2}_\WDC),
\end{align*}
which completes the proof.
\end{proof}

\section{Proof of Theorem \ref{thm:optimum_compare}} \label{sec:proof_thm_23}

The proof of Theorem \ref{thm:optimum_compare} is mainly based on Lemmas \ref{lem:comp-expect} and \ref{lem:comp-variance} proved in the last section and two additional lemmas in the previous literature \citep{huang2018provably} given as below.

\subsection{Lemmas for Theorem \ref{thm:optimum_compare}}

\begin{lemma}[\citet{huang2018provably}] \label{lem:loss_error_bound}
Fix $0 < \psi \leq \frac{1}{4\pi}$. For any $\varphi, \zeta \in [\rho_n , 1]$, it holds that 
\begin{align*}
& \dotp{x}{h_{x, x_0}} - \frac{1}{2^{n+1}} \|x\|^2_2 \leq \frac{1}{2^{n+1}}\l(\varphi^2-2 \varphi + \frac{10 \pi^2 n }{K_0^3} \psi \r)\|x_0\|_2^2, \forall x \in \mathcal{B}( \varphi x_0,  \psi \|x_0\|_2) \\
&\dotp{z}{h_{z, x_0}} - \frac{1}{2^{n+1}} \|z\|^2_2 \geq \frac{1}{2^{n+1}}(\zeta^2-2\zeta \rho_n-10 \pi^2 n^3 \psi  )\|x_0\|_2^2, \forall z \in \mathcal{B}(-\zeta x_0,  \psi \|x_0\|_2)
\end{align*}
where $K_0 =  \min_{n\geq 2} \rho_n$, and $\rho_n$ is defined in Lemma~\ref{lem:converge_set}.
\end{lemma}

\begin{lemma}[\citet{huang2018provably}] \label{lem:rho_lower_bound}
For all $n \geq 2$, there exists a constant $K_1$ such that 
\begin{align*}
\frac{1}{K_1(n + 2)^2} \leq 1-\rho_n.
\end{align*}

\end{lemma}

\subsection{Proof Sketches of Theorem \ref{thm:optimum_compare}}

Our proof of Theorem \ref{thm:optimum_compare} is sketched as follows:
\begin{itemize}
\item We first show that for any $x$, the empirical risk $L(x)$ can be approximated as $2 \langle  h_{x, x_0},x \rangle - \|G(x)\|_2^2$ by the two critical lemmas, Lemma \ref{lem:comp-expect} and Lemma \ref{lem:comp-variance}.

\item Then we bound the approximation error  $|L(x) - 2 \langle  (h_{x, x_0},x \rangle - \|G(x)\|_2^2)|$, where $h_{x, x_0}$ is defined in \eqref{eq:h}.

\item By Lemmas \ref{lem:loss_error_bound}, \ref{lem:rho_lower_bound}, we have that if $x$ and $z$ are around $x_0$ and  $-\rho_n x_0$ respectively, by considering the approximation errors, the upper bound of $L(x)$ is smaller than the lower bound of $L(z)$, which further leads to $L(x) < L(z)$ with high probability.
\end{itemize}

\subsection{Detailed Proof of Theorem \ref{thm:optimum_compare}}

\begin{proof}[Proof of Theorem \ref{thm:optimum_compare}]
By \eqref{eq:wdc_ineq_3} in Lemma \ref{lem:wdc_ineq}, we have have 
\begin{align}
\|G(x)\|_2 \leq (1/2 + \varepsilon_\WDC)^{n/2} \|x\|_2, \label{eq:proof23_0}
\end{align}
combining which and the assumption $\|x_0\|_2 \leq R (1/2 + \varepsilon_\WDC)^{-n/2} $ in Theorem \ref{thm:optimum_compare},  we further have 
\begin{align*}
\|G(x_0)\|_2 \leq R.
\end{align*}
 By Lemma \ref{lem:comp-expect} and $\|G(x_0)\|_2 \leq R$, we set
$$\lambda\geq4\max\{c_1(R\|a\|_{\psi_1}+\|\xi\|_{\psi_1}),1\}\log(64\max\{c_1(R\|a\|_{\psi_1} + \|\xi\|_{\psi_1}),1\}/\varepsilon),$$
and  $z = x$ in Lemma \ref{lem:comp-expect} such that $H_x(x) = G(x)$, and the following holds for any $x$,
\begin{align}
\l|\lambda\expect{y_i\dotp{a_i}{G(x)}} - \dotp{G(x_0)}{G(x)} \r|\leq \frac{1}{4} \varepsilon  \|G(x)\|_2. \label{eq:proof23_1}
\end{align}
According to Lemma \ref{lem:comp-variance} and $|G(x_0)| \leq R$, we have that with probability at least $1-c_4 \exp(-u)$, for any $x$, the following holds:
\begin{align}
\l|\frac{\lambda }{m}\sum_{i=1}^m y_i \dotp{a_i}{G(x)} - \lambda \expect{y_i\dotp{a_i}{G(x)}}\r| \leq \frac{\varepsilon}{8}\|G(x)\|_2, \label{eq:proof23_2}
\end{align}
with sample complexity being
\begin{align*}
m\geq c_2\lambda^2\log^2(\lambda m)(kn\log(ed) + k\log(2R) + k\log m+ u)/\varepsilon^2,
\end{align*}
where we also set $z = x$ in Lemma \ref{lem:comp-variance} such that $H_x(x) = G(x)$.

Combining \eqref{eq:proof23_1} and \eqref{eq:proof23_2}, we will have that with probability at least $1-c_4 \exp(-u)$, for any $x$, setting$$\lambda\geq4\max\{c_1(R\|a\|_{\psi_1}+\|\xi\|_{\psi_1}),1\}\log(64\max\{c_1(R\|a\|_{\psi_1} + \|\xi\|_{\psi_1}),1\}/\varepsilon),$$ and 
\begin{align*}
m\geq c_2\lambda^2\log^2(\lambda m)(kn\log(ed) + k\log(2R) + k\log m+ u)/\varepsilon^2,
\end{align*}
the following holds:
\begin{align}
\l|\frac{\lambda }{m}\sum_{i=1}^m y_i \dotp{a_i}{G(x)} - \dotp{G(x_0)}{G(x)} \r| \leq \varepsilon\|G(x)\|_2. \label{eq:proof23_3}
\end{align}

\noindent \textbf{Bounding the error term:}
We next  bound the term $|L(x) + \|G(x)\|_2^2 - 2\dotp{h_{x,x_0}}{x} |$ as follows.
With $\lambda$, $m$ satisfying the same conditions above, then with probability at least $1-c_4 \exp (-u)$ , the following holds: 
\begin{align*}
&\l|L(x) + \|G(x)\|_2^2 - 2\dotp{h_{x,x_0}}{x}  \r| \\
& \qquad =  \l| 2 \|G(x)\|_2^2  - \frac{2 \lambda }{m}\sum_{i=1}^m y_i \dotp{a_i}{G(x)} - 2 \dotp{h_{x, x_0}}{x}  \r| \\
& \qquad =  \l| 2\dotp{G(x_0)}{G(x)} - \frac{2 \lambda }{m}\sum_{i=1}^m y_i \dotp{a_i}{G(x)}    + 2 \|G(x)\|_2^2 - 2\dotp{G(x_0)}{G(x)}  - 2 \dotp{h_{x, x_0}}{x} \r|.
\end{align*}
Furthermore, we bound the above terms as follows
\begin{align*}
& \l| 2\dotp{G(x_0)}{G(x)} - \frac{2 \lambda }{m}\sum_{i=1}^m y_i \dotp{a_i}{G(x)} \r|   +  \l| 2 \|G(x)\|_2^2 - 2\dotp{G(x_0)}{G(x)}  - 2 \dotp{h_{x, x_0}}{x} \r| \\
& \qquad \leq 2 \varepsilon\|G(x)\|_2 +  48 \frac{n^3 \sqrt{\varepsilon_\WDC}}{2^n} \|x\|^2_2 + 48 \frac{n^3 \sqrt{\varepsilon_\WDC}}{2^n}  \|x_0\|_2 \|x\|_2\\
& \qquad \leq  2 \varepsilon \l(\frac{1}{2} + \varepsilon_\WDC \r)^{n/2} \|x\|_2  + 48 \frac{n^3 \sqrt{\varepsilon_\WDC}}{2^n} \|x\|^2_2 + 48 \frac{n^3 \sqrt{\varepsilon_\WDC}}{2^n}  \|x_0\|_2 \|x\|_2\\
& \qquad \leq 2 \varepsilon \frac{1+2n\varepsilon_\WDC}{2^{n/2}}\|x\|_2  + 48 \frac{n^3 \sqrt{\varepsilon_\WDC}}{2^n} \|x\|^2_2 + 48 \frac{n^3 \sqrt{\varepsilon_\WDC}}{2^n}  \|x_0\|_2 \|x\|_2,
\end{align*}
where the second inequality is due to \eqref{eq:proof23_3} and \eqref{eq:wdc_ineq_1} in Lemma \ref{lem:wdc_ineq}, the third inequality is due to \eqref{eq:proof23_0}, and the last inequality is due to $(1+2\varepsilon_\WDC)^{n/2 } \leq e^{n\varepsilon_\WDC} \leq 1+ 2n\varepsilon_\WDC$ if $\varepsilon_\WDC$ is sufficiently small satisfying the condition of Theorem \ref{thm:optimum_compare}. This result implies  
\begin{align*}
\l|L(x) + \|G(x)\|_2^2 - 2\dotp{h_{x,x_0}}{x}  \r| \leq 
 2 \varepsilon \frac{1+2n\varepsilon_\WDC}{2^{n/2}}\|x\|_2  + 48 \frac{n^3 \sqrt{\varepsilon_\WDC}}{2^n} \|x\|^2_2 + 48 \frac{n^3 \sqrt{\varepsilon_\WDC}}{2^n}  \|x_0\|_2 \|x\|_2.
\end{align*}
Since we only consider the case that $\varepsilon_\WDC \leq 2^{-n} \|x_0\|_2^2$. Letting $\varepsilon = \varepsilon_\WDC$, we have
\begin{align}
&\l|L(x) + \|G(x)\|_2^2 - 2\dotp{h_{x,x_0}}{x}  \r| \nonumber  \\
& \qquad \leq  2 \varepsilon_\WDC \frac{1+2n\varepsilon_\WDC}{2^{n/2}} \|x\|_2  + 48 \frac{n^3 \sqrt{\varepsilon_\WDC}}{2^n} \|x\|^2_2 + 48 \frac{n^3 \sqrt{\varepsilon_\WDC}}{2^n}  \|x_0\|_2 \|x\|_2 \nonumber \\
& \qquad \leq 2 \sqrt{\varepsilon_\WDC} \frac{1+2n\varepsilon_\WDC}{2^n}\|x_0\|_2 \|x\|_2  + 48 \frac{n^3 \sqrt{\varepsilon_\WDC}}{2^n} \|x\|^2_2 + 48 \frac{n^3 \sqrt{\varepsilon_\WDC}}{2^n}  \|x_0\|_2 \|x\|_2 \label{eq:proof23_4}
\end{align}
with probability $1-c_4 \exp(-u)$ if we set
\begin{align}
&\lambda\geq4\max\{c_1(R\|a\|_{\psi_1}+\|\xi\|_{\psi_1}),1\}\log(64\max\{c_1(R\|a\|_{\psi_1} + \|\xi\|_{\psi_1}),1\}/\varepsilon_\WDC),  \label{eq:proof23_lambda}\\
&m \geq c_2  \lambda^2   \log^2(\lambda m) (kn\log(ed) + k\log(2R) + k\log m+ u)  /\varepsilon_\WDC^2.\label{eq:proof23_m}
\end{align}

\noindent \textbf{Upper bound of $L(x)$:}
For any $x \in \mathcal{B}( \varphi x_0,  \psi \|x_0\|_2)$ with $0< \psi \leq 1/(4\pi)$ and any $\varphi  \in [\rho_n, 1]$, we have
\begin{align*}
L(x) =& 2 \dotp{x}{h_{x,x_0}} -\|G(x)\|_2^2 + \l( L(x) - 2 \dotp{x}{h_{x,x_0}} + \|G(x)\|_2^2 \r)\\
= & 2 \dotp{x}{h_{x,x_0}}  - \frac{1}{2^{n}} \|x\|_2^2 - \l(\|G(x)\|_2^2- \frac{1}{2^{n}} \|x\|_2^2 \r) + \l(L(x) - 2 \dotp{x}{h_{x,x_0}} + \|G(x)\|_2^2 \r)\\
\leq & 2 \dotp{x}{h_{x,x_0}}  - \frac{1}{2^n} \|x\|_2^2 + \l| \|G(x)\|_2^2- \frac{1}{2^{n}} \|x\|_2^2 \r| + \l| L(x) - 2 \dotp{x}{h_{x,x_0}} + \|G(x)\|_2^2 \r|\\
\leq & \frac{1}{2^n}(\varphi^2-2 \varphi +  \frac{10 \pi^2 n }{K_0^3} \psi )\|x_0\|_2^2 + \l| \|G(x)\|_2^2- \frac{1}{2^{n}} \|x\|_2^2 \r| +  \l| L(x) - 2 \dotp{x}{h_{x,x_0}} + \|G(x)\|_2^2 \r|,
\end{align*} 
where the last inequality is due to Lemma \ref{lem:loss_error_bound} and \eqref{eq:proof23_0}.
In addition, we can also obtain
\begin{align*}
 &\l| L(x) - 2 \dotp{x}{h_{x,x_0}} + \|G(x)\|_2^2 \r|  \\
& \qquad \leq \sqrt{\varepsilon_\WDC} \frac{1+2n\varepsilon_\WDC}{2^n}\|x_0\|_2 \|x\|_2  + 24 \frac{n^3 \sqrt{\varepsilon_\WDC}}{2^n} \|x\|^2_2 + 24 \frac{n^3 \sqrt{\varepsilon_\WDC}}{2^n}  \|x_0\|_2 \|x\|_2 \\
& \qquad \leq 2\sqrt{\varepsilon_\WDC} \frac{1+2n\varepsilon_\WDC}{2^n}(\varphi + \psi)\|x_0\|^2_2 + 48 \frac{n^3 \sqrt{\varepsilon_\WDC}}{2^n} (\varphi + \psi)^2\|x_0\|^2_2 + 48 \frac{n^3 \sqrt{\varepsilon_\WDC}}{2^n}  (\varphi + \psi)\|x_0\|^2_2 \\
& \qquad \leq 122 \frac{n^3 \sqrt{\varepsilon_\WDC}}{2^n}\|x_0\|^2_2 ,
\end{align*}
and 
\begin{align*}
\l| \|G(x)\|_2^2- \frac{1}{2^{n}} \|x\|_2^2 \r| \leq 24 \frac{n^3 \sqrt{\varepsilon_\WDC}}{2^n} \|x\|_2^2  \leq 30 \frac{n^3 \sqrt{\varepsilon_\WDC}}{2^n} \|x_0\|_2,
\end{align*}
due to $\|x\|_2 \leq (\varphi + \psi)\|x_0\|_2$ when $x \in \mathcal{B}( \varphi x_0,  \psi \|x_0\|_2)$ and $\varphi + \psi \leq 1 + 1/(4\pi) < 1.1$ , and \eqref{eq:wdc_ineq_1} in Lemma \ref{lem:wdc_ineq}.

Combining the above results and letting $\lambda$ and $m$ satisfy \eqref{eq:proof23_lambda} and \eqref{eq:proof23_m}, the following holds with probability at least $1-c_4 \exp(-u)$,
\begin{align*}
L(x) \leq \frac{1}{2^n}(\varphi^2-2 \varphi + \frac{10 \pi^2 n }{K_0^3} \psi + 152 n^3 \sqrt{\varepsilon_\WDC})\|x_0\|_2^2, 
\end{align*}
for any $x \in \mathcal{B}( \varphi x_0,  \psi \|x_0\|_2)$.

\noindent \textbf{Lower bound of $L(z)$:} Next, we should the lower bound of $L(z)$ when $z$ is around $-\rho_n x_0$.
Consider the situation for any $z \in \mathcal{B}( -\zeta x_0,  \psi \|x_0\|_2)$ with $0< \psi \leq 1/(4\pi)$ and any $\zeta  \in [\rho_n, 1]$. We can obtain
\begin{align*}
L(z) =& 2 \dotp{z}{h_{z,x_0}} -\|G(z)\|_2^2 + \l( L(z) - 2 \dotp{z}{h_{z,x_0}} + \|G(z)\|_2^2 \r)\\
\geq & 2 \dotp{z}{h_{z,x_0}} -\|G(z)\|_2^2 - \l| L(z) - 2 \dotp{z}{h_{z,x_0}} + \|G(z)\|_2^2 \r|\\
= & 2 \dotp{z}{h_{z,x_0}}  - \frac{1}{2^{n}} \|z\|_2^2 - \l(\|G(z)\|_2^2- \frac{1}{2^{n}} \|z\|_2^2 \r) - \l| L(z) - 2 \dotp{z}{h_{z,x_0}} + \|G(z)\|_2^2 \r|\\
\geq & 2 \dotp{z}{h_{z,x_0}}  - \frac{1}{2^n} \|z\|_2^2 - \l| \|G(z)\|_2^2- \frac{1}{2^{n}} \|z\|_2^2 \r| - \l| L(z) - 2 \dotp{z}{h_{z,x_0}} + \|G(z)\|_2^2 \r|\\
\geq & \frac{1}{2^n}(\zeta^2-2\zeta \rho_n-10 \pi^2 n^3 \psi  )\|x_0\|_2^2 - \l| \|G(x)\|_2^2- \frac{1}{2^{n}} \|x\|_2^2 \r| -  \l| L(x) - 2 \dotp{x}{h_{x,x_0}} + \|G(x)\|_2^2 \r|,
\end{align*} 
where the last inequality is due to Lemma \ref{lem:loss_error_bound}.
Furthermore, similar to the previous steps in the upper bound of $L(x)$, we have
\begin{align*}
 \l| L(z) - 2 \dotp{z}{h_{z,x_0}} + \|G(z)\|_2^2 \r|  \leq  122 \frac{n^3 \sqrt{\varepsilon_\WDC}}{2^n}\|x_0\|^2_2 ,
\end{align*}
and
\begin{align*}
\l| \|G(z)\|_2^2- \frac{1}{2^{n}} \|z\|_2^2 \r| \leq 30 \frac{n^3 \sqrt{\varepsilon_\WDC}}{2^n} \|x_0\|_2,
\end{align*}
due to $\|z\|_2 \leq (\zeta + \psi)\|x_0\|_2$ when $z \in \mathcal{B}( -\zeta x_0,  \psi \|x_0\|_2)$ and $\zeta + \psi \leq 1 + 1/(4\pi) < 1.1$.

Combining the above results, letting $\lambda$ and $m$ satisfy \eqref{eq:proof23_lambda} and \eqref{eq:proof23_m}, the following holds with probability at least $1-c_4 \exp(-u)$,
\begin{align*}
L(z) \geq \frac{1}{2^n}(\zeta^2-2\zeta \rho_n-10 \pi^2 n^3 \psi - 152 n^3 \sqrt{\varepsilon_\WDC})\|x_0\|_2^2, 
\end{align*}
for any $z \in \mathcal{B}( -\zeta x_0,  \psi \|x_0\|_2)$.

\noindent \textbf{Proving  $L(x) < L(z)$:}
In order to have $L(x) < L(z)$, it is enough to ensure that
\begin{align*}
& \min_{\zeta \in [\rho_n, 1]} \frac{1}{2^n}(\zeta {}^2-2\zeta \rho_n-10 \pi^2 n^3 \psi - 152 n^3 \sqrt{\varepsilon_\WDC} )\|x_0\|_2^2 \\
& \qquad > \max_{\varphi \in [\rho_n, 1]} \frac{1}{2^n}(\varphi^2-2 \varphi + \frac{10 \pi^2 n }{K_0^3} \psi + 152 n^3 \sqrt{\varepsilon_\WDC})\|x_0\|_2^2. 
\end{align*}
The minimizer for the left side of the above inequality is $\varphi = \rho_n$ while the maximizer for the right side is also $\zeta = \rho_n$. Then, to achieve the above inequality, we plug in the minimizer and maximizer for both sides and obtain
\begin{align*}
\rho_n^2-2\rho^2_n-10 \pi^2 n^3 \psi -152 n^3 \sqrt{\varepsilon_\WDC} >  \rho_n^2-2 \rho_n + \frac{10 \pi^2 n }{K_0^3} \psi + 152 n^3 \sqrt{\varepsilon_\WDC}.
\end{align*}
Rearranging the terms, we would obtain
\begin{align*}
2\rho_n - 2\rho_n^2 > \l( 10 \pi^2 n^3 +\frac{10 \pi^2 n }{K_0^3} \r) \psi + 304 n^3 \sqrt{\varepsilon_\WDC}. 
\end{align*}
To make the above inequality hold for all $\rho_n$, by computing the minimal value of the left-hand side according to Lemma \ref{lem:rho_lower_bound}, we require $\varepsilon_\WDC$ to satisfy
\begin{align*}
\frac{2K_0}{K_1 (n+2)^2} > \l( 10 \pi^2 n^3 +\frac{10 \pi^2 n }{K_0^3} \r) \psi + 304 n^3 \sqrt{\varepsilon_\WDC}.
\end{align*}
Due to $n+2 \leq 2n$ and $n \leq  n^3$ (since we assume $n > 1$), it suffices to ensure
\begin{align*}
\frac{K_0}{4 K_1 n^2} > \l( 10 \pi^2 n^3 +\frac{10	 \pi^2 n^3 }{K_0^3} \r) \psi + 304 n^3 \sqrt{\varepsilon_\WDC}, 
\end{align*}
which can be, therefore, guaranteed by the condition
\begin{align*}
35\sqrt{K_1/K_0} n^3 \varepsilon^{1/4}_\WDC \leq 1 \text{ and } \psi \leq \frac{K_0}{50\pi^2 K_1 (1+1/K_0^3) }n^{-5}.
\end{align*}
Thus, under the condition of Theorem \ref{thm:optimum_compare}, for any $x \in \mathcal{B}( \varphi x_0,  \psi \|x_0\|_2)$ and $z \in \mathcal{B}( -\zeta x_0,  \psi \|x_0\|_2)$, letting $\lambda$ and $m$ satisfy \eqref{eq:proof23_lambda} and \eqref{eq:proof23_m},  with probability at least $1-2c_4 \exp(-u)$, we have$$L(x) < L(z).$$
Note that the radius $\psi$ satisfies $\psi < K_0 := \rho_n$, which means there are no overlap between $\mathcal{B}( \varphi x_0,  \psi \|x_0\|_2)$ and $\mathcal{B}( -\zeta x_0,  \psi \|x_0\|_2)$. This is because by Lemma \ref{lem:rho_lower_bound}, we know that  $1/K_1 \leq (n+2)^2 \leq 4n^2$. Therefore, $\psi \leq K_0 n^{-5} / (50\pi^2 K_1 (1+1/K_0^3)) \leq K_0 n^{-3} < K_0$ when $n \geq 2$.
This completes the proof.
\end{proof}

\end{document}